\documentclass[12pt,reqno]{amsart}
\usepackage{amsfonts}
\usepackage{amsthm}
\usepackage{amsmath}
\usepackage{amssymb}
\usepackage{amscd}
\usepackage[latin2]{inputenc}
\usepackage{t1enc}
\usepackage[mathscr]{eucal}
\usepackage{indentfirst}
\usepackage{graphicx}
\usepackage{graphics}
\usepackage{pict2e}
\usepackage{epic}
\numberwithin{equation}{section}
\usepackage[margin=2.9cm]{geometry}
\usepackage{epstopdf} 

\theoremstyle{plain}
\newtheorem{Th}{Theorem}[section]
\newtheorem{Lemma}[Th]{Lemma}
\newtheorem{Cor}[Th]{Corollary}
\newtheorem{Prop}[Th]{Proposition}

\theoremstyle{definition}

\newtheorem{?}[Th]{Problem}

\newcommand{\C}{\mathbb C}
\newcommand{\R}{\mathbb R}

\newcommand{\Z}{\mathbb Z}
\newcommand{\N}{\mathbb N}
\newcommand{\E}{\mathbb{E}}
\newcommand{\Pro}{\mathbb{P}}
\newcommand{\eps}{\varepsilon}
\newcommand{\abs}[1]{\left\lvert#1\right\rvert}
\newcommand{\norme}[1]{\left \| #1 \right \|}

\begin{document}

\title[Liouville metric of star-scale invariant fields: tails and Weyl scaling]{Liouville Metric of star-scale invariant fields: \\ tails and Weyl scaling}

\author[Julien Dub\'{e}dat]{Julien Dub\'{e}dat}
\thanks{Partially supported by NSF DMS 1308476 and NSF DMS 1512853 (JD)}
\author[Hugo Falconet]{Hugo Falconet}

\address{Department of Mathematics \\ Columbia University}

\address{ 2990 Broadway \\  New York \\ 10027 \\ USA}



\begin{abstract}We study the Liouville metric associated to an approximation of a log-correlated Gaussian field with short range correlation. We show that below a parameter $\gamma_c >0$, the left-right length of rectangles for the Riemannian metric $e^{\gamma \phi_{0,n}} ds^2$ with various aspect ratio is concentrated with quasi-lognormal tails, that the renormalized metric is tight when $\gamma <  \min ( \gamma_c, 0.4)$ and that subsequential limits are consistent with the Weyl scaling.\end{abstract}

\maketitle

\tableofcontents

\section{Introduction}  

Gaussian multiplicative chaos (GMC) is the study of random measures of the form $e^{\gamma \phi } \sigma (dx)$ where $\gamma \in (0, \sqrt{2d})$ is a parameter, $\phi$ is a $\log$-correlated Gaussian field on a domain $D$ in $\R^d$ and $\sigma(dx)$ is an independent measure on $D$. Since the field $\phi$ just exists in a Schwartz sense, a regularization procedure and a renormalization have to be done to show the existence of $e^{\gamma \phi} \sigma(dx)$. One classical regularization of the field is the martingale approximation done by Kahane \cite{Ka}, another one is by taking a convolution with a mollifier, done by Robert and Vargas \cite{RoVa}. Shamov \cite{Shamov} then proved that in a rather large setting of regularization, the convergence holds in probability, the limit does not dependent on the regularization procedure and is measurable with respect to the field (see also Berestycki \cite{BerGMC} for an elementary approach). A particular case of the theory, initiated  by Duplantier and Sheffield \cite{DuSh}, is when $d =2$ (which we will always assume from now on) and when the field is the Gaussian free field: this random measure is called Liouville Quantum Gravity (LQG).

\smallskip

One may try to follow the same lines to define the metric whose Riemannian metric tensor is $e^{\gamma \phi} ds^2$:  approximate $\phi$ by a smooth field to obtain a well-defined random Riemannian metric, show that the appropriately renormalized metric converges to a limiting metric which is independent of the limiting procedure and which is measurable with respect to the field. This problem seems to be so far more involved than the measure one where more tools are currently available. In a series of recent papers \cite{MSmet1, MSmet2, MSmet3, MSQLE}, Miller and Sheffield considered the case $\gamma = \sqrt{8/3}$, $d = 2$ and $\phi$ is a Gaussian free field. In particular, they made sense of the limiting object directly in the continuum and established some connections with the Brownian map, universal scaling limit of a large class of random planar maps (see Le Gall \cite{LeGall1, LeGall2} and Miermont \cite{Mie}).

\smallskip

In a discrete setting, Ding and Dunlap \cite{DiDu} studied the first passage percolation associated to the discrete Gaussian free field in the bulk (see \cite{ADH} for an overview on first passage percolation). They showed that the renormalized metric is tight, when $\gamma$ is small enough. A major part of their work was to obtain Russo-Seymour-Welsh (RSW) estimates of the length of left-right crossing of rectangles with various aspect ratio and their approach strongly relies on Tassion's method \cite{Ta}. We mention here that Ding et al. \cite{DiDu, DiGo, DiGw, DiZeZh2, DiZeZh1, DiZh1, DiZh}  studied related topics.

\smallskip

Recently, Ding and Gwynne \cite{DiGw} discussed the fractal dimension of LQG. In their paper, the Liouville first passage percolation is described as follows. Let $\phi$ be a Gaussian free field on a domain $D \subset \R^2$ and fix $\xi > 0$. Denote by $\phi_{\delta}(x)$ the circle average of $\phi$ over $\partial B(x,\delta)$ and consider the distance  $D_{\phi,\mathrm{LFPP}}^{\xi,\delta}$, defined for $x,y \in D$ by $D_{\phi,\mathrm{LFPP}}^{\xi,\delta}(x,y) := \inf \int_0^1 e^{\xi \phi_{\delta}( \pi(t))} \abs{ \pi'(t) }  dt$, where the infimum is taken over all piecewise continuously differentiable paths $\pi : [0,1] \to D$ such that $\pi(0) = x$ and $\pi(1) = y$. They explained that the parameter $\xi$ should be taken as $\frac{\gamma}{d_{\gamma}}$, if $d_{\gamma}$ is the Hausdorff dimension of the $\gamma$-LQG metric, obtained by scaling limits of graph distance on random planar maps, see Section 2.3 in \cite{DiGw} for a discussion.

\smallskip

In this article, the field $\phi_{0,\infty}$ is a log-correlated field with short-range correlations and is approximated by a martingale $\phi_{0,n}$ where each $\phi_{0,n}$ is a smooth field. More precisely, we consider a $\star$-scale invariant field whose covariance kernel is translation invariant and is given by $C_{0,\infty}(x) = \int_{1}^{\infty} \frac{c(u x)}{u}du$, where $c = k * k$, for a nonnegative,  compactly supported and radially symmetric bump function $k$. We decompose the field $\phi_{0,\infty}$ in a sum of self-similar fields i.e. $\phi_{0,\infty} = \sum_{n \geq 0} \phi_n$, where the $\phi_n$'s are smooth independent Gaussian fields, such that $\phi_0$ has a finite range of dependence and $\left( \phi_n(x) \right)_{x \in \R^2}$ has the law of $\left( \phi_0(x 2^n) \right)_{x \in \R^2}$. We then denote by $\phi_{0,n}$ the truncated summation i.e. $\phi_{0,n} = \sum_{0 \leq k \leq n} \phi_k$. This gives rise to a well-defined random Riemannian metric $e^{\gamma \phi_{0,n}} ds^2$, restricted for technical convenience to $[0,1]^2$, which is the main object studied in this paper. Let us point out that the parameter $\xi$ in \cite{DiGw} corresponds to the parameter $\frac{\gamma}{2}$ here, since the length element is given by $e^{\frac{\gamma}{2} \phi_{0,n}} ds$. 

\smallskip

In the recent preprint \cite{DecStarScale}, the authors proved that any log-correlated field $\phi$ whose covariance kernel  is given by $C(x,y) =  - \log \abs{x-y} + g(x,y)$, assuming some regularity on $g$, can be decomposed as $\phi = \phi_{\star} + \psi$ where $\phi_{\star}$ is a $\star$-scale invariant Gaussian field and  $\psi$ is a Gaussian field with H\"older regularity. A similar decomposition where the fields are independent can be obtained modulo a weaker property on $\phi_{\star}$. Using this decomposition, they generalize some results present in the literature only for $\star$-scale invariant fields. Let us also mention that $\star$-scale invariant log-correlated fields are natural since they appear in the following characterization (see \cite{Allez}): if $M$  is a random measure on $\R^d$ such that $\E (M([0,1]^d)^{1+\delta}) < \infty$ for $\delta >0$  and satisfying the cascading rule, for every $\eps \in (0,1)$:
\begin{equation}
\label{eq:DecompoMeasure}
(M(A))_{A \in \mathcal{B}(\R^d)} \overset{(d)}{=} \left(\int_A e^{\omega_{\eps}(x)} M_{\eps}(dx) \right)_{A \in \mathcal{B}(\R^d)},
\end{equation}
where $(M_{\eps}(\eps A))_{A \in \mathcal{B}(\R^d)} \overset{(d)}{=}  \eps^d (M(A))_{A \in \mathcal{B}(\R^d)}$ and where $\omega_{\eps}$ is a stationary Gaussian field, independent of $M_{\eps}$, with continuous sample paths, continuous and differentiable covariance kernel on $\R^d \setminus \lbrace 0 \rbrace$, then, up to some additional technical assumptions, $M$ is the product of a nonnegative random variable $X \in L^{1+\delta}$ and an independent Gaussian multiplicative chaos $e^{\phi} dx$ i.e. $\forall A \in \mathcal{B}(\R^d), ~ M(A) = X \int_{A} e^{\phi(x) - \frac{1}{2} \E (\phi(x)^2)} dx$. Moreover, the covariance kernel of $\phi$ is given by $C(x) = \int_1^{\infty} \frac{c(u x)}{u}du$ for some continuous covariance function $c$ such that $c(0) \leq \frac{2d}{1+\delta}$ and notice that we have $C(x) \underset{x \to 0}{\sim} - c(0) \log \norme{x}$. Again, one can try to follow the same lines for the metric instead of the measure to construct and characterize metrics on $\R^2$ satisfying a property analogous to \eqref{eq:DecompoMeasure} involving the Weyl scaling (see Section \ref{sec:WeylScaling}).

\smallskip

In our approach, we introduce a parameter $\gamma_c > 0$  associated to some observable of the metric and we study the phase where $\gamma < \gamma_c$. More precisely, if $L_{1,1}^{(n)}$ denotes the left-right length of the square $[0,1]^2$ for the random Riemannian metric $e^{\gamma \phi_{0,n}} ds^2$ and $\mu_n$ is its median, we then define $\gamma_c := \inf \lbrace \gamma :  ( \log L_{1,1}^{(n)} - \log \mu_n )  \mathrm{ ~ is ~ not ~ tight }\rbrace$. We expect that the set of $\gamma$ such that $(\log L_{1,1}^{(n)} - \log \mu_n)_{n \geq 0}$ is tight is $(0,\gamma_c)$ . We prove that as soon as $\gamma < \gamma_c$, we have the following concentration result: for $s$ large, uniformly in $n$,
\begin{align*}
c e^{-Cs^2} \leq \Pro \left( \log L_{1,1}^{(n)} - \log \mu_n \leq -s \right) \leq C e^{-cs^2}, \\
c e^{-Cs^2} \leq \Pro \left( \log L_{1,1}^{(n)} - \log \mu_n \geq s \right) \leq C e^{-c\frac{s^2}{\log s}}.
\end{align*}
When $\gamma < \min (\gamma_c, 0.4)$, we obtain the tightness of the metric spaces $([0,1]^2,d_{0,n})_{n \geq 0}$, where $d_{0,n}$ is the geodesic distance associated to the Riemannian metric tensor $e^{\gamma \phi_{0,n}} ds^2$, renormalized by $\mu_n$. The main difference with the proof of Ding and Dunlap  is that the RSW estimates do not rely on the method developped by Tassion \cite{Ta} but follow from an approximate conformal invariance of $\phi_{0,n}$, obtained through a white noise coupling.

\smallskip

We also investigate the Weyl scaling: if $d_{0,\infty}$ is a metric obtained through a subsequential limit associated to the field $\phi_{0,\infty}$ and $f$ is in the Schwartz class, then we prove that the metric associated to the field $\phi_{0,\infty}+f$ is $e^{ \frac{\gamma}{2} f} \cdot d_{0,\infty}$, that the couplings $(\phi_{0,\infty} + f, e^{\frac{\gamma}{2} f} \cdot d_{0,\infty})$ and $(\phi_{0,\infty}, d_{0,\infty})$ are mutually absolutely continuous with respect to each other and that their Radon-Nikod\'ym derivative is given by the one of the first marginal. Notice that if the metric $d_{0,\infty}$ is a measurable function of the field $\phi_{0,\infty}$, this property is expected. Here, this property tells us that the metric is not independent of the field $\phi_{0,\infty}$ and is in particular non-deterministic. In fact, this property is fundamental in the work of Shamov \cite{Shamov} on Gaussian multiplicative chaos, where the metric is replaced by the measure. It is used to prove that subsequential limits are measurable with respect to the field, which then implies its uniqueness and that the convergence in law holds in probability.

\smallskip

Shamov \cite{Shamov} takes the following definition of GMC. If $\phi$ is  a Gaussian field on a domain $D$ and $M$ is a random measure on $D$, measurable with respect to $\phi$ and hence denoted by $M(\phi,dx)$, which satisfies, for $f$ in the Cameron-Martin space of $\phi$, almost surely,
\begin{equation}
\label{eq:ShamovGMC}
M(\phi + f,dx) = e^{f(x)} M(\phi,dx),
\end{equation}
then $M$ is called a Gaussian multiplicative chaos. Furthermore, $M$ is said to to be subcritical if $\E M$ is a $\sigma$-finite measure. Note that the left-hand side is well-defined since $M$ is $\phi$ measurable. It is easy to check that  the condition \eqref{eq:ShamovGMC} implies uniqueness among $\phi$-measurable subcritical random measures and we insist that the measurability of $M$ with respect to $\phi$ is built in the definition. A natural question is thus the following: replace the measure $M$ by the metric $d_{0,\infty}$, assume in a similar way the measurability with respect to $\phi$ and suppose that in \eqref{eq:ShamovGMC}, the operation is the Weyl scaling defined in Section \ref{sec:WeylScaling}, then is there uniqueness?

\smallskip

The article is organized as follows. In Section 2, we introduce the fields $\phi_{0,n}$ as well as the definitions and notations that will be used throughout the subsequent sections. Section 3 contains our main theorems. In Section 4, we derive the approximate conformal invariance of $\phi_{0,n}$ together with the RSW estimates. Section 5 is concerned with lognormal tail estimates for crossing lengths, upper and lower bounds. Under the assumption $\gamma < \min (\gamma_c, 0.4)$, we derive the tightness of the metric in Section 6. The Weyl scaling is discussed in Section 7. Section 8 is concerned with $\gamma_c > 0$. Lastly, in Section 9 we prove some independence of $\gamma_c$ with respect to the bump function $k$ used to define $\phi_{0,n}$. The appendix gathers estimates for the supremum of the field $\phi_{0,n}$ as well as an estimate for a summation which appears when deriving diameter estimates.

\section*{Acknowledgments}

We would like to thank an anonymous referee for many useful comments.

\section{Definitions}

\subsection{Log-correlated Gaussian fields with short-range correlations}

\label{Sec:DefField}

A  \textit{white noise} on $\R^d$ is a random Schwartz distribution such that for every test function $f$, $\langle \zeta, f \rangle$ is a centered Gaussian variable with variance $\norme{f}_{L^{2}(\R^d)}^2$. If $(\Omega,\mathcal{F},\Pro)$ denotes a probability space on which it is defined, we have a natural isometric embedding $ L^2(\R^d)  \hookrightarrow L^2(\Omega, \mathcal{F},\Pro)$. By extension, for $f \in L^2(\R^d)$, the pairing $\langle \zeta, f \rangle$ is also a centered Gaussian variable with variance $\norme{f}_{L^2(\R^d)}^2$. 

\smallskip

Let $k$ be a smooth, radially symmetric and nonnegative bump function supported in $B(0,r_0) \subset \R^2$ and normalized in $L^2(\R^2)$ ($\int_{\R^2} k^2 dx = 1$), where $r_0$ is a fixed small positive real number. If $\zeta$ denotes a standard white noise on $\R^2$, then the convolution $k * \zeta$ is a smooth Gaussian field with covariance kernel $c := k * k$ whose compact support is included in $B(0,2r_0)$. This can be taken as a starting point to define more general Gaussian fields. Let $\xi (dx, dt)$ be a white noise on $\R^2 \times [0,\infty)$. Then one can define a distributional Gaussian field on $\R^2$ by setting
$$
\phi_{0,\infty}(x) := \int_{\R^2} \int_0^1 k \left( \frac{y-x}{t} \right) t^{-3/2} \xi (dy, dt) 
$$
with covariance kernel given by 
\begin{align*}
\E \left( \phi_{0,\infty}(x) \phi_{0,\infty}(x') \right) & = \int_{\R^2} \int_0^1 k \left( \frac{x-y}{t} \right) k\left(\frac{y-x'}{t} \right)  t^{-3} dy dt = \int_0^1 k * k \left( \frac{x-x'}{t} \right) \frac{dt}{t} \\
& = \int_0^1 c \left( \frac{x-x'}{t} \right) \frac{dt}{t}.
\end{align*}
Remark that for $x \neq x'$, the integrand vanishes near $0$ since $c$ has compact support, and that if $ \abs{x-x'}  > 2r_0$, $\E (\phi_{0,\infty}(x) \phi_{0,\infty}(x'))  = 0$.  Denote $C(r) := \int_0^1 c(r/t) \frac{dt}{t}$. Then 
$$
C'(r) = \int_0^1 c'(r/t) \frac{dt}{t^2} = \int_0^{\infty} c'(r/t) \frac{dt}{t^2} - \int_1^{\infty} c'(r/t) \frac{dt}{t^2} = \frac{\alpha}{r} + f(r)
$$
where $\alpha = \int_0^{\infty} c'(t^{-1}) \frac{dt}{t^2} = - c(0)$ and $f$ is a smooth function. Consequently, 
$$
C(r) = \alpha \log r + F(r)
$$
where $F$ is smooth. By normalizing $k$ in $L^2(\R^2)$, we have $c(0) = k * k (0) = \int_{\R^2} k^2 dx = 1$ and 
$$
C(r) = - \log r + F(r).
$$

\subsection{Decomposition of $\phi_{0,\infty}$ in a sum of self-similar fields }
One can decompose $\phi_{0,\infty}$ as a sum of independent self-similar fields. Indeed, for $0 \leq m \leq n$, set 
\begin{equation}
\label{Rep}
\phi_{m,n}(x) := \int_{\R^2} \int_{2^{-n-1}}^{2^{-m}} k \left( \frac{y-x}{t} \right) t^{-3/2} \xi (dy,dt)
\end{equation}
as well as $\phi_n := \phi_{n,n}$ so that $\phi_{0,n} = \sum_{0 \leq k \leq n} \phi_k$ and $\phi_{0,\infty} = \sum_{n \geq 0} \phi_n$ where the $\phi_n$'s are independent. Notice also that for $1 \leq m \leq n$, $\phi_{0,n} = \phi_{0,m-1} + \phi_{m,n}$. The covariance kernel of $\phi_n$ is 
$$
\E \left( \phi_n(x) \phi_n(x') \right) = \int_{2^{-n-1}}^{2^{-n}} c\left( \frac{x-x'}{t} \right) \frac{dt}{t} =: C_n(\norme{x-x'})
$$
so that $C_n(r) = C_0(r2^n)$. We will also denote by $C_{0,n}$ the covariance kernel of $\phi_{0,n}$. 
The following properties are clear from the construction.
\begin{Prop} For every $n \geq 0$, 
\begin{enumerate}
\item
$\phi_n$ is smooth,

\item
the law of $\phi_n$ is invariant under Euclidean isometries,

\item
$\phi_n$ has finite range dependence with range of dependence $2^{-n} \cdot 2 r_0$,

\item
and $\left( \phi_n(x) \right)_{x \in \R^2}$ has the law of $\left( \phi_0(x 2^n) \right)_{x \in \R^2}$ (scaling invariance).

\item
The $\phi_n$'s are independent Gaussian fields.
\end{enumerate}
\end{Prop}
Let us precise that one can see that $\phi_n$ is smooth from the representation \eqref{Rep} since $k$ has compact support and $\xi$ is a distribution (in the sense of Schwartz). This is a deterministic statement.

We will use repeatedly these properties throughout the paper in particular the independence and scaling ones. Furthermore, one can decompose the field at scale $n$ in spatial blocks. Specifically, we denote by $\mathcal{P}_n$ the set of dyadic blocks at scale $n$, viz. $$
\mathcal{P}_n := \left\lbrace 2^{-n}  \left( [i,i+1] \times [j,j+1] \right) ~ : ~ i,j \in \Z^2 \right\rbrace.
$$
For $P \in \mathcal{P}_n$ we set  $$
\phi_{n,P}(x) := \int_P \int_{2^{-n-1}}^{2^{-n}} k \left( \frac{y-x}{t} \right) t^{-3/2} \xi (dy,dt).
$$
The following properties are immediate.
\begin{Prop}   ~
\begin{enumerate}
\item
The $\phi_{n,P}$'s are independent Gaussian fields.
\item
For every $n \geq 0$ and $P \in \mathcal{P}_n$, $\phi_{n,P}$ is smooth and compactly supported in $P + B(0,2^{-n} \cdot 2 r_0)$. 
\item
If $P \in \mathcal{P}_n$, $Q \in \mathcal{P}_m$ and $l : P \to Q$ is an affine bijection, then $\phi_{m,Q} \circ l $ has the same law as $\phi_{n,P}$. 
\end{enumerate}
\end{Prop}
\noindent Finally, we have the decomposition 
$$
\phi_{0,\infty} = \sum_{n \geq 0} \sum_{P \in \mathcal{P}_n} \phi_{n,P}
$$
in which all the summands are independent smooth Gaussian fields, all identically distributed up to composition by an affine map and $\phi_{n,P}$ is supported in a neighborhood of $P$. In the following sections, we will work with the smooth fields $\phi_{0,n}$, approximations of the field $\phi_{0,\infty}$, and we denote by $\mathcal{F}_{0,n}$ the $\sigma$-algebra generated by the $\phi_k$'s for $0 \leq k \leq n$. 

\subsection{Rectangle lengths and definition of $\gamma_c$}

\label{DefLengths}

For $a,b > 0$ and $0 \leq m \leq n$, we denote by $L_{a,b}^{(m,n)}$ the left-right length of the rectangle $[0,a] \times [0,b]$ for the Riemannian metric $e^{\gamma \phi_{m,n}} ds^2$, where the metric tensor is restricted to $[0,a] \times [0,b]$. When $m =0$ we simply write $L_{a,b}^{(n)}$.  To avoid confusion, let us point out that this is not the Riemannian metric on the full space restricted to the rectangle. In particular, all admissible paths are included in $[0,a] \times [0,b]$. It is clear that the spaces $([0,1]^2, e^{\gamma \phi_{0,n}} ds^2)$ and $([0,1]^2,ds^2)$ are bi-Lipschitz. Consequently, $([0,1]^2, e^{\gamma \phi_{0,n}} ds^2)$ is a complete metric space and it has the same topology as the unit square with the Euclidean metric. We will denote by $\pi_{m,n}$ a minimizing path associated to $L_{a,b}^{(m,n)}$ and it will be clear depending on the context which $a,b$ are involved. Notice that such a path exists by the Hopf-Rinow theorem and a compactness argument. We will say that a rectangle $R$  is visited by a path $\pi$ if $\pi \cap R \neq \emptyset$ and crossed by $\pi$ if a subpath of $\pi$ connects two opposite sides of $R$ by staying in $R$. 

\medskip

We recall the \textit{positive association} property and refer the reader to \cite{FKG} for a proof.
\begin{Th}
\label{PosAssoc}
If $f$ and $g$ are increasing functions of a continuous  Gaussian field $\phi$ with pointwise nonnegative covariance, depending only on a finite-dimensional marginal of $\phi$, then $\E\left( f(\phi) g(\phi) \right)  \geq \E \left( f(\phi) \right) \E \left(g(\phi) \right)$.
\end{Th}
We will use this inequality several times in situations where the field considered is $\phi_{0,n}$ (since $k \geq 0$) and the functions $f$ and $g$ are lengths associated to different rectangles, without being restricted to a finite-dimensional marginal of $\phi_{0,n}$. If $R$ is a rectangle, denote by $L^{(n)}(R,k)$ the left-right distance of $R$ for the field $\phi_{0,n}^k$, piecewise constant on each dyadic block of size $2^{-k}$ where it is equal to the value of $\phi_{0,n}$ at the center of this block. We also denote by $L^{(n)}(R)$ the left-right distance of $R$ for the field $\phi_{0,n}$. We have the following comparison,
$$
e^{- O(2^{-k}) \sup_{P \in \mathcal{P}_k, P \subset R} \norme{\nabla \phi_{0,n}}_{P}} L^{(n)}(R) \leq  L^{(n)}(R,k) \leq L^{(n)}(R) e^{O(2^{-k}) \sup_{P \in \mathcal{P}_k,  P \subset R} \norme{\nabla \phi_{0,n}}_{P}}
$$
which gives a.s. $\lim_{k \to \infty} L^{(n)}(R,k) = L^{(n)}(R)$. 

If $R_1, \dots, R_p$ denote $p \geq 2$ fixed rectangles, by Portmanteau theorem and since $(L^{(n)}(R_1), \dots, L^{(n)}(R_p) )$ has a positive density with respect to the Lebesgue measure on $(0,\infty)^d$ (by the argument used in the proof of Proposition \ref{LBounds}), if $l > 0$, we have, using Theorem \ref{PosAssoc},
\begin{align*}
\Pro \left( L^{(n)}(R_1) > l, \dots, L^{(n)}(R_p) > l \right) & = \lim_{k \to \infty} \Pro \left( L^{(n)}(R_1,k) > l, \dots, L^{(n)}(R_p,k) > l \right) \\
&  \geq \lim_{k \to \infty } \Pro \left( L^{(n)}(R_1,k) > l \right) \dots \Pro \left( L^{(n)}(R_p,k) > l \right)  \\
& =  \Pro \left( L^{(n)}(R_1) > l \right) \dots \Pro \left( L^{(n)}(R_p) > l \right).
\end{align*}
Furthermore, if $F,G : (0,\infty)^{[0,1]^2} \to (0,\infty)$ are increasing functions such that
\begin{enumerate}
\item
a.s. $\lim_{k \to \infty} F(\phi_{0,n}^k) = F(\phi_{0,n})$ and $\lim_{k \to \infty} G(\phi_{0,n}^k) = G(\phi_{0,n})$ 

\item
$\E \left( F(\sup_{[0,1]^2} \phi_{0,n}) G(\inf_{[0,1]^2} \phi_{0,n})^{-1} \right) < \infty$, $\E \left( F(\sup_{[0,1]^2} \phi_{0,n}) \right) < \infty$\\ and $\E \left( G(\inf_{[0,1]^2} \phi_{0,n})^{-1} \right) < \infty$,
\end{enumerate}
then, by dominated convergence theorem and the negative association we have
\begin{multline*}
\E  ( \frac{F(\phi_{0,n})}{G(\phi_{0,n})} ) = \lim_{k \to \infty } \E ( \frac{F(\phi_{0,n}^k)}{G(\phi_{0,n}^k)} )  \leq \lim_{k \to \infty } \E ( F(\phi_{0,n}^k) ) \E ( \frac{1}{G(\phi_{0,n}^k)} ) = \E ( F(\phi_{0,n}) ) \E ( \frac{1}{G(\phi_{0,n})} ).
\end{multline*}

\medskip

We introduce the notations $l_{a,b}^{(n)}(p) := \inf \lbrace l \geq 0 ~ | ~ \mathbb{P} ( L_{a,b}^{(n)} \leq l ) > p \rbrace$ for the $p$-th quantile associated to $L_{a,b}^{(n)}$ and $\bar{l}_{a,b}^{(n)}(p) := l_{a,b}^{(n)}(1-p)$.  Since we will use repetitively $l_{1,3}^{(n)}(\eps)$ and $\bar{l}_{3,1}^{(n)}(\eps)$ for a small fixed $\eps$, we introduce the notation $l_n$ for the first one and $\bar{l}_n$ for the second one. Also, we will be interested by the ratio between these quantiles hence we introduce the notation $\delta_n := \max_{0 \leq k \leq n} l_{k}^{-1} \bar{l}_k $ for $n \geq 0$. Finally, we introduce $\mu_n$ for the median of $L_{1,1}^{(n)}$ (note that $L_{1,1}^{(n)}$ has a positive density on $(0,\infty)$ with respect to the Lebesgue measure by the argument used in the proof of Proposition \ref{LBounds}). We then define the critical parameter $\gamma_c$ as
$$
\gamma_c := \inf \left\lbrace \gamma :  \left( \log L_{1,1}^{(n)} - \log \mu_n \right)  \mathrm{ ~ is ~ not ~ tight } \right\rbrace
$$
and we call \emph{subcriticality} the regime $\gamma < \gamma_c$. Note that anytime we use the assumption $\gamma < \gamma_c$, we use only the tightness of $ \log L_{1,1}^{(n)} - \log \mu_n$. However, we expect that the set of $\gamma$ such that $(\log L_{1,1}^{(n)} - \log \mu_n)_{n \geq 0}$ is tight is the interval $(0,\gamma_c)$.

\subsection{Compact metric spaces: uniform and Gromov-Hausdorff topologies}

We recall first the notion of uniform convergence. A sequence $(d_n)_{n \geq 0}$ of real-valued functions on $[0,1]^2 \times [0,1]^2$ \emph{converges uniformly} to a function $d$ if 
$$
\underset{ x,x' \in [0,1]^2}{\sup} \abs{d_n (x,x') - d(x,x')} \underset{n\to \infty}{\longrightarrow} 0.
$$
If $d_n$ are moreover distances on $[0,1]^2$, then $d$ is a priori only a pseudo-distance i.e. $d(x,y) = 0$ with $x \neq y$ may occur.

\medskip

Moreover, we recall the definition of the Hausdorff distance. If $K_1$, $K_2$ are two compact subsets of a metric space $(E, d)$, the Hausdorff distance $d_H$ between $K_1$ and $K_2$ is defined by
$$
d_H(K_1, K_2) := \inf \left\lbrace \eps > 0 : K_1 \subset U_\eps(K_2) ~ \mathrm{ and } ~ K_2 \subset U_{\eps}(K_1) \right\rbrace
$$
where for $i = 1,2$, $U_{\eps}(K_i) := \lbrace x \in E : d(x, K_i) <\eps\rbrace $ is the $\eps$-enlargement of $K_i$.

\medskip

We recall now the definition of the Gromov-Hausdorff distance. Let $(E_1,d_1)$ and $(E_2,d_2)$ be two compact metric spaces. The Gromov-Hausdorff distance $d_{GH}$ between $E_1$ and $E_2$ is defined as
$$
d_{GH}(E_1, E_2) := \inf \left\lbrace d_H (\phi_1(E_1), \phi_2(E_2)) \right\rbrace
$$
where the infimum is over all isometric embeddings $\phi_1 : E_1 \to E$ and $\phi_2 : E_2 \to E$ of $E_1$ and $E_2$ into the same metric space $(E, d)$. Here, $d_H$ is the Hausdorff distance associated to the space $(E,d)$. Denote by $\mathbb{M}$ the set of all isometry classes of compact metric spaces (see \cite{Gromov} Section 3.11). The Gromov-Hausdorff distance $d_{GH}$ is a metric on $\mathbb{M}$ and $(\mathbb{M}, d_{GH})$ is a Polish space. We refer the reader to the textbook \cite{BBI}, Section 7 for more details on these topologies. 

\medskip

In our framework, we introduce the sequence of compact metric spaces $( M_n )_{n \geq 0}$ where  $M_n := ([0,1]^2, d_{0,n})$ and where $d_{0,n}$ is the geodesic distance induced by the Riemannian metric tensor $\mu_n^{-2} e^{\gamma \phi_{0,n}} ds^2$ restricted to $[0,1]^2$ and we aim to study the convergence in law of $M_n$ to a random metric space $M_{\infty}$ with respect to the Gromov-Hausdorff topology.

\subsection{Notation}
\label{sec:Notation}

We will denote by $c$ and $C$ constants whether they should be thought as small or large. They may vary from line to line and depend on the parameters (e.g. the bump function $k$) or geometry when these are fixed. At the only place of the paper when we take $\gamma$ small, but fixed, $\gamma$ is taken small compared to a constant which does not depend on $\gamma$ (as soon as we assume that $\gamma$ is less than an absolute constant, upper bounds like $e^{\gamma \sqrt{k}}$ may be replaced by $e^{C \sqrt{k}}$).

\smallskip

If $F : E \to \C$ is a complex-valued function, we denote by $\norme{F}_{\infty} := \sup_{x \in E} \abs{F(x)}$ and by $\norme{F}_{C^{\alpha}(E)} := \norme{F}_{\infty}  + \sup_{x \neq y \in E} \frac{\abs{F(x)-(y)}}{\abs{x-y}^{\alpha}} $. For $d \geq 1$, $\mathcal{S}(\R^d)$ denotes the space of Schwartz functions and $\mathcal{S}'(\R^d)$ denotes the space of tempered distributions. Our convention for the Fourier transform of a function $ \varphi \in \mathcal{S}(\R^d)$ is $\hat{\varphi}(\xi) := \int_{\R^d} \varphi(x) e^{- i x \cdot \xi} dx$. If $x$ is a real number we will denote by $x_{+}$ the maximum of $x$ and $0$. For two real numbers $a$ and $b$ we denote by $a \vee b := \max (a,b) $ as well as $a \wedge b := \min (a,b)$. Finally, if $X$ is a random variable, $\mathcal{L}(X)$ denotes its law and for $x \in \R$ we set $F_X(x) := \Pro(X \leq x)$.

\section{Statement of main results}

Our first main result concerns the relation between lengths of rectangles with different aspect ratio. We want to compare the tails of $L_{a,b}^{(n)}$ for various choices of $(a, b)$. Notice that if $a' \leq a$, $b' \leq b$, a.s.
$$
L_{a',b}^{(n)} \leq L_{a,b}^{(n)}  \leq L_{a,b'}^{(n)}.
$$
In particular, this gives $l_{a',b}^{(n)}(p) \leq l_{a,b}^{(n)}(p) \leq l_{a,b'}^{(n)}(p)$ for every $p$ in $(0,1)$. The following Russo-Seymour-Welsh estimates give upper bounds of left-right crossing lengths of long rectangles in terms of left-right crossing lengths of short rectangles.
\begin{Th}
\label{th:RSW}
\label{RSWQuantiles}
If $[A,B] \subset (0,\infty)$ there exists $C > 0$  such that for every $(a,b), (a',b') \in [A,B]$ with $a/b < 1 < a'/b'$ and for every $n \geq 0$, $\eps < 1/2$  we have 
\begin{gather}
\label{eq:RSW1}
l_{a',b'}^{(n)} (\eps/C) \leq l_{a,b}^{(n)}(\eps)C e^{C \sqrt{\abs{\log \eps /C}}},  \\
\label{eq:RSW2}
 \bar{l}_{a',b'}^{(n)} (3 \eps^{1/C}) \leq \bar{l}_{a,b}^{(n)}(\eps) C e^{C \sqrt{\abs{\log \eps/C}}}.
\end{gather}
\end{Th}

In the article \cite{DiDu}, Ding and Dunlap obtained a difficult result (see Theorem 5.1 in \cite{DiDu}), inspired by \cite{Ta}. Their result applies to a rather general setting whereas here we rely on some approximate conformal invariance of the field considered. However the result in \cite{DiDu} holds for $\gamma$ small and this is a comparison for low quantiles only. Here we obtain comparisons for low, as well as high, quantiles, and there is no assumption on $\gamma$. Furthermore, the RSW estimates obtained here are also quantitative: this is instrumental for instance in the proof of left tail estimates.
\begin{Th}
\label{th:Tails}
If $\gamma < \gamma_c$, the left-right length for various aspect ratio renormalized by $\mu_n$ is tight and its tails are quasi-lognormal i.e. if $[A,B] \subset (0,\infty)$ there exist constants $c > 0$, $C > 0$  such that for every $(a,b) \in [A,B]$, $n \geq 0$, $s > 1$: 
\begin{gather}
\label{QuasiGaussianUp}
\Pro \left( L_{a,b}^{(n)} \geq \mu_n e^{ s \sqrt{\log s}} \right) \leq  C e^{- c s^2},\\
\label{QuasiGaussianLow}
\Pro \left( L_{a,b}^{(n)} \leq \mu_n e^{- s} \right) \leq  C e^{- c s^2}.
\end{gather}
\end{Th}
These estimates are fundamental ingredients to get:
\begin{Th}
\label{th:Metric} Assume that $\gamma < \min(\gamma_c, 0.4)$. Then:
\begin{enumerate} 
\item 
The sequence of compact metric spaces $\left( M_n \right)_{n \geq 0}$ where  $M_n := \left([0,1]^2, d_{0,n} \right)$ and where $d_{0,n}$ is the geodesic distance induced by the Riemannian metric $\mu_n^{-2} e^{\gamma \phi_{0,n}} ds^2$ is tight with respect to the uniform and Gromov-Hausdorff topologies. 

\item
If $(n_k)$ is a subsequence along which $(d_{n_k})_{k \geq 0}$ converges in law to some $d_{0,\infty}$, then for $f \in \mathcal{S}(\R^2)$, $(d_{n_k},e^{\frac{\gamma}{2} f} \cdot d_{n_k})_{k \geq 0}$ converges in law to $(d_{0,\infty}, e^{\frac{\gamma}{2} f} \cdot d_{0,\infty})$ (see Section \ref{sec:WeylScaling} for a definition of the Weyl scaling).

\item
Moreover, $(\phi_{0,\infty} + f, e^{\frac{\gamma}{2} f} \cdot d_{0,\infty})$ is absolutely continuous with respect to $(\phi_{0,\infty}, d_{0,\infty})$ and the associated Radon-Nikod\'ym derivative is the one associated to the first marginal i.e. $\frac{d \mathcal{L}(\phi_{0,\infty} + f)}{d \mathcal{L}(\phi_{0,\infty})}$.
\end{enumerate}
\end{Th}
We will also check that $\gamma_c > 0$ which is the content of:
\begin{Th}
\label{th:BoundGamma}
For every choice of bump function $k$, $\gamma_c(k) > 0$.  
\end{Th}
The general proof scheme of this result is similar to the one in \cite{DiDu}. The key tool is the Efron-Stein inequality, which was introduced by Kesten in the context of {\em Euclidean} first passage percolation. It was first used by Ding and Dunlap in a multiscale analysis to study Liouville first passage percolation metrics. Let us mention a few key differences in the implementation of that concentration argument.

In \cite{DiDu}, the authors use the Efron-Stein inequality to give an upper bound of $\mathrm{Var} (L_{1,1}^{(n)})$, in order to control inductively the coefficient of variation of $L_{1,1}^{(n)}$, defined as
$$
CV^2(L_{1,1}^{(n)}) := \frac{\mathrm{Var} (L_{1,1}^{(n)})}{\E(L_{1,1}^{(n)})^2}.
$$
Here, since we expect that the logarithm of the normalized left-right distance is tight, we apply the Efron-Stein inequality to $\log L_{1,1}^{(n)}$ (the underlying product structure is provided naturally by the white noise representation of the field). We recall the notation for quantiles $\bar{l}_{1,1}^{(k)}(p)$, $l_{1,1}^{(k)}(p)$, defined such that $\Pro ( L_{1,1}^{(k)} \geq \bar{l}_{1,1}^{(k)}(p) ) = p$ and $\Pro ( L_{1,1}^{(k)} \leq l_{1,1}^{(k)}(p)) = p$, and set
$$
\delta_n(p) := \max_{k \leq n} \frac{\bar{l}_{1,1}^{(k)}}{l_{1,1}^{(k)}}(p)
$$
which is the quantity we want to bound inductively; $p$ is chosen small enough but fixed so that our tail estimates hold. The starting point of the induction is the inequality
$$
\frac{\bar{l}_{1,1}^{(n)}}{l_{1,1}^{(n)}}(p) \leq e^{C_p \sqrt{\mathrm{Var} \log L_{1,1}^{(n)}}}.
$$
Here the multiscale analysis, relying in particular on tail estimates (let us point out that instead of quasi-Gaussian bounds, super-exponential bounds would suffice) shows that, for $\gamma$ small (but which can be quantified) for some $c_{\gamma} < 1$, we have
$$
\mathrm{Var} \log L_{1,1}^{(n)} \leq  \gamma^2 \left(C + C \delta_{n-1}(p)^2 \sum_{k=1}^{\infty} c_{\gamma}^k \right)
$$
The absence of an explicit bound on $\gamma_c$ comes from the fact that we take $\gamma$ small enough in this inequality to bound inductively $\delta_{n}(p)$.

Finally, we will work out some independence of the parameter $\gamma_c$ with respect to the choice of the bump function which is the content of
\begin{Th}
\label{th:Indep}
If $k_1$ and $k_2$ are two bump functions such that $\hat{k}_1(\xi) =e^{- a \norme{\xi}^\alpha (1 + o(1) )}$ and $\hat{k}_2(\xi) = e^{- b \norme{\xi}^\alpha (1 + o(1) )}$, as $\xi$ goes to infinity, for some $\alpha \in (0,1)$ and $a,b >0$, then $\gamma_c(k_1) = \gamma_c(k_2)$.  
\end{Th}

\section{Russo-Seymour-Welsh estimates: proof of Theorem \ref{th:RSW}}

In this section we prove that our approximation $\phi_{0,n}$ of $\phi_{0,\infty}$ is approximately conformally invariant.  We will then investigate its consequences on the length of left-right crossings:  the RSW estimates, Theorem \ref{th:RSW}, which is a key result of our analysis. Let us already point out that these RSW estimates eventually lead, as a first corollary, to a lognormal decay of the left tail (inequality \eqref{QuasiGaussianLow}, without assuming $\gamma < \gamma_c$ but with a small quantile instead of the median).

\subsection{Approximate conformal invariance of $\phi_{0,n}$}

Let $F : U \rightarrow V$ be a conformal map between two Jordan domains. We wish to compare the laws of $\phi_{0,n}$ and $\phi_{0,n} \circ F$ in $U$ and look for a uniform estimate in $n$. For this we go back to the defining white noises. We write, for $\xi$ and $\tilde{\xi}$ two standard white noises
\begin{align*}
\phi_{0,n}(x) := \int_{\R^2} \int_{2^{-n-1}}^1 k\left( \frac{x-y}{t} \right) t^{-3/2} \xi(dy,dt), \\
\tilde{\phi}_{0,n}(x) := \int_{\R^2} \int_{2^{-n-1}}^1 k\left( \frac{x-y}{t} \right) t^{-3/2} \tilde{\xi}(dy,dt),
\end{align*}
and we want to couple $\phi_{0,n}$ and $\tilde{\phi}_{0,n} \circ F$ , in particular for the high-frequency modes. 
We couple the defining white noises $\xi, \tilde{\xi}$ in the following way: if $y' \in V$, $y \in U$, $y' = F(y)$, $t' = t|F'(y)|$, then
$$
\tilde{\xi}(dy',dt') = \abs{F'(y)}^{3/2} \xi(dy,dt)
$$
i.e. for a test function $\phi$ compactly supported in $V \times (0, \infty)$,
$$
\int \phi(y',t') \tilde{\xi}(dy',dt') = \int \phi(F(y),t |F'(y)|) \abs{F'(y)}^{3/2} \xi(dy,dt)
$$
and both sides have variance $\norme{\phi}_{L^2}^2$. The rest of the white noises are chosen to be independent, i.e. $\xi_{|U^c \times (0, \infty)}$, $\xi_{|U \times (0, \infty)}$ and $\xi_{|\tilde{V}^c \times (0, \infty)}$ are jointly independent. Assuming $|F'| \geq 1$ on $U$, since 
$$
\int_V \int_{2^{-n-1}}^{1} k \left( \frac{F(x)-y}{t} \right) t^{-3/2} \tilde{\xi}(dy,dt) = \int_U \int_{2^{-n-1}\abs{F'(y)}^{-1}}^{\abs{F'(y)}^{-1}} k \left( \frac{F(x)-F(y)}{t \abs{F'(y)}} \right) t^{-3/2} \xi(dy,dt),
$$
we can decompose $\phi_{0,n} (x) - \tilde{\phi} _{0,n} (F(x)) = \delta \phi_1(x) + \delta \phi_2(x) + \delta \phi_3(x)$ where
\begin{align*}
\delta \phi_1(x) = & \int_{U^c} \int_{2^{-n-1}}^1 k \left( \frac{x-y}{t} \right) t^{-3/2} \xi(dy,dt) - \int_{V^c} \int_{2^{-n-1}}^1 k \left( \frac{F(x)-y}{t} \right) t^{-3/2} \tilde{\xi}(dy,dt) \\
&  + \int_{U} \int_{\abs{F'(y)}^{-1}}^1 k \left( \frac{x-y}{t} \right) t^{-3/2} \xi(dy,dt),    \\
\delta \phi_2(x) = & \int_{U} \int_{2^{-n-1}}^{\abs{F'(y)}^{-1}} \left( k \left( \frac{x-y}{t} \right) - k \left( \frac{F(x)-F(y)}{t \abs{F'(y)}} \right) \right) t^{-3/2} \xi(dy,dt),   \\
\delta \phi_3(x) = & - \int_{U} \int_{2^{-n-1} \abs{F'(y)}^{-1}}^{2^{-n-1}} k \left( \frac{F(x)-F(y)}{t \abs{F'(y)}} \right) t^{-3/2} \xi(dy,dt).
\end{align*}
Remark also that $\delta \phi_3$ is independent of $\phi_{0,n}$, $\delta \phi_1$, and $\delta \phi_2$. We will estimate these three terms separately on a convex compact subset $K$ of an open convex set $U$ under the assumption that $\norme{F'}_{U,\infty} < \infty$ and $\norme{F''}_{U,\infty}  < \infty$ and $\abs{F'} \geq 1$ on $U$. 

\begin{Lemma}
\label{SmoothVar}
$\delta \phi_1$ restricted to $K$ is a smooth field; more precisely there exists $C > 0$ such that for every $n \geq 0$ 
$$
\mathbb{E}\left( \norme{\delta \phi_{1} }_{C^1(K)} \right) \leq C.
$$
\end{Lemma}

\begin{proof}
If $x \in K$, since $k$ has compact support included in $B(0,r_0)$ we can write
$$
\int_{U^c} \int_{2^{-n-1}}^1 k \left( \frac{x-y}{t} \right) t^{-3/2} \xi(dy,dt) = \int_{U^c} \int_{(1 \wedge d(K,U^c)/r_0) \vee 2^{-n-1}}^1 k \left( \frac{x-y}{t} \right) t^{-3/2} \xi(dy,dt).
$$
The idea is the same for the second term. For the third term, $\abs{F'(y)} \leq \norme{F'}_{U,\infty}$ hence 
$$
\int_{U} \int_{\abs{F'(y)}^{-1}}^1 k \left( \frac{x-y}{t} \right) t^{-3/2} \xi(dy,dt)   = \int_{U} \int_{ \norme{F'}_{U,\infty}^{-1}}^1 1_{1 \leq t \abs{F'(y)}} k \left( \frac{x-y}{t} \right) t^{-3/2} \xi(dy,dt)
$$
which concludes the proof: the smoothness follows standard results of distribution in the sense of Schwartz.
\end{proof}

\begin{Lemma}
\label{CI}
There exists $C > 0$ such that for every $n \geq 0$ and every $x,x' \in K$,
$$
\mathbb{E}\left( (\delta \phi_2 (x) - \delta \phi_2 (x') )^2 \right)  \leq C \abs{x-x'}.
$$ 
We also have $\E \left( \delta \phi_2 (x)^2 \right) \leq C$ uniformly in $x \in K$ and $n\geq 0$.
\end{Lemma}

\begin{proof}
Since $k$ is rotationally invariant and has compact support, we will see that
\begin{equation}
\label{O(t)}
k \left( \frac{x-y}{t} \right) = k \left( \frac{F(x)-F(y)}{t \abs{F'(y)}} \right) + O(t).
\end{equation}
First, $k$ having a compact support included in $B(0,r_0)$ gives 
\begin{align*}
& k \left(\frac{x-y}{t} \right)  = k \left(\frac{x-y}{t} \right)   1_{\frac{|x-y|}{t}  \leq r_0} =  k \left(\frac{x-y}{t} \right)   1_{ t \geq \frac{|x-y|}{r_0} } \\ 
& k \left(\frac{F(x)-F(y)}{t\abs{F'(y)}} \right)  = k \left(\frac{F(x)-F(y)}{t\abs{F'(y)}} \right)   1_{\frac{|F(x)-F(y)|}{t \abs{F'(y)}}  \leq r_0} =  k \left(\frac{F(x)-F(y)}{t\abs{F'(y)}} \right)   1_{ t \geq \frac{|F(x)-F(y)|}{r_0 |F'(y)|} }
\end{align*}
Since $\abs{F'} \geq 1$ on $U$ and $\norme{F'}_{U,\infty} < \infty$
$$
\frac{|F(x)-F(y)|}{|F'(y)|} \geq \frac{ |F^{-1}(F(x)) - F^{-1}(F(y)) | }{\norme{F'}_{U,\infty} \norme{(F^{-1})'}_{V,\infty}} = \frac{|x-y|}{C}
$$
hence we can directly replace the term $1_{ t \geq \frac{|F(x)-F(y)|}{r_0 |F'(y)|} }$ by $1_{ t \geq \frac{|x-y|}{C r_0} }$. By Taylor's inequality, $|F(x)-F(y)-F'(y)(x-y)| \leq \frac{1}{2} |x-y|^2 \norme{F''}_{U,\infty}$ thus 
$$
\abs{\frac{F(x)-F(y)}{t \abs{F'(y)}} - \frac{x-y}{t} \frac{F'(y)}{\abs{F'(y)}}} \leq \frac{|x-y|^2}{2t} \frac{\norme{F''}_{U,\infty}}{\abs{F'(y)}}.
$$
The consequences of the compact support seen above together with the rotational invariance of $k$ give 
\begin{align*}
\abs{ k \left( \frac{F(x)-F(y)}{t \abs{F'(y)}}  \right)  - k  \left(\frac{x-y}{t} \right)  } & \leq \norme{ \nabla k }_{\infty}   \frac{\norme{F''}_{U,\infty}}{\abs{F'(y)}} \frac{|x-y|^2}{2t} 1_{t \geq \frac{|x-y|}{Cr_0}} \\
&  \leq \frac{1}{2} \norme{ \nabla k }_{\infty} \norme{F''}_{U,\infty}(C r_0)^2 t
\end{align*}
which gives \eqref{O(t)}. Finally, we obtain the following bound
\begin{align*}
& \left( k\left( \frac{x-y}{t} \right) - k \left( \frac{F(x)-F(y)}{t \abs{F'(y)}} \right) \right) -  \left(k\left( \frac{x'-y}{t} \right) - k \left( \frac{F(x')-F(y)}{t \abs{F'(y)}} \right) \right) \\
& = \left(k\left( \frac{x-y}{t} \right) - k\left( \frac{x'-y}{t} \right) \right) - \left( k \left( \frac{F(x)-F(y)}{t \abs{F'(y)}} \right) - k \left( \frac{F(x')-F(y)}{t \abs{F'(y)}} \right) \right) \\ 
& = O\left( t \wedge \frac{|x'-x|}{t} \right)
\end{align*}
where in the last equation we both used equation \eqref{O(t)} and the inequalities, for $x,x' \in K$ and $y \in U$: 
$$
\abs{ k\left( \frac{x-y}{t} \right) - k\left( \frac{x'-y}{t} \right) } \leq \norme{ \nabla k }_{\infty} \frac{|x-x'|}{t} 
$$
and
\begin{align*}
 \abs{ k \left( \frac{F(x)-F(y)}{t \abs{F'(y)}} \right) - k \left( \frac{F(x')-F(y)}{t \abs{F'(y)}} \right) } & \leq   \norme{ \nabla k }_{\infty} \frac{| F(x) - F(x') |}{t \abs{F'(y)}} \\
&  \leq \norme{ \nabla k }_{\infty}  \norme{F'}_{K,\infty} \frac{ | x - x' |}{t}.
\end{align*}
It follows that 
\begin{align*}
\mathbb{E}\left( (\delta \phi_2 (x) - \delta \phi_2 (x') )^2 \right) 
& =   \int_{U}  \int_{2^{-n-1}}^{|F'(y)|^{-1}} \left( \left( k\left( \frac{x-y}{t} \right) - k \left( \frac{F(x)-F(y)}{t \abs{F'(y)}} \right) \right)  \right. \\
& ~ \left. ~  -  \left(k\left( \frac{x'-y}{t} \right) - k \left( \frac{F(x')-F(y)}{t \abs{F'(y)}} \right) \right) \right)^2 t^{-3} dt dy \\
& \leq \int_0^1  O \left( t \wedge \frac{\abs{x-x'}}{t} \right)^2 \int_{\R^2} 1_{y \in B(x,t C r_0) \cup B(x',t C r_0)} dy t^{-3} dt \\
& \leq C \int_0^1 \left( t \wedge \frac{\abs{x-x'}}{t} \right)^2 \frac{dt}{t} \\
& \leq C \int_0^{\sqrt{\abs{x-x'}}} t dt + C|x-x'|^2 \int_{\sqrt{\abs{x-x'}}}^1 t^{-3} dt  \\
& \leq C \abs{x-x'}.
\end{align*}
where the constant $C$ in the right-hand side is uniform in $n$.  The second assertion directly follows from an analogous computation without keeping track of the $x,x'$.

\end{proof}

\begin{Prop}
\label{UniformEstimate}
There exist $C > 0$, $\sigma^2 > 0$ such that for every $n \geq 0$, $x\geq  0$,
$$
\mathbb{P}\left( \norme{ ( \delta \phi_1 + \delta \phi_{2} )_{|K} }_{\infty}  \geq x \right) \leq C e^{- x^2 / \sigma^2}.
$$
\end{Prop}

\begin{proof}
We have obtained in Lemma \ref{CI} a bound on the variance of $\delta \phi_2(x) - \delta \phi_2 (x')$ which is a centered Gaussian variable, hence it follows that $\mathbb{E}\left( \left(\delta \phi_2(x) - \delta \phi_2 (x')\right)^{2p} \right) = O(\abs{x-x'}^p)$. By the Kolmogorov continuity criterion, for any $\alpha < 1/2$, $ \mathbb{E}( \norme{ \delta \phi_{2}} _{C^{\alpha}(K)} ) $ is bounded in $n$. Together with Lemma \ref{SmoothVar}, this shows $\mathbb{E} ( \norme{ ( \delta \phi_1 + \delta \phi_{2} )_{|K} }_{\infty} ) $ is bounded. Consequently by Fernique (see \cite{Fernique}), we have a uniform Gaussian tail estimate in $n$.
\end{proof}

We are left with the noise $\delta \phi_3$ which is independent of $\phi_{0,n}$, $\delta \phi_1$ and $\delta \phi_2$. 

\begin{Lemma}
There exists $C > 0$ such that for every $x \in K$, $n\geq 0$, $\mathbb{E}\left( \delta \phi_3(x)^2 \right) \leq C$.
\end{Lemma}

\begin{proof}
Since  $|F'(y)|^{-1} \geq \norme{F'}_{U,\infty}^{-1} = c > 0$ holds for every $y \in U$ and as seen in the proof of Lemma \ref{CI} we can directly replace  the term $1_{ t \geq \frac{|F(x)-F(y)|}{r_0 |F'(y)|} }$ by $1_{ t \geq \frac{|x-y|}{C r_0} }$. This gives:
\begin{align*}
\mathbb{E}\left( \delta \phi_3(x)^2 \right) & = \int_{U} \int_{2^{-n-1} \abs{F'(y)}^{-1}}^{2^{-n-1}} k \left( \frac{F(x)-F(y)}{t \abs{F'(y)}} \right)^2 t^{-3} dt dy \\
& \leq \norme{k}_{\infty}^2 \int_{c 2^{-n-1}}^{2^{-n-1}} \int_{\R^2} 1_{y \in B(x,t C r_0)} t^{-3} dy dt\\
& \leq \norme{k}_{\infty}^2 \int_{c 2^{-n-1}}^{2^{-n-1}} C t^2 t^{-3} dt
\end{align*}
which concludes the proof. 
\end{proof}

In summary, we have seen that along this white noise coupling,
\begin{equation}
\label{eq:RecapCoupling}
\phi_{0,n} - \tilde{\phi}_{0,n} \circ F = \delta \phi_1 + \delta \phi_2  + \delta \phi_3
\end{equation}
where  $\delta \phi_1$ and $\delta \phi_2$ are low frequency noises with uniform Gaussian tails and $\delta \phi_3$ is a high frequency noise with bounded pointwise variance and dependence scale $O(2^{-n})$, which is independent of $\phi_{0,n}$, $\delta \phi_1$ and $\delta \phi_2$. 

\subsection{RSW estimates for crossing lengths}

Now we investigate the consequences of the approximate conformal invariance on crossing lengths. More precisely we want to show that the tails of the crossing lengths of rectangles of varying aspect ratios are comparable, uniformly in the roughness of the conformal factor by using \eqref{eq:RecapCoupling}.

\smallskip

Let $A,B$ be two boundary arcs of $K$ and denote by $L$ the distance from $A$ to $B$ in $K$ for the Riemannian metric $e^{\gamma \phi_{0,n}}ds^2$; we denote $A' :=F(A)$, $B' :=F(B)$, $K' :=F(K)$, and $L'$ is the distance from $A'$ to $B'$ in $K'$ for $e^{\gamma \tilde{\phi}_{0,n}}ds^2$. 
\begin{Prop} (Left tail estimate).
\label{LeftTail}
If for some  $l > 0$ and  $\eps < 1/2$, $\mathbb{P}\left(  L \leq l \right) \geq \eps$ ,then 
$$
\mathbb{P} \left( L' \leq l' \right) \geq \eps/4
$$
with $l' = C l e^{\frac{\gamma}{2} \sigma \sqrt{\abs{\log\eps / 2C}}} $ and $C, \sigma$ depend only on the geometry.
\end{Prop}
\begin{proof}
Assume that for some positive $l$, $\eps$, $\mathbb{P}\left(  L \leq l \right) \geq \eps$. Setting $x = \sigma \sqrt{ | \log( \eps /2C)|}$, we have, using the Proposition \ref{UniformEstimate}:
$$
\mathbb{P}\left( \norme{ ( \delta \phi_1 + \delta \phi_{2} )_{|K} }_{\infty}  \geq x \right) \leq  \eps /2
$$
and
$$
\mathbb{P}\left( \norme{ ( \delta \phi_1 + \delta \phi_{2} )_{|K} }_{\infty}  \leq x, L \leq l \right) \geq \eps/2.
$$
Thus, with probability at least $\eps/2$, the distance from $A$ to $B$ in $K$ for the metric $e^{\gamma( \phi_{0,n} - \delta \phi_1 - \delta \phi_2) }ds^2$  is $\leq l e^{\frac{\gamma}{2}x}$. On this event, we fix such a path of length $ \leq l e^{\frac{\gamma}{2}x}$ and average over the independent small scale noise $\delta \phi_3 $; the expected length of the path is $ \leq l e^{\frac{\gamma}{2} x} e^{C \gamma^2}$. With conditional probability at least $1/2$, this length is no more than twice the conditional expectation. Consequently, with probability at least $\eps/4$, the distance from $A$ to $B$ in $K$ for $e^{\gamma \tilde{\phi}_{0,n} \circ F} ds^2$ is less than $2le^{\frac{\gamma}{2}x} e^{C \gamma^2}$. Since $F'$ is bounded on $K$, we get that $\mathbb{P}\left( L' \leq l' \right) \geq \eps/4$ where $l' = 2  \norme{F'}_{K,\infty}  l e^{\frac{\gamma}{2}x} e^{C \gamma^2}$. Indeed, since $F$ is holomorphic, if $ \pi = (\pi_t)_{t \in [0,1]}$ is a $C^1$ path and if $\phi$ is a smooth field, we have: 
\begin{align*}
& L \left( F \circ \pi, e^{\gamma \phi} ds^2 \right) = \int_0^1 e^{\frac{\gamma}{2}  \phi \circ F (\pi(t))}  \abs{F'(\pi(t))} \abs{\pi'(t)} dt \\
& L \left( \pi, e^{\gamma \phi \circ F} ds^2 \right) = \int_0^1 e^{\frac{\gamma}{2} \phi \circ F (\pi(t))} \abs{\pi'(t)} dt.
\end{align*}
Thus, on the event $\lbrace L(A,B,e^{ \gamma \tilde{\phi}_{0,n} \circ F } ds^2) \leq 2 l e^{\frac{\gamma}{2} x}e^{C \gamma^2} \rbrace$ we have, taking such a path $\pi$: 
\begin{align*}
L \left(A',B', e^{\gamma \tilde{\phi}_{0,n}} ds^2 \right) & \leq L \left( F \circ \pi, e^{\gamma \tilde{\phi}_{0,n}} ds^2 \right) \\
& \leq \norme{F'}_{K,\infty}  L \left( \pi, e^{\gamma \tilde{\phi}_{0,n} \circ F} ds^2 \right) \\
& \leq 2  \norme{F'}_{K,\infty}  l e^{\frac{\gamma}{2}x} e^{C \gamma^2}
\end{align*}
 hence $\mathbb{P} \left( L' \leq l' \right) \geq \eps/4$ with $l' = C l e^{\frac{\gamma}{2} \sigma \sqrt{\abs{\log \eps / 2C}}} e^{C \gamma^2} \leq C l e^{\frac{\gamma}{2} \sigma \sqrt{\abs{\log \eps / 2C}}} $.
\end{proof} 

\begin{Prop} (Right tail estimate).
\label{RightTail}
If for some  $l > 0$ and  $\eps < 1/2$, $\mathbb{P}\left( L \leq l \right) \geq 1- \eps $ then
$$
\mathbb{P} \left( L' \leq l' \right) \geq 1 - 3 \eps 
$$
with $l'  = C l e^{C \gamma \sqrt{\abs{\log \eps/2C}}}$ and $C$ depends only on the geometry.
\end{Prop}

To prove Proposition \ref{RightTail}, we will need the following lemma which is a consequence of the moment method and which will be used in the next sections.

\begin{Lemma}
\label{MomentMethod}
Let $\mu$ be a Borel measure on a metric space $(X,d)$. If $S$ is a Borel set such that $\mu(S) \in (0,\infty)$ and $\psi$ is a continuous centered Gaussian field on $S$, satisfying $\sigma^2 := \sup_{x \in S} \mathrm{Var}(\psi(x)) < \infty$, then for every $s > \sigma^2$ we have
$$
\mathbb{P}\left( \int_S e^{\psi(x)} \mu(dx) \geq \mu(S) e^s \right) \leq e^{-s^2 / 2\sigma^2}.
$$
\end{Lemma}

\begin{proof}
By using first Chebychev inequality, then Jensen inequality and finally explicit formula for moment generating function of Gaussian variables, we have for $k > 1/2$:
\begin{align*}
\mathbb{P} \left( \int_S e^{\psi(x)} \mu(dx) \geq \mu(S) e^s \right) & \leq e^{-2ks} \mathbb{E} \left( \left( \frac{1}{\mu(S)} \int e^{\psi(x)} \mu(dx) \right)^{2k} \right) \\
&  \leq e^{-2ks} \mu(S)^{-1} \int_S \mathbb{E}\left( e^{ 2k \psi(x)} \right) \mu(dx) \\
& \leq e^{2k^2 \sigma^2 - 2ks}.
\end{align*}
By setting $k = \frac{s}{2 \sigma^2}$, we get the tail estimate for $s > \sigma^2$.
\end{proof}

We are now ready to prove Proposition \ref{RightTail}.

\begin{proof}[Proof of Proposition \ref{RightTail}]
Assume that for some positive $l, \eps$, $\mathbb{P}\left( L \leq l \right) \geq 1- \eps$.  Setting $x = \sigma \sqrt{ | \log( \eps /C)|}$ and using the estimate from Proposition \ref{UniformEstimate} we have:
$$
\mathbb{P}\left( \norme{ ( \delta \phi_1 + \delta \phi_{2} )_{|K} }_{\infty}  \geq x \right) \leq \eps
$$
and
$$
\mathbb{P}\left( \norme{ ( \delta \phi_1 + \delta \phi_{2} )_{|K} }_{\infty}  \leq x, L \leq l \right) \geq 1 -2 \eps.
$$
Consequently, with probability at least $1 - 2 \eps$, the distance from $A$ to $B$ in $K$ for the metric $e^{\gamma (\phi_{0,n} - \delta \phi_1 - \delta \phi_2)} ds^2$ is $\leq  l e^{ \frac{\gamma}{2} x}$. On this event, we fix such a path of length $\leq  l e^{ \frac{\gamma}{2} x}$ and average over the independent small scale noise $\delta \phi_3$. Let $\mu$ be the occupation measure of that path, so that $\abs{\mu} \leq l e^{\frac{\gamma}{2}x}$ and $\psi = \frac{\gamma}{2}(\delta \phi_3)$ is independent of $\mu$. Since $\sigma^2 := \sup_{[0,1]^2} \mathrm{Var ~} \psi = O(\gamma^2)$, by using Lemma \ref{MomentMethod}, we note that adding the noise $\delta \phi_3$ increases the length by a factor $\geq e^{C \gamma \sqrt{\abs{\log \eps}}}$ with probability $\leq \eps$. Consequently, with probability $\geq 1 - 3\eps$, the distance from $A$ to $B$ in $K$ for $e^{\gamma \tilde{\phi}_{0,n} \circ F}ds$ is less than $l e^{\frac{\gamma}{2}x} e^{C \gamma \sqrt{\abs{\log \eps}}}$. Using again $L(A',B', e^{\gamma \tilde{\phi}} ds^2 ) \leq \norme{F'}_{K,\infty} L(A,B,e^{\gamma \tilde{\phi} \circ F} ds^2 )$ we have $\mathbb{P} ( L' \leq l' ) \geq 1 - 3 \eps$
where $l' = \norme{F'}_{K,\infty} l e^{\frac{\gamma}{2}x} e^{C\gamma \sqrt{\abs{\log \eps}}}$.\end{proof}

To prove Theorem \ref{th:RSW}, we will need the following elementary lemma.

\begin{Lemma}
\label{LemCro}
If $a$ and $b$ are two positive real numbers with $a < b$, there exists $j = j(b/a)$ and $j$ rectangles isometric to $[0, a/2] \times [0, b/2]$  such that if $\pi$ is a left-right crossing of the rectangle $ [0,a] \times [0,b]$, at least one of the $j$ rectangles is crossed in the thin direction by a subpath of that crossing. 
\end{Lemma}

\begin{proof} 
To see it, cover for instance $[0,a/2] \times [0, b]$ by thin rectangles $[0,a/2] \times [0, b/2]$ from bottom to top and spaced by $(b-a)/4$, add also squares of length $a/2$ with the same spacing (see the first two parts on Figure \ref{fig:Lemma}). Then, starting with a crossing of $[0,a] \times [0,b]$, consider the subpath from the left side to the first hitting point of $\lbrace a/2 \rbrace \times [0,b]$, and denote by $h$ is height (max of $y$ - min of $y$). Consider first the case where $h \leq a/2 + (b-a)/4$ (see the last part on Figure \ref{fig:Lemma}). Since the bottom part of the path is at distance $\leq (b-a)/4$ of a side of a rectangle of size $[0,a/2] \times [0,b/2]$ the crossing is included in this rectangle of the cover. Now we treat the other case where $h > a/2 + (b-a)/4$ (see the third part on Figure \ref{fig:Lemma}). Since the bottom part is at distance $\leq (b-a)/4$ of a square which is above, this square of size $a/2$ is then crossed vertically.
\end{proof}

\begin{figure}[h!]
\centering
\includegraphics[scale=0.75]{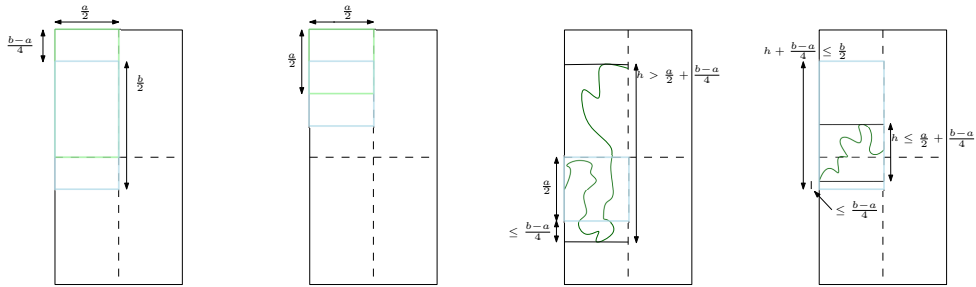}
\caption{Crossing at a smaller scale.}
\label{fig:Lemma}
\end{figure}

Now, we want to relate crossings of short rectangles with crossings of long rectangles. Our previous results say that the crossing lengths in $K$ between sides $A$ and $B$ are uniformly (in $n$) comparable to crossing lengths in $F(K)$ between sides $F(A)$ and $F(B)$. Thus, we would like to take the sides $A$ and $B$ to be those of a short rectangle and to map them to the sides of a long rectangle with a conformal map $F$ such that $F'$ and $F''$ are bounded and satisfying $\abs{F'} \geq 1$. This cannot be done directly but this is the main idea:  to produce a crossing from a short domain to a longer one. In particular, it is enough to consider ellipses and to relate crossings in ellipses with crossings in rectangles and by using the previous lemma one can begin with crossing of sides in a very small domain and then map it to a much larger domain.

\begin{proof}[Proof of Theorem \ref{th:RSW}]

The proof is divided in two steps. First we prove the inequality \eqref{eq:RSW1} associated with the left tail and then the inequality \eqref{eq:RSW2} associated to the right one.

\medskip

\textbf{Step 1.} We study first the left tail under the assumption $\mathbb{P} (L_{a,b}^{(n)} \leq l ) \geq \eps$ and we want to obtain a similar estimate for $L_{a',b'}^{(n)} $( in particular if $a/b <1<a'/b'$). We assume $a < b$, i.e. $L_{a,b}^{(n)}$ is the length of a crossing in the thin direction. 

\smallskip

First, by using Lemma \ref{LemCro}, we observe that there is an integer $j = j(b/a)$ and $j$ rectangles isometric to $[0, a/2] \times [0, b/2]$ such that on the event $L_{a,b}^{(n)} \leq l$, at least one of the $j$ rectangles is crossed in the thin direction by a subpath of that crossing. Thus, by union bound, we get $\mathbb{P} (L_{a/2,b/2}^{(n)}  \leq l ) \geq \eps/j$, and by iterating, $\mathbb{P} \left( L_{a/2^p,b/2^p}^{(n)} \leq l \right) \geq \eps / j^p$.

\smallskip

Consider now ellipses $E$, $E'$, each with two marked arcs, such that: any left-right crossing of $[0, a/2^p] \times [0, b/2^p]$ is a crossing of $E$, and any crossing of $E'$ is a left-right crossing of $[0, a'] \times [0, b']$. 

Divide the marked arcs of $E$ into $m$ subarcs of, say, equal length. With probability at least $ \eps/ (j^{p} m^2 )$, one of the crossings between pairs of subarcs has length at most $l$.

\begin{figure}[h!]
\centering
\includegraphics[scale=0.5]{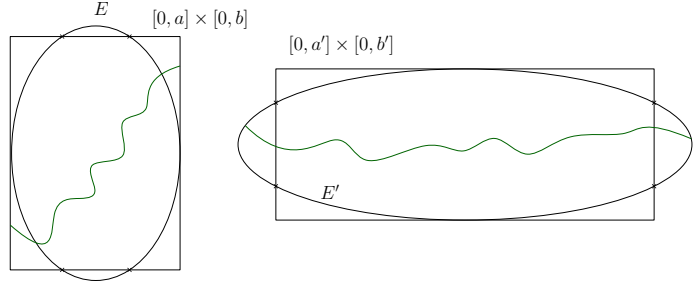}
\caption{Rectangles and ellipses}
\label{fig:ellipses}
\end{figure}

For $m$ large enough (depending on $E$, $E'$), for any pair of such subsegments (one on each side), there is a conformal equivalence $F : E \rightarrow E'$ such that the pair of subarcs is mapped to subarcs of the marked arcs of $E'$. Remark that ellipses are analytic curves (they are images of circles under the Joukowski map, see \cite{Joukowski}  Chapter 1 Exercise 15) and consequently (by Schwarz reflection) $F$ extends to a conformal equivalence $U \rightarrow V$, where $\bar{E}$ (resp. $\bar{E}'$) is a compact subset of $U$ (resp. $V$ ).

\smallskip

By choosing $p$ large enough, $|F'| \geq 1$ on $U$. By the left tail estimate Proposition \ref{LeftTail}, we obtain that there is $C > 0$ such for all $\eps, l > 0$:
$$
\mathbb{P} \left( L_{a,b}^{(n)} \leq l \right) \geq \eps \Rightarrow \mathbb{P}\left( L_{a',b'}^{(n)} \leq C l e^{\frac{\gamma}{2}\sigma \sqrt{\abs{\log \eps / (2C j^p m^2)}}} \right) \geq \eps/(4 j^p m^2)
$$
which we rewrite as:
\begin{equation}
\label{LeftRect}
\mathbb{P} \left( L_{a,b}^{(n)} \leq l \right) \geq \eps \Rightarrow  \mathbb{P}\left( L_{a',b'}^{(n)} \leq C l e^{C \gamma \sqrt{\abs{\log \eps / C}}} \right) \geq \eps/C.
\end{equation}

\bigskip

\textbf{Step 2.} For the right tail we reason similarly: let $a < b$ and take $l,\eps$ so that $ \mathbb{P} ( L_{a,b}^{(n)} \leq l ) \geq 1 - \eps$. On the event $\lbrace L_{a,b}^{(n)} \leq l \rbrace$, one of $j$ variables distributed like $L_{a/2,b/2}^{(n)}$ is $\leq l$; moreover these variables have positive association. By the the positive association property (Theorem \ref{PosAssoc}) and the square-root trick (see \cite{Ta} Proposition 4.1), we have $ \mathbb{P} ( L_{a/2,b/2}^{(n)}  \leq  l )  \geq 1 - \eps^{1/j}$ and then, by iterating, $\mathbb{P} ( L_{a/2^p,b/2^p}^{(n)} \leq l ) \geq 1 - \eps^{j^{-p}}$. 

\smallskip

On the event $\lbrace L_{a/2^p,b/2^p}^{(n)} \leq l \rbrace$, the ellipse $E$ has a crossing of length $\leq l$ between two marked arcs. Again by subdividing each of these arcs into $m$ subarcs, and applying the square-root trick we see that for at least one pair of subarcs, there is a crossing of length $\leq l$ with probability $\geq 1 - \eps^{j^{-p} m^{-2}}$. Combining with the right-tail estimate Proposition \ref{RightTail}, we get:
\begin{equation}
\label{RightRect}
\mathbb{P} \left( L_{a,b}^{(n)} \leq l \right) \geq 1 - \eps \Rightarrow \mathbb{P}\left( L_{a',b'}^{(n)} \leq C l e^{\gamma C \sqrt{\abs{\log \eps /C}}} \right) \geq 1 - 3 \eps^{1/C}
\end{equation}
which completes the proof of Theorem \ref{th:RSW}.
\end{proof} 

\section{Tail estimates for crossing lengths: proof of Theorem \ref{th:Tails}}

\subsection{Concentration: the left tail} Denote by $\tilde{L}_{1,3}^{(n)}$ (resp. $\tilde{L}_{3,1}^{(n)}$) the left-right crossing length of the rectangle $[2,3] \times [0,3]$ (resp. $[0,3] \times [2,3]$). In this subsection we investigate the consequences of the RSW estimates combined with the following inequalities (see Figure 2): 
$$
L_{1,3}^{(n)} + \tilde{L}_{1,3}^{(n)}  \leq L_{3,3}^{(n)} \leq \min \left( L_{3,1}^{(n)} , \tilde{L}_{3,1}^{(n)} \right)
$$
which implies the following: 
\begin{align*}
L_{3,3}^{(n)} \leq l \Rightarrow \left( L_{1,3}^{(n)} \leq l ~~  \mathrm{and} ~~ \tilde{L}_{1,3}^{(n)} \leq l \right) \\
L_{3,3}^{(n)} \geq l \Rightarrow \left( L_{3,1}^{(n)} \geq l ~~  \mathrm{and} ~~ \tilde{L}_{3,1}^{(n)} \geq l  \right).
\end{align*}

\begin{figure}[h!]
\centering
\includegraphics[scale=0.5]{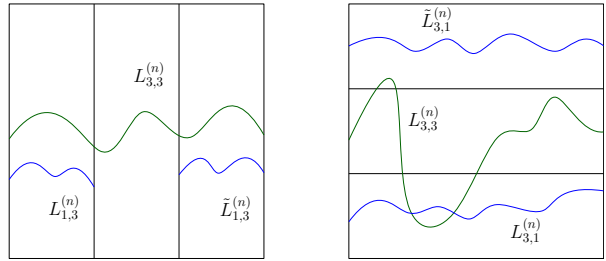}
\caption{Inequalities between lengths of geodesics associated to different rectangles}
\label{fig:Inequa}
\end{figure}

The following result is a consequence of the first inequality. It gives lognormal tail estimates on the left tail of crossing lengths renormalized by a small quantile, without any assumption on $\gamma$.
\begin{Prop}
\label{PropLower} There exists a small $p_0 >0$ such that for $p \leq p_0$ there exists $c > 0$ so that for every $s > 0$ 
$$
\mathbb{P} \left( L_{3,3}^{(n)} \leq l_{3,3}^{(n)} \left( p \right) e^{- s} \right) \leq C e^{-c s^{2}},
$$
where $c,C$ do not depend on $n$.
\end{Prop}

\begin{proof}
Our left tail estimate \eqref{LeftRect} gives:
$$
\mathbb{P}\left( L_{1,3}^{(n)} \leq l \right) \geq \eps \Rightarrow \mathbb{P}\left( L_{3,3}^{(n)} \leq l' \right) \geq \eps/C ~~ \mathrm{with} ~~ l' = Cl e^{C \gamma \sqrt{\abs{\log \eps / C}}}
$$
which can be rewritten as: 
\begin{equation}
\label{ReUse}
\mathbb{P}\left( L_{3,3}^{(n)} \leq l \right) \leq \eps \Rightarrow \mathbb{P}\left( L_{1,3}^{(n)} \leq l C^{-1} e^{-C \gamma \sqrt{\abs{\log C \eps}}} \right) \leq C \eps.
\end{equation}
Now, if $L_{3,3}^{(n)}$ is less than $l$, then both $[0,1] \times [0,3]$ and $[2,3] \times [0,3]$ have a left-right crossing of length $\leq l$ and the field in these two rectangles is independent (if $r_0$ is small enough). Consequently, 
\begin{equation}
\label{decaysquare}
\mathbb{P}\left( L_{3,3}^{(n)} \leq l \right) \leq \mathbb{P}\left( L_{1,3}^{(n)} \leq l \right)^2.
\end{equation}
These two results allow us to get the uniform tail bound. Indeed, take $\eps_0$  small, such that $C^2 \eps_0 < 1$ and set $r_0^{(n)} := l_{3,3}^{(n)}(\eps_0)$. We define by induction $\eps_{i+1} := (C \eps_i)^2$ (which gives $\eps_i = (\eps_0 C^2)^{2^i} C^{-2}$ as well as $r_{i+1}^{(n)} := r_{i}^{(n)} C^{-1} \exp (-C \gamma \sqrt{ | \log (C \eps_i) |})$.  It follows by induction that $\mathbb{P} ( L_{3,3}^{(n)} \leq r_i^{(n)} ) \leq \eps_i$ for every $i \geq 0$. Indeed, the case $i = 0$ follows by definition and then notice that the RSW estimates under the induction hypothesis implies that 
$$
\mathbb{P}\left( L_{3,3}^{(n)} \leq r_i^{(n)} \right) \leq \eps_i \Rightarrow \mathbb{P}\left( L_{1,3}^{(n)} \leq r_{i+1}^{(n)} \right) \leq C \eps_i
$$
which gives, using the inequality \eqref{decaysquare}: 
$$
\mathbb{P}\left( L_{3,3}^{(n)} \leq r_{i+1}^{(n)} \right) \leq \mathbb{P}\left( L_{1,3}^{(n)} \leq r_{i+1}^{(n)} \right)^2 \leq  (C \eps_i )^2 = \eps_{i+1}.
$$
Notice that we have the lower bound on $r_i^{(n)}$ for $i \geq 1$:
\begin{align*}
r_i^{(n)} \geq l_{3,3}^{(n)}(\eps_0) C^{-i} e^{-C \gamma \sum_{k = 0}^{i-1} \sqrt{| \log (C \eps_k)  |} } \geq  l_{3,3}^{(n)}(\eps_0) e^{-C i} e^{-C \gamma \sqrt{|\log \eps_0 C^2 |} 2^{i/2}}.
\end{align*}
Our estimate then takes the form, for $i \geq 0$: 
$$
\mathbb{P} \left( L_{3,3}^{(n)} \leq l_{3,3}^{(n)}(\eps_0) e^{-C i} e^{-\gamma C \sqrt{|\log \eps_0 C^2 |} 2^{i/2}} \right) \leq \left( \eps_0 C^2 \right)^{2^i} C^{-2}.
$$
Which can be rewritten, taking $i = \lfloor 2 \log_2 s \rfloor$, with absolute constants, for $s \geq 1$: 
$$
\mathbb{P}\left( L_{3,3}^{(n)} \leq l_{3,3}^{(n)}(\eps_0) C^{-1} e^{-C \log s} e^{-\gamma s} \right) \leq  e^{-c s^2}.
$$
Notice that dropping the dependence on $\gamma$ as we impose it is bounded from above by a large number we get Proposition \ref{PropLower}.
\end{proof}
\begin{Cor}
\label{LowerTailThin}
We have a uniform (in $n$) lognormal tail estimates for the lower bound of thin rectangles i.e. if $\eps_0$ is small enough for every $n \geq 0$, $s \geq 0$:
$$
\Pro \left( L_{1,3}^{(n)} \leq l_{1,3}^{(n)}(\eps_{0}) e^{-s} \right) \leq C e^{-c s^2},
$$
where $c,C$ are absolute constants. 
\end{Cor}
\begin{proof}
The proof follows from the RSW estimate \eqref{ReUse}, the bound $l_{1,3}^{(n)}(\eps_0) \leq l_{3,3}^{(n)}(\eps_0)$ and the previous proposition.
\end{proof}
\noindent It is tempting to follow the lines of this proof using the second inequality (see also Figure \ref{fig:Inequa}) in order to derive a right tail estimate. However, this approach cannot be readily extended  because of the power $1/C$ in the RSW estimate, inequality \eqref{eq:RSW2}.

\subsection{Concentration: the right tail}

As mentioned in the previous section, we cannot generalize the method used for the left tails to the right one and the following proposition remediates to this. Before stating it, we refer the reader to the definitions of $l_n$  and $\delta_n$ in Subsection \ref{DefLengths}.

\begin{Prop}
\label{LemmeRecu} 
If $\eps$ is small enough we have the following tail estimate:\\
For $0 \leq k \leq n$, $s > 1$
$$
\Pro \left( L_{3,1}^{(k)} \geq \delta_n l_{k} e^{s \sqrt{\log s}} \right) \leq C e^{-c s^2},
$$
where $c$ and $C$ are absolute constants. 
\end{Prop}

\begin{proof} We proceed according to the following steps:

\begin{enumerate}
\item Use the RSW estimates to reduce the problem to the case of squares instead of long rectangles. 

\item Use a comparison to 1-dependent oriented site percolation to prove that with probability going to one exponentially in $k$, $L_{k,k}^{(n)}$ is less than $C k \bar{l}_n$.

\item By scaling and the moment method, obtain a first tail estimate of $L_{1,1}^{(n)}$ with respect to $\bar{l}_{n-m}$: \\ For a constant $\alpha \in (0,1)$, $\mathbb{P}\left( L_{1,1}^{(n)} \geq C \bar{l}_{n-m} e^{ \gamma s \sqrt{m} } \right) \leq  C \alpha^{2^m} + e^{- \frac{ 2 s^2}{\log 2}}$ .

\item Give an upper bound of $\bar{l}_{n-m}$ in terms of $l_n$.  

\item Obtain a tail estimate when the tails are not too large. 

\item For the large tails, use a moment method and a lower bound on the quantiles. 
\end{enumerate}

\medskip

\textbf{Step 1.} First, notice by the RSW estimates \eqref{RightRect} that it is enough to prove that for $0 \leq k \leq n$, $s > 1$
$$
\Pro \left( L_{1,1}^{(k)} \geq \delta_n l_{k} e^{s \sqrt{\log s}} \right) \leq C e^{-c s^2}.
$$

\textbf{Step 2.} We will see here that taking $\eps$ small enough, there exist $C > 0$, $\alpha < 1$ such that for every $k,n \geq 0$: 
\begin{equation}
\label{PercoLin}
\mathbb{P}\left( L_{k,k}^{(n)} \leq 4 k \bar{l}_n \right) \geq 1 - C \alpha^k.
\end{equation}

We consider a graph whose sites $x$ are made by squares of size $3 \times 3$ and spaced so that two adjacent squares intersect each other along a rectangle of size $ (3,1)$ or $(1,3)$. Denote by $L_{3,1,\mathrm{right}}^{(n)}(x)$ the rectangle crossing length, in the long direction, associated to the rectangle of size $(3,1)$ on the bottom of $x$ and included in $x$. Similarly, denote by $L_{3,1,\mathrm{up}}^{(n)}(x)$  the rectangle crossing length, in the long direction, associated to the rectangle of size $ (1,3)$ on the left of $x$ and included in $x$. To each site of our graph, we assign the value $0$ if the site is closed and $1$ if the site is open.  A site $x$ is open if the event $\lbrace L_{3,1,\mathrm{up}}^{(n)}(x)  + L_{3,1,\mathrm{right}}^{(n)}(x)  \leq 2 \bar{l}_n \rbrace$ occurs (see Figure \ref{fig:DefPerco}). 
\begin{figure}[h!]
\centering
\includegraphics[scale=0.75]{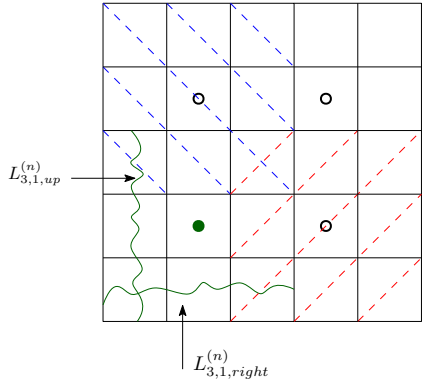}
\caption{Definition of the model. The green site $x$ is open. Three of its neighbors are drawn, with some colored dashed lines filling their cell and with white vertices at their center.}
\label{fig:DefPerco}
\end{figure}

We have the following bound on the probability that a site $x$ is open: 
$$
\Pro \left( \omega_x = 1 \right) \geq \Pro \left( L_{3,1,\mathrm{up}}^{(n)} \leq \bar{l}_n, L_{3,1,\mathrm{right}}^{(n)} \leq \bar{l}_n \right) \geq 2 \Pro \left( L_{3,1}^{(n)} \leq \bar{l}_n \right) - 1 \geq 1 - 2\eps.
$$
Therefore,  taking $\eps$ small gives a highly supercritical 1-dependent percolation model (notice that a site $x$ is independent of sites that are not directly weakly adjacent to it). Then, notice that $L_{k,k}^{(n)} $ is smaller than the weight associated to oriented paths from left to right at the percolation level that can go only up or right. Such a path contains at most $2k$ sites. Thus, if there is an open oriented percolation path from left to right, then $L_{k,k}^{(n)} \leq 4k \bar{l}_{n}$. Hence it is enough to show that the probability that there is such an open oriented path goes to $1$ exponentially in $k$. This follows from a contour argument for highly supercritical 1-dependent percolation model, see for instance \cite{DurrettOr} Section 10.
\begin{figure}[h!]
\centering
\includegraphics[scale=0.5]{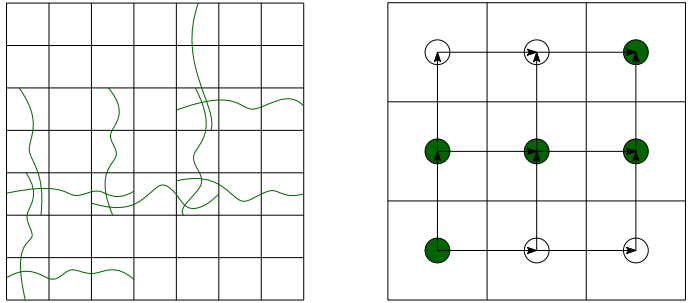}
\caption{Comparison with 1-Dependent Oriented Site Percolation. The figure on the right is the representation of the figure on the left.}
\label{fig:Perco}
\end{figure}

\medskip

\textbf{Step 3.} In order to obtain an upper bound for $L_{1,1}^{(n)}$, by scaling and the percolation bound \eqref{PercoLin} we see that there exists $\alpha \in (0,1)$ such that for $m \leq n$, we have,
$$
\mathbb{P}\left( L_{1,1}^{(m,n)} \leq C \bar{l}_{n-m}  \right) = \mathbb{P}\left( L_{2^m,2^m}^{(n-m)} \leq C 2^m \bar{l}_{n-m}  \right) \geq 1 - C \alpha^{2^m}
$$
which can be rewritten in term of $L_{1,1}^{(n)}$ as
\begin{align*}
\mathbb{P}\left( L_{1,1}^{(n)} \leq  C \bar{l}_{n-m} e^{s} \right) & \geq \mathbb{P} \left( L_{1,1}^{(n)} \leq C  \bar{l}_{n-m} e^{s}  , L_{1,1}^{(m,n)} \leq C \bar{l}_{n-m} \right) \\
& = \mathbb{P}\left( L_{1,1}^{(m,n)} \leq C \bar{l}_{n-m}  \right) - \mathbb{P} \left( L_{1,1}^{(n)} \geq C \bar{l}_{n-m} e^{s}  , L_{1,1}^{(m,n)} \leq C \bar{l}_{n-m} \right) \\
& \geq 1 - C \alpha^{2^m} - \mathbb{P}\left( L_{1,1}^{(n)} \geq e^s L_{1,1}^{(m,n)} \right).
\end{align*}
Now, using that $L_{1,1}^{(n)} \leq \int_{\pi_{m,n}} e^{\frac{\gamma}{2} \phi_{0,m-1}} e^{\frac{\gamma}{2} \phi_{m,n}} ds$ where $\pi_{m,n}$ is a geodesic for $e^{\gamma \phi_{m,n}} ds^2$ and using the bound coming from Lemma \ref{MomentMethod} we have
$$
\mathbb{P}\left( L_{1,1}^{(n)} \geq e^{\gamma \sqrt{m} s} L_{1,1}^{(m,n)} \right) \leq \E \left( \mathbb{P}\left(\int_{\pi_{m,n}} e^{\frac{\gamma}{2} \phi_{0,m-1}} e^{\frac{\gamma}{2} \phi_{m,n}} \geq e^{\gamma \sqrt{m} s}  L_{1,1}^{(m,n)}  ~ | ~ \mathcal{F}_{m,n}  \right) \right) \leq  e^{- \frac{ 2 s^2}{\log 2}}
$$
hence for every $0 \leq m \leq n$ and $s \geq 0$
\begin{equation}
\label{PreRightTail}
\mathbb{P}\left( L_{1,1}^{(n)} \leq C \bar{l}_{n-m} e^{ \gamma s \sqrt{m} } \right) \geq 1 - C \alpha^{2^m} - e^{- \frac{ 2 s^2}{\log 2}}.
\end{equation}

\medskip

\textbf{Step 4.} At this stage we want to replace $\bar{l}_{n-m}$ by $l_n$. We introduce a notation for a collection of short rectangles that we will use by setting
\begin{equation}
\label{Def:Ik}
I_k :=  \lbrace \text{horizontal, vertical rectangles of size } 2^{-k}(1,3) \text{ with corners in } [0,1] \times [0,3] \cap 2^{-k} \Z^2 \rbrace.
\end{equation}
It is clear from the definition that $\abs{I_k} \leq C 4^k$.Then, notice that a left-right crossing of $[0,1] \times [0,3]$ has to cross at least $2^k$ rectangles from $I_k$ (by definition of $I_k$, these are short crossings). For $P \in I_k$, we set
\begin{equation}
\label{Def:LengthFun}
L^{(n)}(P):= \text{length of the left-right crossing of the rectangle } P \text{ for } e^{\frac{\gamma}{2} \phi_{0,n}} ds 
\end{equation}
and we use similarly the notation $L^{(k,n)}(P)$ when the field considered is $\phi_{k,n}$. We have, almost surely,
\begin{equation}
\label{LowerBoundScale}
L_{1,3}^{(n)} \geq  2^k  \min_{P \in I_k} L^{(n)}(P)   \geq 2^k e^{ \frac{\gamma}{2} \underset{[0,1]^2}{\inf} \phi_{0,k-1}}  \min_{P \in I_k} L^{(k,n)}(P).
\end{equation} 
Hence by union bound and scaling, we have, for $s_1 > 0$ and $s_2 > 0$ to be specified
\begin{align*}
\mathbb{P} \left( L_{1,3}^{(n)} \leq  e^{-\frac{\gamma}{2} s_1} l_{n-k} e^{-s_2} \right) & \leq \mathbb{P} \left( e^{ \frac{\gamma}{2} \underset{[0,1]^2}{\inf} \phi_{0,k-1}} 2^{k} \min_{P \in I_k} L^{(k,n)}(P)   \leq  e^{-\frac{\gamma}{2}  s_1} l_{n-k} e^{-s_2}  \right) \\ 
& \leq \mathbb{P}\left( e^{\frac{\gamma}{2} \underset{[0,1]^2}{\inf} \phi_{0,k-1}} \leq e^{-\frac{\gamma}{2} s_1} \right) + \mathbb{P} \left( \min_{P \in I_k} L^{(k,n)}(P)    \leq 2^{-k} l_{n-k} e^{-s_2} \right) \\
& \leq \mathbb{P}\left( \underset{[0,1]^2}{\sup} \abs{\phi_{0,k-1}} \geq s_1 \right) +  C 4^k \mathbb{P} \left( L_{1,3}^{(n-k)}    \leq l_{n-k} e^{-s_2} \right).
\end{align*}
Using the supremum tail estimate from the appendix \eqref{TailSupSmall} with $s_1 = k \log 4 + C \sqrt{k} + C s$ and the lognormal tails from Corollary \ref{LowerTailThin} with $s_2 = C \sqrt{k \log 4 + s}$ we have
$$
\mathbb{P}  \left( L_{1,3}^{(n)} \leq l_{n-k} 2^{-\gamma k} e^{-C \sqrt{k}} e^{-Cs} e^{ - C \sqrt{s}}  \right)  \leq C e^{-s},
$$
which gives 
\begin{equation}
\label{LowerQuantilesScales}
l_n \geq 2^{-\gamma k} e^{-C \sqrt{k}} e^{-C} l_{n-k},
\end{equation}
hence $\bar{l}_{n-m} \leq l_{n-m} \delta_n  \leq l_n   \delta_n 2^{\gamma m} e^{C \sqrt{m}} C$. 

\medskip

\textbf{Step 5.} Using this bound and coming back to our estimate \eqref{PreRightTail}, for every $ m \leq n$ and $s \geq 0$
$$
\mathbb{P}\left( L_{1,1}^{(n)} \leq l_n  \delta_n 2^{\gamma m} e^{C \sqrt{m}} C  e^{\gamma s \sqrt{m}} \right) \geq 1 - C \alpha^{2^m} - e^{- \frac{ 2 s^2}{\log 2}}.
$$
We deal with the range $s \in [1,2^{n/2}]$, taking $m$ such that $s = 2^{m/2}$ i.e. $m = \lfloor 2 \log_2 s \rfloor$ we get: 
$$
\mathbb{P}\left( L_{1,1}^{(n)} \leq l_n   \delta_n e^{C \gamma \log s} e^{\gamma s \sqrt{\log s}} \right) \geq 1 - C e^{-c s^2},
$$
which gives, dropping the dependence on $\gamma$ for $s >1$:
$$
\mathbb{P}\left( L_{1,1}^{(n)} \geq l_n   \delta_n e^{ s \sqrt{\log s}} \right) \leq C e^{-c s^2}.
$$

\medskip

\textbf{Step 6.} We then treat the case $s \geq 2^{n/2}$. To do it, we use a moment method (Lemma \ref{MomentMethod}) to get a right tail estimate on $L_{1,1}^{(n)}$ together with a lower bound on its quantiles. The moment method (taking a straight line) gives: 
\begin{equation}
\label{MomentLine}
\mathbb{P} \left( L_{1,1}^{(n)} \geq e^{\gamma s} \right) \leq e^{- \frac{2 s^2}{(n+1)\log 2}}.
\end{equation}
For the lower bound on quantile, we get a  bound by a direct comparison with the supremum of the field $\mathbb{P} ( L_{1,3}^{(n)} \leq e^{- \frac{\gamma}{2} x} ) \leq \mathbb{P} (\sup_{[0,1]^2} \phi_{0,n} \geq x )$. Using the supremum tails from the appendix \eqref{TailSupSmall} i.e. taking $x = n \log 4 + C \sqrt{n} + C s $ gives $l_n \geq e^{-\frac{\gamma}{2}(n \log 4+C\sqrt{n} + C)} =: e^{- \gamma x_n}$. Since we consider the case $s \geq 2^{n/2}$, $s \geq x_n$ and $n \leq  2 \log_2 s $ and coming back to \eqref{MomentLine} leads to
$$
\mathbb{P} \left( L_{1,1}^{(n)}  \geq l_{n} e^{\gamma s}  \right) \leq \mathbb{P} \left( L_{1,1}^{(n)}  \geq e^{\gamma (s-x_n)}  \right)  \leq  e^{ - 2 \frac{ (s-x_n)^2}{(n+1) \log 2}} \leq e^{Cs} e^{-\frac{s^2}{\log s}}.
$$
Finally, combining the two inequalities ends the proof.
\end{proof}

\subsection{Quasi-lognormal tail estimates at subcriticality}

In this subsection we prove Theorem \ref{th:Tails}. The main idea is the following: the tightness of $\log L_{1,1}^{(n)} - \log \mu_n$ shows that the ratio between low and high quantiles of $L_{1,1}^{(n)}$ is bounded. Using the RSW estimates, it implies that $\delta_{\infty} < \infty$ which gives, uniformly in $n$, $\mu_n \leq C l_n$. The tails are then obtained using Corollary \ref{LowerTailThin} (with $l_n \geq \mu_n C^{-1}$)  and Proposition \ref{LemmeRecu} (with $\delta_n l_n \leq \delta_{\infty} \mu_n$).

\begin{proof}[Proof of Theorem \ref{th:Tails}]  Assuming $\gamma < \gamma_c$ gives the tightness of  $( \log L_{1,1}^{(n)} - \log \mu_n )_{n \geq 0} $. Thus, for every $\eps > 0$ there exists $C_{\eps} > 0$ such that for every $n \geq 0$, $\mathbb{P} ( L_{1,1}^{(n)} \leq \mu_n e^{-C_{\eps}} ) \leq \eps/C$ and $\mathbb{P} ( L_{1,1}^{(n)} \geq \mu_n e^{C_{\eps}} ) \leq \eps^C /3$ which can be rewritten as 
$$
\mu_n  e^{-C_{\eps} } \leq l_{1,1}^{(n)} (\eps/C) \leq \mu_n \leq \bar{l}_{1,1}^{(n)}(\eps^C /3)  \leq \mu_n e^{C_{\eps}}.
$$
Combining with the RSW estimates \eqref{RSWQuantiles}, we have
$$
\mu_n e^{-C_{\eps}} \leq l_{1,1}^{(n)} (\eps/C) e^{-C_{\eps}} \leq l_{1,3}^{(n)} (\eps) \leq l_{1,1}^{(n)}(\eps) \leq \mu_n \leq \bar{l}_{1,1}^{ (n)} (\eps) \leq \bar{l}_{3,1}^{(n)}(\eps) \leq \bar{l}_{1,1}^{(n)}(\eps^C /3) e^{C_{\eps}} \leq \mu_n e^{C_{\eps}}.
$$
In particular, $\delta_n \leq e^{C_{\eps}}$ holds for every $n \geq 0$ hence $\delta_{\infty}(\eps) = \sup_{n \geq 0} \delta_{n}(\eps) < \infty$. 

\medskip

We prove now the lower tail estimates. We have $l_n \geq \mu_n e^{-C_{\eps}}$ for every $n \geq 0$ hence using Corollary \ref{LowerTailThin} we get Theorem \ref{QuasiGaussianLow} when $(a,b) = (1,3)$. For the upper tails since $\delta_{\infty} < \infty$ and $l_n \leq \mu_n$ we can use Proposition \ref{LemmeRecu} to get Theorem \ref{QuasiGaussianUp} for the case $(a,b) = (3,1)$. The general case follows from the RSW estimates. 
\end{proof}
When $\gamma < \gamma_c$, we expect the existence of a $ \rho \in (0, 1)$ such that $l_n = \rho^{n + o(n)}$ and $\bar{l}_n = \rho^{n + o(n) }$. However, we don't need this level of precision and the following a priori bounds are enough for our analysis.
\begin{Lemma}
\label{BoundsScales}
If $0 < \eps < 1/2$ we have the following inequalities relating quantiles, for every $0 \leq k \leq n$: 
\begin{enumerate}
\item
for the the lower quantiles $l_{n-k} \leq 2^{\gamma k} e^{C \sqrt{k}} l_n$,

\item
if $\gamma < \gamma_c$, $\bar{l}_{n} \leq e^{C \sqrt{k}} \bar{l}_{n-k}$,

\item
and still under the assumption $\gamma < \gamma_c$, $e^{-C} \mu_n \leq  l_n \leq \mu_n \leq \bar{l}_n  \leq e^{C} \mu_n$.
\end{enumerate}
\end{Lemma}

\begin{proof}
The first point follows from the proof of Proposition \ref{LemmeRecu}, see \eqref{LowerQuantilesScales}. For the second point, using Lemma \ref{MomentMethod} gives
\begin{align*}
\mathbb{P}\left( L_{1,1}^{(n)} \geq e^{\gamma \sqrt{k} s} L_{1,1}^{(n-k)} \right) & \leq \E \left( \mathbb{P}\left(\int_{\pi_{n-k}} e^{\frac{\gamma}{2} \phi_{0,n-k}} e^{\frac{\gamma}{2} \phi_{n-k,n}} \geq e^{\gamma \sqrt{k} s}  L_{1,1}^{(n-k)}  ~ | ~ \mathcal{F}_{0,n-k}  \right) \right)  \\
& \leq  e^{- \frac{ 2 s^2}{\log 2}}
\end{align*}
hence $\Pro \left( L_{1,1}^{(n)} \geq \bar{l}_{n-k} e^{\gamma \sqrt{k} s} e^s \right) \leq e^{- \frac{ 2 s^2}{\log 2}} + \Pro \left( L_{1,1}^{(n-k)} \geq \mu_{n-k} e^s \right)$ and the result follows from Theorem \ref{th:Tails}.
The last point follows from the previous proof. 
\end{proof}

\subsection{Lower bounds on the tails of crossing lengths}

The following result, independent of the value of $\gamma$, shows that we cannot expect better than uniform lognormal tails. Its proof is essentially an application of the Cameron-Martin theorem and we see there that the lower bounds are already provided by the low frequencies of the field. 

\begin{Prop}
\label{LBounds}
There exist positive constants $c, C$ such that for every $n \geq 0$, $x > 0$: \\ $\Pro \left(L_{1,1}^{(n)} \leq  \mu_n e^{-  x} \right) \geq  c  e^{- C x^2}$ and $\Pro \left(L_{1,1}^{(n)} \geq  \mu_n e^{ x} \right) \geq   c  e^{- C x^2} $.
\end{Prop}

\begin{proof}
If $x \in [0,1]^2$, for every $t \in (0,1)$, the Euclidean ball centered at $x$ with radius $t r_0$ is included in the $r_0$ neighborhood of $[0,1]^2$, denoted by $([0,1]^2)^{r_0}$. Since $k$ has compact support in $B(0,r_0)$, 
\begin{align*}
\int_{\frac{1}{2}}^1 \int_{\R^2}  k \left( \frac{x-y}{t} \right) t^{-3/2}  1_{y \in ([0,1]^2)^{r_0}} dy dt & = \int_{\frac{1}{2}}^1 \int_{B(x, t r_0)} k \left( \frac{x-y}{t} \right) t^{-3/2}  dy dt \\
& = \int_{\frac{1}{2}}^1 \int_{B(0, t r_0)} k \left( \frac{y}{t} \right) t^{-3/2}  dy dt  
\end{align*}
is independent of $x$ and is equal to some positive real number $h$.

\smallskip

Let $M$ be a real number. By the Cameron-Martin theorem (see \cite{DaPrato} Section 2), since $M 1_{[\frac{1}{2},1] \times ([0,1]^2)^{r_0}}$ is square-integrable, $\xi + M 1_{[\frac{1}{2},1] \times ([0,1]^2)^{r_0}}$ is absolutely continuous with respect to $\xi$ and its Radon-Nikod\'ym derivative is given by the Cameron-Martin formula:
$$
\frac{d \mathcal{L} \left(\xi + M 1_{[\frac{1}{2},1] \times ([0,1]^2)^{r_0}} \right)}{d \mathcal{L} \left( \xi \right)} = \exp \left( M \langle \xi, 1_{[\frac{1}{2},1] \times ([0,1]^2)^{r_0}}  \rangle - g \frac{M^2}{2}  \right)
$$
where $g := \frac{1}{2} \mathrm{Leb}(([0,1]^{2})^{r_0})$. We introduce the field $\phi_{0,n}^M$ associated to $\xi + M 1_{[\frac{1}{2},1] \times ([0,1]^2)^{r_0}}$, i.e. for $x \in \R^2$,
$$
\phi_{0,n}^M(x) := \int_{2^{-n-1}}^1 \int_{\R^2} k \left( \frac{x-y}{t} \right) t^{-3/2} \left( \xi(dy,dt) +  M 1_{[\frac{1}{2},1] \times ([0,1]^2)^{r_0}}(t,y) dy dt \right)
$$
and using the previous remark, we notice that $\phi_{0,n}^{M}$ is equal to $\phi_{0,n} + M h$ on $[0,1]^2$. Thus, using the Cameron-Martin theorem, if $I$ is an interval, we have for $n \geq 0$ and $a >0$:
\begin{align*}
\Pro \left( L_{1,1}^{(n)} \in e^{- \frac{\gamma}{2} h M} I \right) & = \Pro \left(L_{1,1} \left( \phi_{0,n}^M \right) \in  I \right)  \\
& = \E \left( 1_{L_{1,1}^{(n)}  \in I}  \exp \left( M \langle \xi, 1_{[\frac{1}{2},1] \times ([0,1]^2)^{r_0}}  \rangle - g \frac{M^2}{2}  \right) \right) \\
& \geq \left( \Pro \left( L_{1,1}^{(n)} \in I \right) + \Pro \left(  \langle \xi, 1_{[\frac{1}{2},1] \times ([0,1]^2)^{r_0}}  \rangle \in (-a,a) \right) -1 \right) e^{- a \abs{M}} e^{- \frac{g M^2}{2}}.
\end{align*}
Taking $I = (0,\mu_n]$ and $M = x > 0$ gives, with $a$ large enough but fixed, 
$$
\Pro \left(L_{1,1}^{(n)} \leq  \mu_n e^{-\frac{\gamma}{2} h x} \right) \geq  c e^{-a x} e^{-\frac{g x^2}{2}}.
$$
Similarly, taking $I = [\mu_n,\infty)$ and $M = -x < 0$ gives, with $a$ large enough but fixed, 
$$
\Pro \left(L_{1,1}^{(n)} \geq  \mu_n e^{\frac{\gamma}{2} h x} \right) \geq  c e^{-a x} e^{-\frac{g x^2}{2}}
$$
for every $x > 0$, $n \geq 0$. This completes the proof.
\end{proof}

\section{Tightness of the metric at subcriticality: proof of Theorem \ref{th:Metric}}

\subsection{Diameter estimates}

We focus on the diameter of $[0,1]^2$ for the metric $e^{\gamma \phi_{0,n}} ds^2$. Notice that there may be a gap between it and the left-right length studied in the previous sections since left-right geodesics are between points where the field $\phi_{0,n}$ is small whereas geodesics associated to diameter have their extremities at points where the field $\phi_{0,n}$ may be high. Before going into exponential tail estimates, we start with a first moment estimate.
\begin{Prop}
\label{pro:diam}
If $\gamma < \min ( \gamma_c, 1/2)$ then $\left( \log \mathrm{Diam}\left([0,1]^2, \mu_{n}^{-2} e^{\gamma \phi_{0,n}} ds^2 \right) \right)_{n \geq 0}$ is tight.
\end{Prop}
\begin{proof} 

The proof is divided in four steps: in the first step we use a chaining argument to give an upper bound of the diameter in terms of crossing lengths of rectangles at lower scales and in term of the supremum of $\phi_{0,n}$. In the second and third steps, we bound the expected value of the term associated to the crossing lengths of rectangles and the one of term associated to the supremum. By Chebychev inequality, this gives a control of the right tail of $\log \mathrm{Diam}\left([0,1]^2, \mu_{n}^{-2} e^{\gamma \phi_{0,n}} ds^2 \right)$. In the last step,  we compare the diameter to the left-right crossing length to obtain a left tail estimate.

\medskip

\textbf{Step 1.} Let us denote by $H_k$ (resp $V_k$) the set of horizontal (resp vertical) thin rectangles of size $2^{-k-1} (2,1)$ spaced by $2^{-k-1}$ and tiling $[0,1]^2$. Each dyadic square of size $2^{-k}$ in $[0,1]^2$ is split in two thin horizontal rectangles in $H_k$ and two thin vertical rectangles in $V_k$. For each of these four rectangles, we pick a path minimizing the crossing length in the long direction. We call \textit{system} the union of these four geodesics (on Figure \ref{fig:chaining}, the purple and the green sets are systems associated to different squares). At a scale $k$, there are $4^k$ systems, each giving rise to four geodesics.

\smallskip

If $x$ and $y$ are two points in $[0,1]^2$, the geodesic distance between $x$ and $y$ is less than the length associated to any path between them. The majorizing path we use is defined as follows: if $P \in \mathcal{P}_n$ is the dyadic block at scale $n$ containing $x$, we take an Euclidean straight line (red path on Figure \ref{fig:chaining}) to join  the system of four geodesics (purple set on the Figure \ref{fig:chaining}) associated to $H_n$ and $V_n$ in the block $P$. By following successively systems associated to larger dyadic blocks, we eventually reach to the one associated to $[0,1]^2$. For instance, on Figure \ref{fig:chaining}, the path goes from scale $n$ to scale $n-1$ by using the purple and green systems. Proceeding similarly with $y$ gives a path from $x$ to $y$, constituted by $n$ systems and two Euclidean straight lines.Taking a uniform bound over these gives an upper bound which is uniform for every $x$ and $y$ in $[0,1]^2$,  hence a.s.
\begin{equation}
\label{eq:chaining}
\mathrm{Diam} \left( [0,1]^2 , e^{\gamma \phi_{0,n}} ds^2 \right) \leq 8 \sum_{k=0}^n  \underset{P \in H_k \cup V_k}{\max} L^{(n)}(P) + 2 \times 2^{-n} e^{\frac{\gamma}{2} \underset{[0,1]^2}{\sup} \phi_{0,n}}.
\end{equation}
\begin{figure}[h!]
\centering
\includegraphics[scale=0.9]{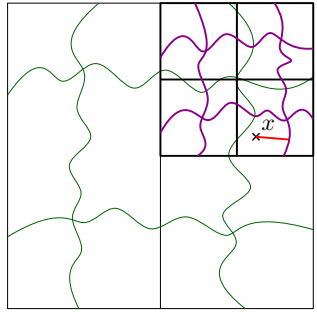}
\caption{Chaining argument}
\label{fig:chaining}
\end{figure}

\medskip

\textbf{Step 2.}  Now, we bound the expected value of the first term in \eqref{eq:chaining}.  We decouple the first scales, a.s. $\max_{P \in H_k \cup V_k} L^{(n)}(P) \leq e^{\frac{\gamma}{2} \sup_{[0,1]^2} \phi_{0,k-1}} \max_{P \in H_k \cup V_k} L^{(k,n)}(P)$ and  use independence, $\E ( \max_{P \in H_k \cup V_k} L^{(n)}(P) ) \leq \E ( e^{\frac{\gamma}{2} \sup_{[0,1]^2} \phi_{0,k-1}} ) \E( \max_{P \in H_k \cup V_k} L^{(k,n)}(P))$. Then, by using the bound on the exponential moment of the supremum of $\phi_{0,n}$ (Lemma \ref{ExpoMoment}), we get $ \E ( e^{\frac{\gamma}{2} \sup_{[0,1]^2} \phi_{0,k-1}} ) \leq 2^{\gamma k} e^{C \sqrt{k}}$. By scaling and union bound, the upper tails \eqref{QuasiGaussianUp} (since $\gamma < \gamma_c$) give the tail estimate $\Pro ( \max_{P \in H_k \cup V_k} L^{(k,n)}(P) \geq  2^{-k} \mu_{n-k} e^{s \sqrt{\log s}}) \leq C 4^k e^{-s^2}$ hence $\E( \max_{P \in H_k \cup V_k} L^{(k,n)}(P)) \leq  2^{-k} \mu_{n-k} e^{C \sqrt{k \log k}}$ by Lemma \eqref{AddLemma}. Gathering all the pieces leads to
$$
\E \left( \sum_{k=0}^n  \underset{P \in H_k \cup V_k}{\max} L^{(n)}(P)   \right)  \leq  C \sum_{k=0}^n 2^{-k} 2^{\gamma k} \mu_{n-k} e^{C \sqrt{k \log k}}.
$$
By the bound relating quantiles of different scales (Lemma \ref{BoundsScales}) we have
$$
\E \left( \sum_{k=0}^n  \underset{P \in H_k \cup V_k}{\max} L^{(n)}(P)) \right) \leq C \mu_n \sum_{k=0}^n 2^{-k} 2^{2\gamma k}  e^{C \sqrt{k \log k}}.
$$
The series converges for $\gamma < 1/2$. 

\medskip

\textbf{Step 3.} For the second term, using the exponential moment bound for the supremum (Lemma \ref{ExpoMoment}), the bound $2^{-\gamma n} e^{-C \sqrt{n}} \leq l_n$ for $\gamma < 1/2$ (by comparison with the supremum)  we find 
$$
\E \left(  2^{-n} e^{\frac{\gamma}{2} \underset{[0,1]^2}{\sup} \phi_{0,n}} \right) \leq 2^{-n} 2^{\gamma n} e^{C \sqrt{n}}  = 2^{-n} 2^{2\gamma n} e^{C \sqrt{n}}  2^{-\gamma n} e^{-C \sqrt{n}} \leq C l_n \leq C \mu_n.
$$

\medskip

\textbf{Step 4}. Since the diameter of the square $[0,1]^2$ is larger than the left-right distance, by using Theorem \ref{th:Tails} we get 
$$\Pro \left( \mathrm{Diam} ( [0,1]^2 , \mu_{n}^{-2} e^{\gamma \phi_{0,n}} ds^2 )   \leq e^{-s}  \right) \leq \Pro \left( L_{1,1}^{(n)} \leq \mu_n e^{-s} \right) \leq C e^{-cs^2}$$
which completes the proof of Proposition \ref{pro:diam}.
\end{proof} 
We now look for exponential tails, when $\gamma$ is small enough. The following proposition will be used both for the tightness of $d_{0,n}$ and to prove that $\gamma_c > 0$. We refer the reader to the definitions of $\delta_n$  and $l_n$ in Subsection \ref{DefLengths}.
\begin{Prop}
\label{SupTails} If $\eps$ is small enough, then for every $c > \frac{\gamma^2}{8(1-2 \gamma)}$ there exists $C > 0$ such that for every $n \geq 0$, $s > 0$: 
$$
\Pro \left( \mathrm{Diam}\left( [0,1]^2,  e^{\gamma \phi_{0,n} } ds^2 \right) \geq  \delta_n l_n e^{c s} \right) \leq  C e^{-s}.
$$
\end{Prop}
\begin{proof}

The proof is divided in three steps. In the two first steps, we give a tail estimate for the first term in \eqref{eq:chaining}. More precisely, in the first step, we give a tail estimate for $L^{(n)}(P)$ with $P \in H_k \cup V_k$. By union bound, we get one for $\sum_{k=0}^n \max_{P \in H_k \cup V_k} L^{(n)}(P) $ in the second step. The third step deals with the second term in \eqref{eq:chaining}.

\medskip

\textbf{Step 1.} In order to reuse directly the Proposition \ref{LemmeRecu}, note first if $P \in  H_k \cup V_k$ is fixed, we have a stochastic domination $L^{(n)}(P) \leq L_{2^{-k}(3,1)}^{(n)}$ (since any left-right crossing of $2^{-k}(3,1)$ is a crossing of $2^{-k}(2,1)$) thus we look for a tail estimate for this term. To this end, we decouple the scales by taking a geodesic $\pi_{k,n}$ for the left-right crossing of the rectangle $2^{-k}(1,3)$  for the field $\phi_{k,n}$ and we obtain
$$
L_{2^{-k}(3,1)}^{(n)} \leq \int_{\pi_{k,n}} e^{\frac{\gamma}{2} \phi_{0,k-1}} e^{\frac{\gamma}{2} \phi_{k,n}} ds.
$$
Therefore, we have the bound
\begin{multline*}
\Pro \left( L^{(n)}(P)  \geq  2^{-k} \delta_n l_{n-k}  e^{C s \sqrt{ \log s}} e^{\frac{\gamma}{2} s \sqrt{k \log 4 }}  \right) \\
 \leq \Pro \left( \int_{\pi_{k,n}} e^{\frac{\gamma}{2} \phi_{0,k-1}} e^{\frac{\gamma}{2} \phi_{k,n}} ds \geq 2^{-k} \delta_n l_{n-k}  e^{C s \sqrt{ \log s}} e^{\frac{\gamma}{2} s \sqrt{k \log 4 } }  \right).
\end{multline*}
By union bound, we have
\begin{multline*}
\Pro \left( \int_{\pi_{k,n}} e^{\frac{\gamma}{2} \phi_{0,k-1}} e^{\frac{\gamma}{2} \phi_{k,n}} ds \geq 2^{-k} \delta_n l_{n-k}  e^{C s \sqrt{ \log s}} e^{\frac{\gamma}{2} s \sqrt{k \log 4 } }  \right)  \\ \leq  \Pro \left( \int_{\pi_{k,n}} e^{\frac{\gamma}{2} \phi_{0,k-1}} e^{\frac{\gamma}{2} \phi_{k,n}} ds \geq L_{2^{-k}(3,1)}^{(k,n)}  e^{\frac{\gamma}{2} s \sqrt{k \log 4} }\right) + \Pro \left( L_{2^{-k}(3,1)}^{(k,n)} \geq 2^{-k} \delta_n l_{n-k}e^{C  s \sqrt{\log s}} \right).
\end{multline*}
Using Lemma \ref{MomentMethod} for the first term, scaling and the upper tail estimate from Proposition \ref{LemmeRecu} for the second term, we get
$$
\Pro \left( \int_{\pi_{k,n}} e^{\frac{\gamma}{2} \phi_{0,k-1}} e^{\frac{\gamma}{2} \phi_{k,n}} ds \geq L_{2^{-k}(3,1)}^{(k,n)}  e^{\frac{\gamma}{2} s \sqrt{k \log 4} }\right) + \Pro \left( L_{2^{-k}(3,1)}^{(k,n)} \geq 2^{-k} \delta_n l_{n-k}e^{C  s \sqrt{\log s}} \right)   \leq C  e^{-s^2}.
$$
Hence, we get for $P \in H_k \cup V_k$: 
\begin{equation}
\label{Step1Tail}
\Pro ( L^{(n)}(P) \geq 2^{-k} \delta_n l_{n-k}    e^{ C s \sqrt{\log s}} e^{\frac{\gamma}{2} s \sqrt{k \log 4} } )  \leq C e^{-s^2}.
\end{equation}

\medskip

\textbf{Step 2.}  In this step we want to give a tail estimate for $\sum_{k=0}^n M_{k}^{(n)}$ where $M_{k}^{(n)} := \max_{P \in H_k \cup V_k} L^{(n)}(P) $. By union bound ($| H_k \cup V_k| \leq C 4^k$) and by replacing $s$ in \eqref{Step1Tail} by $t(s) := \sqrt{k \log (4+\eps) + s^2}  $ so that the right-hand side in this inequality becomes $(4+\eps)^{-k} e^{-s^2}$, we get 
$$
\Pro \left( M_{k}^{(n)}  \geq  \delta_n 2^{-k}  l_{n-k}    e^{ C t(s) \sqrt{\log t(s)}} e^{\frac{\gamma}{2}  t(s) \sqrt{k \log 4} }  \right) \leq C \frac{4^k}{(4 + \eps)^k} e^{-s^2}
$$

\medskip

Since $\log s \leq C s^{2 \delta}$ for some small fixed $\delta > 0$,  $t(s) \sqrt{\log t(s)} \leq C t(s)^{1+\delta} $. Moreover, since we have $t(s) \leq \sqrt{k \log (4+\eps)} + s$, the convexity of the map $s \mapsto s^{1+ \delta}$ gives the bound $C t(s) \sqrt{\log t(s)}   \leq C k^{1/2 + \delta/2} + C s^{1+\delta}$. 

\medskip

Using that $\sqrt{a+b} \leq \sqrt{a} + \sqrt{b}$ for $a,b >0$, we have 
$$
t(s) \sqrt{k \log 4}  = \sqrt{k^2 \log(4+\eps) \log 4 + s^2 k \log 4} \leq  a_{\eps} k \log 4 + s \sqrt{k \log 4}
$$
by introducing $a_{\eps} := \sqrt{\log(4+\eps)/\log 4}$. Therefore, we have $e^{\frac{\gamma}{2} t(s) \sqrt{k \log 4}} \leq 2^{a_{\eps} \gamma k} e^{\frac{\gamma}{2} s \sqrt{k \log 4}}$ and by using the upper bound $l_{n-k}  \leq l_n 2^{\gamma k} e^{C \sqrt{k}}$  (Lemma \ref{BoundsScales}), we get the bound
\begin{align*}
2^{-k}  l_{n-k}    e^{ C t(s) \sqrt{\log t(s)}} e^{\frac{\gamma}{2}  t(s) \sqrt{k \log 4} } &  \leq 2^{-k}  ( l_n 2^{\gamma k} e^{C \sqrt{k}} ) ( e^{C k^{1/2 + \delta/2} + C s^{1+\delta}}) ( 2^{a_{\eps} \gamma k} e^{\frac{\gamma}{2} s \sqrt{k \log 4}}) \\
& \leq  l_{n} 2^{-k}  2^{(1+a_{\eps}) \gamma k} e^{C k^{1/2+ \delta/2}  } e^{ C s^{1+\delta}} e^{\frac{\gamma}{2} s \sqrt{k \log 4} }
\end{align*}
which leads to the following tail estimate:
$$
\Pro \left( M_k^{(n)}  \geq  \delta_n l_{n} 2^{-k}  2^{(1+a_{\eps}) \gamma k} e^{C k^{1/2+ \delta/2}  } e^{ C s^{1+\delta}} e^{\frac{\gamma}{2} s \sqrt{k \log 4} } \right) \leq C \frac{4^k}{(4 + \eps)^k} e^{-s^2}.
$$

\medskip

We now introduce $F(s) := \sum_{k=0}^{\infty} 2^{-k}  2^{\lambda k} e^{C k^{1/2+\alpha}} e^{ \beta s \sqrt{k}}$, where $\lambda := (1+a_{\eps}) \gamma$, $\alpha := \frac{\delta}{2}$ and $\beta := \frac{\gamma}{2} \sqrt{\log 4} $. We obtain by union bound, $\Pro ( \sum_{k=0}^n  M_k^{(n)} \geq  \delta_n l_n e^{C s^{1+\delta}} F(s))  \leq C \eps^{-1} e^{-s^2}$.

\medskip

We thus want an upper bound on $F(s)$. To this end, we introduce the function $f_s(t) := - t(1-\lambda) \log 2 + C t^{1/2+ \alpha} +  \beta s \sqrt{t}$. We notice that $f$ increases on $[0,t_s]$ and decreases on $[t_s, \infty]$ for some $t_s > 0$. By series/integral comparison we have: 
$$
\sum_{k = 0}^{\infty} a_k = \sum_{k = 0}^{[t_s]-1}a_k + a_{[t_s]} + a_{[t_s]+1} + \sum_{k = [t_s]+2}^{\infty} a_k \leq \int_{0}^{[t_s]} a_t dt + 2 a_{t_s} + \int_{[t_s]+1}^{\infty} a_t dt  \leq 2 a_{t_s} + \int_{0}^{\infty} a_t dt,
$$
where $a_k := \exp (f_s(k))$. 

\medskip

By introducing $c_{\eps} := \frac{\gamma^2}{8 (1 - (1 + a_{\eps}) \gamma)}$, we obtain $F(s) = \sum_{k=0}^{\infty} a_k \leq C e^{c_{\eps} s^2} e^{C s^{1+\delta}}$, see the appendix, Subsection \ref{UpperF(s)} for more details. Thus $\Pro (\sum_{k=0}^n  M_{k}^{(n)} \geq  \delta_n l_n e^{c_{\eps} s^2}   e^{C s^{1+\delta}})  \leq C e^{-s^2}$. Notice that when $\eps \to 0$, $c_{\eps} = \frac{\gamma^2}{8 (1 - (1 + a_{\eps}) \gamma)} \to \frac{\gamma^2}{8(1-2 \gamma)}$ which is less than $1$ if and only if $\gamma < 6 \sqrt{2} - 8 \approx 0.485$. 

\medskip

\textbf{Step 3.} Now, we focus on the second term in the chaining inequality \eqref{eq:chaining}. Since $ l_n \geq 2^{-\gamma n} e^{-C \sqrt{n}}$ (Lemma \ref{BoundsScales}), we have for $\gamma < 1/2$ and using the tail estimates obtained in Lemma \ref{TailSup}:
$$
\Pro \left( 2^{-n} e^{\frac{\gamma}{2} \sup_{[0,1]^2} \abs{ \phi_{0,n}}} \geq l_n e^{\frac{\gamma}{2} s} \right)   \leq \Pro \left(  e^{\frac{\gamma}{2} \sup_{[0,1]^2} \abs{ \phi_{0,n}}} \geq 2^{\gamma n} e^{C \sqrt{n}} e^{\frac{\gamma}{2} s} \right) \leq C e^{-s}
$$
which concludes the proof.
\end{proof} 
\subsection{Tightness of the metric}
We are ready to prove Theorem \ref{th:Metric} i.e. the tightness of the metric when $\gamma < \gamma_c \wedge 0.4$.
\begin{proof}[Proof of Theorem \ref{th:Metric}]

The proof is divided in two main steps. In the first one, we prove the tightness of the metric in the space of continuous functions by giving a H\"older upper bound. In the second one we prove that the pseudo-metric obtained is a metric. This is done by establishing a H\"older lower bound.

\medskip

\textbf{Step 1.} We suppose $\gamma < \gamma_c$. We start by proving that for every  $0 < h < 1 - 2 \gamma - \frac{\gamma^2}{4(1-2 \gamma)}$,  if $\eps > 0$ there exists a large $C_{\eps} > 0$ so that for every $n \geq 0$
\begin{equation}
\label{UpperHolder}
\Pro \left( \exists x,x' \in [0,1]^2 ~ : ~ d_{0,n}(x,x') \geq C_{\eps} \norme{x-x'}^{h} \right) \leq \eps.
\end{equation} 
By union bound we will estimate $\Pro (  \exists x,x'   \norme{x-x'}  <  2^{-n}  ,   d_{0,n}(x,x') \geq  e^{s} \norme{x-x'}^{h} ) $ and
$$ \sum_{k=0}^{n} \Pro \left( \exists x,x' : 2^{-k} \leq \norme{x-x'} \leq 2^{-k+1},  d_{0,n}(x,x') \geq e^{s} \norme{x-x'}^h  \right).
$$

\smallskip

We start with the term $\Pro ( \exists x,x' :  2^{-k} \leq \norme{x-x'} \leq 2^{-k+1},  d_{0,n}(x,x') \geq e^{s} \norme{x-x'}^h )$. Note that if $2^{-k-1} \leq \norme{x-x'} \leq 2^{-k}$, there exists a square $P$ of size $2^{-k+2}$ among fewer than $C 4^k$ fixed such squares such that $x,x' \in P$. Also, for two such $x$ and $x'$, by writing $h = 1 - 2 \gamma - c(\gamma) - \delta$ with $c(\gamma) > \frac{\gamma^2}{4(1-2\gamma)}$, $\delta > 0$ we have $\norme{x-x'}^h \geq 2^{-k} 2^{2 \gamma k} 2^{c(\gamma) k} 2^{\delta k}$. Hence, by union bound, this term is bounded by 
$$
C 4^k \Pro \left( \mathrm{Diam} \left( P , d_{0,n} \right) \geq 2^{-k} 2^{2 \gamma k} 2^{c(\gamma) k} 2^{\delta k} e^s \right).
$$
We separate the first $k$ scales of the fields $\phi_{0,n}$ as follows. Recall that  $\mathrm{Diam} (P,e^{\gamma \phi_{0,n}} ds^2)$ is larger than $e^{\frac{\gamma}{2} \sqrt{k} t} \mathrm{Diam}(P,e^{\gamma \phi_{k,n}} ds^2)$ with probability less than $e^{-\frac{t^2}{\log 4}}$ (by  Lemma \ref{MomentMethod}). By taking $t = \sqrt{k} \log 4 + \delta \sqrt{k} + s /\sqrt{k}$, this event has probability less than $4^{-k} e^{-c k} e^{-2s}$. On the complementary event, $ \mu_n^{-1} \mathrm{Diam}(P,e^{\gamma \phi_{0,n}} ds^2) $ is less than $\mu_n^{-1} \mathrm{Diam}(P,e^{\gamma \phi_{k,n}} ds^2) 2^{\gamma k} 2^{\frac{\gamma}{2} \delta k} e^{\frac{\gamma}{2} s}$. Under this event, by scaling the former bound becomes 
$$
C 4^k \Pro \left( \mathrm{Diam} \left( [0,1]^2 , d_{n-k} \right) \geq  \mu_{n-k}^{-1} \mu_n 2^{\gamma k} 2^{c(\gamma) k} 2^{(1-\frac{\gamma}{2}) \delta k} e^{(1-\frac{\gamma}{2})s}\right).
$$
Using Lemma \ref{BoundsScales} we get that $\mu_n \geq \mu_{n-k} 2^{- \gamma k} e^{-C \sqrt{k}} $ thus we are left with estimating 
$$
C 4^k \Pro \left( \mathrm{Diam} \left( [0,1]^2 , d_{n-k} \right) \geq  2^{c(\gamma) k} 2^{(1-\frac{\gamma}{2}) \delta k}  e^{-C \sqrt{k}} e^{(1-\frac{\gamma}{2})s}\right).
$$
We use the diameter estimates obtained in Proposition \ref{SupTails}: since $2^{c(\gamma) k} =e^{\frac{1}{2} c(\gamma) k \log 4}$ and $\frac{1}{2} c(\gamma) > \frac{\gamma^2}{8(1-2\gamma)}$, taking $\tilde{s}(k,s) = k \log 4 + \delta' k - C \sqrt{k} + c (1-\gamma/2)s$, we have by gathering all the pieces for $s$ large enough, uniformly in $n$:
$$
 \sum_{k=0}^{n} \Pro \left( \exists x,x' : 2^{-k} \leq \norme{x-x'} \leq 2^{-k+1},  d_{0,n}(x,x') \geq e^{s} \norme{x-x'}^h  \right)  \leq C e^{-cs}.
$$
Taking $s$ large enough, the right-hand side is less than $\eps$. 

\medskip

We are left with the term $ \Pro  (  \exists x,x'   \norme{x-x'}  <  2^{-n}  ,   d_{0,n}(x,x') \geq  e^{s} \norme{x-x'}^{h}  )$ i.e. with the case of small dyadic blocks where the field is approximately constant. By direct comparison with the supremum of the field i.e. $d_{0,n}(x,x') \leq \mu_n^{-1} e^{\frac{\gamma}{2} \sup_{[0,1]^2} \phi_{0,n}} \norme{x-x'}$ and since on the associated event $\norme{x-x'}^{h-1} \geq 2^{n(1-h)}$, this probability is less than the probability $\Pro ( e^{\frac{\gamma}{2} \sup_{[0,1]^2} \abs{\phi_{0,n}}} \geq e^s 2^{n(1 - h)} \mu_n )$. Recalling that one can write $h = 1 - 2 \gamma - c(\gamma)$ with $c(\gamma) > \frac{\gamma^2}{4(1-2\gamma)}$ and that we have the lower bound on the median $\mu_n \geq 2^{-\gamma n} e^{-C \sqrt{n}}$ (see the proof of Proposition \ref{LemmeRecu}, Step 6) the former probability is less than 
$$
\Pro \left( \sup_{[0,1]^2} \phi_{0,n}  \geq  n \log 4 + \frac{\gamma}{4(1-2\gamma)} n \log 4 - \frac{C}{\gamma} \sqrt{n} + s \right)
$$
which goes uniformly (in $n$) to $0$ as $s$ goes to infinity according to Lemma \ref{SupTailsAppendix}. Altogether we get the intermediate result \eqref{UpperHolder}. One can check that the interval $(0,1 - 2 \gamma - \frac{\gamma^2}{4(1-2 \gamma)})$  is nonempty if and only if $0 < \gamma < 2/5 = 0.4$.  

\medskip

Hence we obtain the tightness of $\left( d_{0,n} \right)_{n \geq 0}$ as a random element of $C([0,1]^2 \times [0,1]^2, \R^{+})$ and every subsequential limit is (by Skorohod's representation theorem) a pseudo-metric. 

\medskip

\textbf{Step 2.} Now we deal with the separation of the pseudo-metric. We prove that if $h>  1 + \gamma $ and  if $\eps > 0$ there exists a small constant $c_{\eps}$ such that for every $n \geq 0$
\begin{equation}
\label{LowerHolder}
\Pro \left( \exists x,x' \in [0,1]^2 ~ : ~ d_{0,n}(x,x') \leq c_{\eps} \norme{x-x'}^{h} \right) \leq \eps.
\end{equation}
Similarly as in the proof of \eqref{UpperHolder}, by union bound it is enough to estimate the term $\Pro (  \exists x,x'   \norme{x-x'}  <  2^{-n}  ,   d_{0,n}(x,x') \leq e^{-s} \norme{x-x'}^{h}  ) $ and  the term
$$
\sum_{k=0}^{n} \Pro \left( \exists x,x' : 2^{-k} \leq \norme{x-x'} \leq 2^{-k+1},  d_{0,n}(x,x') \leq e^{-s} \norme{x-x'}^{h}  \right).
$$

We start with $\Pro ( \exists x,x' : 2^{-k} \leq \norme{x-x'} \leq 2^{-k+1},  d_{0,n}(x,x') \leq e^{-s} \norme{x-x'}^{h}  )$. Assume there exists $ x,x' \in [0,1]^2$ such that $2^{-k} \leq \norme{x-x'} \leq 2^{-k+1}$. Note that any path from $x$ to $x'$ crosses one of the fixed $C 4^k$ rectangles of size $2^{-k-1}(1,3)$ that fill vertically and horizontally $[0,1]^2$. Hence $d_{0,n}(x,x') \geq \mu_n^{-1} \underset{C 4^k}{\min}  L_{2^{-k-1}(1,3)}^{(n)}$. By writing $h = 1 +  \gamma + \delta$ with $\delta > 0$, we can bound the term in the summation above by 
$$
\Pro \left ( e^{\frac{\gamma}{2} \inf_{[0,1]^2} \phi_{0,k-1}} \underset{C 4^k}{\min}  L_{2^{-k-1}(1,3)}^{(k,n)} \leq \mu_n 2^{-k} 2^{-\gamma k} 2^{-\delta k}  e^{-s}   \right).
$$
By separating the infimum with the term $\Pro \left(  \sup_{[0,1]^2} \phi_{0,n} \geq k \log 4 + \delta' k + s  \right) $, by scaling and using the bound $\mu_n \leq l_{n-k}  e^{C \sqrt{k}}$ from Lemma \ref{BoundsScales}, what is left is 
$$
\Pro \left ( \underset{C 4^k}{\min}  L_{(1,3)}^{(n-k)} \leq l_{n-k}   2^{-\delta'' k}  e^{-(1-\frac{\gamma}{2})s}   \right). 
$$
By union bound, the tail estimates from Corollary  \ref{LowerTailThin} and gathering all the pieces we get that the summation is less than $C e^{- cs}$ uniformly in $n$. 

\medskip

Finally, we control again the second term by comparison with the supremum of the field. On the event $\lbrace  \exists x,x'   \norme{x-x'}  <  2^{-n}  ,   d_{0,n}(x,x') \leq e^{- \frac{\gamma}{2}s} \norme{x-x'}^{h} \rbrace$, note that $\exp ( \frac{\gamma}{2} \inf_{[0,1]^2} \phi_{0,n}) \leq 2^{-n(h-1)} e^{-\frac{\gamma}{2} s} \leq 2^{-(\gamma+\delta) n} e^{-\frac{\gamma}{2} s}$.  The probability of this event is less than $\Pro ( \sup_{[0,1]^2} \phi_{0,n} \geq n \log 4 + \delta' n + s ) $ hence the result as before.
\end{proof}

\textbf{Definition of a metric on $\R^2$.}  Let us mention here that one can define a random metric associated to $\phi_{0,\infty}$ on the full two-dimensional space. We saw that $(d_{0,n}^{[0,1]^2})_{n \geq 0}$ is tight thus there exists some subsequence that converges in law to $d_{0,\infty}$. The same result remains true for $(d_{0,n}^{[-p,p]^2})_{n \geq 0}$ with $p > 0$. By a diagonal argument, there exists a subsequence $(n_k)$ such that for every $p \in \N$, $(d_{0,n_k}^{[-p,p]^2})_{ k \geq 0}$ converges in law to some $d_{0,\infty}^{[-p,p]^2}$. Then, one can define $d_{0,\infty}^{\R^2}$ as the limit of $d_{0,\infty}^{[-p,p]^2}$ when $p$ goes to $\infty$. Indeed, if we denote by $d_{0,\infty}^{[-p,p]^2}([-1,1]^2)$ the restriction of $d_{0,\infty}^{[-p,p]^2}$ to $[-1,1]^2$, we have
$$
\lim_{p_0 \to \infty} \Pro \left( \forall p \geq p_0, ~ d_{0,\infty}^{[-p,p]^2}([-1,1]^2)  = d_{0,\infty}^{[-p_0,p_0]^2}([-1,1]^2)   \right)  =1.
$$
Indeed, with high probability, there is a crossing of an annulus around $[0,1]^2$ whose length for $d_{0,n}$ is larger than the diameter of $[0,1]^2$ for $d_{0,n}$, uniformly in $n$.  Also, if we fix $x\in \R^2$ and denote by $T_{x}$ the map $\phi \mapsto \phi( \cdot -  x)$, for a field $\phi$ and $d \mapsto d(\cdot - x, \cdot - x)$ for a metric $d$, if the measure on fields is $\phi_{0,\infty}$ and the measure on metrics is $d_{0,\infty}^{\R^2}$, then the transformation $T_x$ is mixing thus ergodic in each case. This ergodic property for the Gaussian multiplicative chaos measure is a useful property to characterize $\log$-normal $\star$-scale invariant random measures. We refer the interested reader to Theorem 4  and the remark following Proposition 5 in \cite{Allez}.

\section{Weyl scaling}

\label{sec:WeylScaling}

In this section we will see that any limiting metric space is non trivial. In particular, we will show they are not deterministic and not independent of field $\phi_{0,\infty}$. 

\smallskip

The main idea of the proof is the following. Take $d_{0,\infty}$ a limiting metric whose existence comes from the previous subsection. Define for some suitable function $f$ the metric $e^{ \frac{\gamma}{2} f} \cdot d_{0,\infty}$ associated to the field $\phi_{0,\infty} + f$. Thanks to the approximation procedure together with the Cameron-Martin theorem for Gaussian measures, we will prove that the couplings $P_{\infty} := \mathcal{L}(\phi_{0,\infty},d_{0,\infty})$ and $P_{\infty}^f := \mathcal{L}(\phi_{0,\infty}+f, e^{\frac{\gamma}{2} f} \cdot d_{0,\infty})$ are mutually absolutely continuous and that the associated Radon-Nikod\'ym derivative satisfies $\frac{d P_{\infty}^f}{d P_{\infty}} = \frac{d \mathcal{L} (\phi_{0,\infty} + f)}{d \mathcal{L} \phi_{0,\infty} }$, which implies the result we look for: if $\phi_{0,\infty}$ and $d_{0,\infty}$ are independent, it implies $e^{ \frac{\gamma}{2} f} \cdot d_{0,\infty} \overset{(d)}{=} d_{0,\infty}$ which leads to a contradiction.

\smallskip

In what follows, we recall some background on metric geometry and we refer the reader to Chapter 2 in \cite{BBI} for more details. Let $(X,d)$ be a metric space and $\pi$ be a continuous map from an interval $I$ to $X$. We define the length $L_d(\pi)$ of $\pi$ with respect to the metric $d$ by setting 
$$
L_d(\pi) := \sup \sum_{i=1}^n d(\pi(t_{i-1}), \pi(t_{i}))
$$
where the supremum is taken over all $n \geq 1$, $t_0 < t_1 < \dots < t_n$ in $I$. If $L_d(\pi) < \infty$, we say that $\pi$ is \textit{rectifiable}. We also say that $\pi$ has \textit{constant speed} if there exists a constant $\lambda \geq 0$ such that $L_d(\pi_{|_{[s,t]}})  = \lambda \abs{t-s}$ holds for every $s,t \in I$. 

\smallskip

Starting with such a length functional $L = L_d$ we can define a metric space $(X, d_L)$ by setting, for every $x,y \in X$,
$$
d_L (x, y) := \inf \lbrace L(\pi) ~ | ~ \pi \text{ is rectifiable },  \pi(0) = x \text{ and } \pi(1) = y \rbrace.
$$ 
We say that a metric $d$ is \textit{intrinsic} if $d = d_{L_d}$. In this case, $(X, d)$ is called a \textit{length space}. Notice that a Riemannian manifold $(M,g)$ is a length space.  Moreover, we say that this metric is \textit{strictly intrinsic} if for any $x,y \in X$ there exists a path $\pi$ such that  $\pi(0) = x$, $\pi(1) = y$ and $d(x,y) = L_d(\pi)$. In this case the path $\pi$ is called a \textit{shortest path} between $x$ and $y$. 

\smallskip

Let $(X,d)$ be a metric space. A path ($\pi, I$) is called a \textit{geodesic} if $\pi$ has constant speed and if $L_d(\pi_{|_{[s,t]}})= d(\pi(s),\pi(t))$ for every $s,t \in I$. A path ($\pi, I$) is called a \textit{local geodesic} if for every $t \in I$, there exists an $\eps > 0$ such that $\pi_{|_{[t-\eps,t+\eps]}}$ is a geodesic. $(X,d)$ is a \textit{geodesic space} if for every $x,y \in X$, there exists a geodesic $\pi : [0,1] \to X$ with $\pi(0) = x$, $\pi(1) = y$. It is clear from the definition that every geodesic space is a length space. 

\smallskip

For a complete metric space, one can characterize the notion of intrinsic metric using midpoints (see Lemma 2.4.8 and Theorem 2.4.16 in \cite{BBI} for a reference). A point $z \in (X,d)$ is called a \textit{midpoint} between points $x$ and $y$ if $d(x,z) = d(z,y) = \frac{1}{2} d(x,y)$. The following holds: 
\begin{enumerate} 
\item
Assume that $(X,d)$ is a metric space. If $d$ is a strictly intrinsic metric, then for every points $x$ and $y$ in $X$ there exists a midpoint $z$ between them.

\item
If $(X, d)$ is a complete metric space and if for every $x,y \in X$ there exists a midpoint $z$ between $x$ and $y$, then $d$ is strictly intrinsic.
\end{enumerate}

\smallskip

Given a continuous  function $f$ and an intrinsic metric $d$, both defined on $[0,1]^2$, with $d$ homeomorphic to the Euclidean metric on the unit square, we define the metric $e^{ f} \cdot d$ by first describing its length. For a continuous path $\pi : [a,b] \to [0,1]^2$ we define
$$
L_d^f (\pi) := \limsup_{n \to \infty} \sum_{i=1}^n e^{ f(\pi(t_{i-1}^n))} d(\pi(t_{i-1}^n), \pi(t_{i}^n)),
$$
where $a= t_0^n < \dots < t_n^n = b$ and $ \lim_{n \to \infty} \max_{0 \leq i \leq n-1} (t_{i+1}^n - t_i^n) = 0$. Notice that $L_d(\pi) < \infty$ if and only if $L_d^f(\pi) < \infty$. We then define $e^{f} \cdot d := d_{L_d^f}$. Notice that if $f$ is constant since $d$ is intrinsic we have $e^{f} \cdot d = e^f d$. Notice also that if $\phi$ and $\psi$ are smooth functions, then the Riemannian metric associated to the metric tensor $e^{\phi + \psi} ds^2$ is equal to $e^{\frac{1}{2} \phi} \cdot d$ where $d$ is the metric associated to the metric tensor $e^{\psi} ds^2$. 

\smallskip

The following lemma will be useful to identify the metric associated to $\phi_{0,\infty} + f$ in terms of the one associated to $\phi_{0,\infty}$. 

\begin{Lemma}
\label{StabMetric}
Let $f$ be a continuous function on $[0,1]^2$ and $r,R : (0,\infty) \to (0,\infty)$ be continuous increasing functions with $r(0^+)=R(0^+)=0$. If a sequence of intrinsic metrics $(d_{n})_{n \geq 0}$ on $[0,1]^2$ satisfying  for every $x,y \in [0,1]^2$, $n \geq 0$ the condition
$$
r(\norme{x-y}) \leq d_n(x,y) \leq R(\norme{x-y}),
$$ 
converges uniformly to a metric $d_{\infty}$ on $[0,1]^2$, then the sequence of metrics $(e^f \cdot d_n)_{n \geq 0}$ converges simply to the metric $e^f \cdot d_{\infty}$ i.e. for every fixed $x,y \in [0,1]^2$ we have $ \lim_{n \to \infty } e^{f} \cdot d_n(x,y) = e^f \cdot d_{\infty}(x,y)$.
\end{Lemma}

\begin{proof}
We fix $x,y \in [0,1]^2$ and we want to prove that $e^f \cdot d_n(x,y)$ converges to $e^f \cdot d_{\infty}(x,y)$. We separate the proof in three parts: first we control the oscillation of $f$ over geodesics then  the upper bound and finally the lower bound. 

By assumption, $d_n$ converges uniformly to $d_{\infty}$ hence $d_{\infty}$ is an intrinsic metric (see Exercise 2.4.19 in \cite{BBI}). Again by assumption, there exists some positive $c$ and $C$ such that for every $n$ 
$$
r(\norme{x-y}) \leq d_n(x,y) \leq R(\norme{x-y}).
$$
This condition is then satisfied by $d_{\infty}$ and since for $n \in \N \cup \lbrace  \infty \rbrace$, $ e^{-\norme{f}_{\infty}} d_n  \leq e^f \cdot d_n \leq e^{\norme{f}_{\infty}} d_n$ this condition is also satisfied by $e^f \cdot d_n$ and $e^f \cdot d_{\infty}$ by replacing $c$ by $e^{- \norme{f}_{\infty}} c$ and $C$ by $e^{\norme{f}_{\infty}} C$. This tells us that the spaces $([0,1]^2,d_n)$ and $([0,1]^2,e^{f} \cdot d_n )$ are complete and locally compact for $n \in \N \cup \lbrace \infty \rbrace$. Hence, by Theorem 2.5.23 in \cite{BBI}, these spaces are strictly intrinsic. 

Now we look at the oscillation of $f$ over small parts of shortest path associated to the metrics $e^f \cdot d_n$ and $d_n$ for all $n$'s. The first step is to understand that locally $e^{f(x)} d_n(x,y) \approx e^{f} \cdot d_n (x,y$). To this end notice the inequality 
$$
e^{-\mathrm{osc} (f,K_{x,y}^{d_n} )} e^{f(x)} d_n(x,y) \leq e^f \cdot d_n(x,y) \leq e^{\mathrm{osc}(f,K_{x,y}^{d_n})} e^{f(x)} d_n(x,y)
$$
where $\mathrm{osc}(f,K) := \sup_{x,y \in K} \abs{f(x)-f(y)}$ and where $K_{x,y}^{d_n} := \mathrm{Geo}_{d_n} ( x, y) \cup  \mathrm{Geo}_{e^f \cdot d_n} ( x, y)$. Then notice that if $x$ is close to $y$ then $K_{x,y}^{d_n} $ is small with respect to the Euclidean topology. More precisely, notice that $\mathrm{Geo}_{d_n} ( x, y) \subset B (x, r^{-1}(R(\norme{x-y})))$. Indeed, if $z \in \mathrm{Geo}_{d_n} ( x, y)$ then
$$
r(\norme{x-z}) \leq d_n(x,z) \leq d_n(x,y) \leq R(\norme{x-y}).
$$
For every $x$ and $y$ such that $d_n(x,y) < \delta$,  $\mathrm{osc} ( f , K_{x,y}^{d_n} ) \leq \omega (f, r^{-1}(\delta) )$ where $\omega(f,\delta)$ denotes the modulus of continuity of the function $f$ i.e. $\omega(f,\delta) := \sup \lbrace \abs{f(x) - f(y)} : x,y \in [0,1]^2  \text{ st } \abs{x-y} < \delta \rbrace $. Note that the bound of the oscillation is independent of $n$. 

We start with the upper bound. Since $e^f \cdot d_{\infty}$ is strictly intrinsic, take by a dichotomy procedure $x = x_0, \dots , x_N = y$ such that $e^f \cdot d_{\infty} (x,y) = \sum_{i=0}^{n-1} e^f \cdot d_{\infty}(x_i,x_{i+1})$ and $d_{\infty}(x_{i} , x_{i+1} ) < \delta$. For $n$ large enough, for every $i$, $d_n(x_i,x_{i+1}) < \delta$. Hence, by triangle inequality, for $n$ large enough
\begin{align*}
e^f \cdot d_n(x,y) & \leq \sum_{i=0}^{N-1} e^f \cdot d_n(x_i,x_{i+1}) \\
&  \leq  \sum_{i=0}^{N-1} e^{\mathrm{osc} (f,K_{x_i,x_{i+1}}^{d_n} )} e^{f(x_i)} d_n(x_i,x_{i+1})  \\
& \leq e^{\omega(f,C \delta^{1/\alpha})} \sum_{i=0}^{N-1}  e^{f(x_i)} d_n(x_i,x_{i+1}).
 \end{align*} Hence by taking the $\limsup$ and using the convergence of $d_n$ to $d_{\infty}$
\begin{align*}
\limsup_{n \to \infty}  e^f \cdot d_n(x,y) & \leq e^{\omega(f,C \delta^{1/\alpha})} \sum_{i=0}^{N-1}  e^{f(x_i)} d_{\infty}(x_i,x_{i+1}) \\
& \leq e^{\omega(f,C \delta^{1/\alpha})} \sum_{i=0}^{N-1} e^{\mathrm{osc} (f,K_{x_i,x_{i+1}}^{d_{\infty}} )}  e^{f} \cdot d_{\infty}(x_i,x_{i+1}) \\
& \leq e^{2 \omega(f,C \delta^{1/\alpha})} \sum_{i=0}^{N-1}  e^{f}  \cdot d_{\infty}(x_i,x_{i+1}) \\
& =  e^{ 2\omega(f,C \delta^{1/\alpha})}  e^{f} \cdot d_{\infty} (x,y).
\end{align*}
By the uniform continuity of $f$, we obtain the upper bound by letting $\delta$ going to $0$.

Now we deal with the lower bound. Up to extracting a subsequence we may assume that $e^f \cdot d_n (x,y)$ converges to its $\liminf$. Again, since $e^f \cdot d_n$ is strictly intrinsic, take $x_0^n = x, \dots , x_{N_n}^n = y$, such that 
$$
e^f \cdot d_n(x,y) = \sum_{i=0}^{N_n-1} e^f \cdot d_n(x_{i}^n,x_{i+1}^n)
$$ 
and $d_n(x_{i}^n,x_{i+1}^n) < \delta$. Taking the minimal number $N_{n}$ (still using the midpoints method) $N_n$ is bounded and up to taking a subsequence, we may assume that  $N_n$ converges. In particular, $N_n$ is eventually constant and equal to some $N$. We may then also assume that the $x_{i}^n$'s  also converges to some $x_{i}$'s for $0 \leq i \leq N$ and these $x_i$'s satisfy $d_{\infty}(x_{i},x_{i+1}) \leq \delta$. Then for $n$ large enough
\begin{align*}
e^{f} \cdot d_n (x,y) & \geq \sum_{i=0}^{N-1} e^{-\mathrm{osc} \left(f,K_{x_i^n,x_{i+1}^n}^{d_n} \right)} e^{f(x_i^n)} \cdot d_n (x_{i}^n, x_{i+1}^n)  \\
& \geq e^{- \omega(f,C \delta^{1/\alpha})} \sum_{i=0}^{N-1} e^{f(x_i^n)} \cdot d_n (x_{i}^n, x_{i+1}^n).
\end{align*}
Taking the limit as $n$ goes to $\infty$ we get by the uniform convergence of $d_n$ to $d_{\infty}$
$$
\abs{\sum_{i=0}^{N-1} e^{f(x_i^n)} d_n(x_i^n,x_{i+1}^n) - \sum_{i=0}^{N-1} e^{f(x_i^n)} d_{\infty}(x_i^n,x_{i+1}^n)  } \leq N e^{\norme{f}_{\infty}} \norme{d_n-d_{\infty}}_{\infty} \to 0
$$
\begin{align*}
\liminf_{n \to \infty} e^f \cdot d_n(x,y) & \geq e^{- \omega(f,C \delta^{1/\alpha})} \sum_{i=0}^{N-1} e^{f(x_i)}  d_{\infty} (x_{i}, x_{i+1}) \\
& \geq e^{- \omega(f,C \delta^{1/\alpha})} \sum_{i=0}^{N-1} e^{- \mathrm{osc} (f,K_{x_i,x_{i+1}}^{d_{\infty}} )} e^{f} \cdot  d_{\infty} (x_{i}, x_{i+1})  \\
&  \geq e^{-2 \omega(f,C \delta^{1/\alpha})} \sum_{i=0}^{N-1} e^{f} \cdot d_{\infty} (x_{i}, x_{i+1}) \\
&  \geq  e^{- 2 \omega(f,C \delta^{1/\alpha})} e^f \cdot d_{\infty}(x,y)
\end{align*}
by the triangle inequality. Letting $\delta$ going to $0$ we get the result.
\end{proof}

It is easy to see that the same result holds if instead of $f$, we assume that a sequence of continuous functions $(f_n)_{n \geq 0}$ converges uniformly to $f$ on $[0,1]^2$, then under the same assumptions $( e^{f_n} \cdot d_{n})_{n \geq 0}$ converges simply to the metric $e^f \cdot d_{0,\infty}$. This lemma is a key ingredient to prove the following corollary.

\begin{Cor}
\label{CorConvLaw}
Let $(f_n)$ be a sequence of continuous real-valued functions defined on $[0,1]^2$ and converging uniformly to a function $f$. If $\gamma < \min (\gamma_c, 0.4)$ then the following statements hold:

\begin{enumerate}
\item
\label{test}
$(d_{0,n}, e^{ \frac{\gamma}{2} f_n} \cdot d_{0,n})_{n \geq 0}$ is tight. 

\item
If $(n_k)$ is a subsequence along which $(d_{0,n_k},e^{\frac{\gamma}{2} f_{n_k}} \cdot d_{0,n_k})_{k \geq 0}$ converges in law to some $(d_{0,\infty}, d_{0,\infty}')$ then $d_{0,\infty}' = e^{ \frac{\gamma}{2} f} \cdot d_{0,\infty}$. 

\item
In particular, $(\phi_{0,n_{k}},d_{0,n_k})_{k \geq 0}$ converges in law to a coupling $P_{\infty} := \mathcal{L}(\phi_{0,\infty},d_{0,\infty})$ and $(\phi_{0,n_k}+f_{n_k},  e^{ \frac{\gamma}{2} f_{n_k} } \cdot d_{0,n_k})_{k \geq 0}$ converges in law to a coupling $P_{\infty}^f := \mathcal{L} (\phi_{0,\infty}+f, e^{ \frac{\gamma}{2} f} \cdot d_{0,\infty})$, both couplings are probability measures on the same space.
\end{enumerate}

\end{Cor}
\begin{proof}

We start with the proof of (i). Since for $n \geq 0$, a.s.   $e^{-\frac{\gamma}{2} \sup_{n \geq 0} \norme{f_n}_{\infty}} d_{0,n} \leq e^{\frac{\gamma}{2} f_n} \cdot d_{0,n} \leq e^{\frac{\gamma}{2} \sup_{n \geq 0}  \norme{f_n}_{\infty}} d_{0,n}$, the argument giving the tightness of $(d_{0,n})_{n \geq 0}$ then extends to give the one of $( e^{\frac{\gamma}{2} f_n} \cdot d_{0,n} )_{n \geq 0}$, see the proof of Theorem \ref{th:Metric}.

\smallskip

We now prove (ii). We first fix $\alpha > 1 + \gamma$ and $\beta \in (0, 1- 2\gamma - \frac{\gamma^2}{4(1-2 \gamma)})$ and we then define $C_{\alpha}^n := \sup_{x,x' \in [0,1]^2} \frac{\norme{x-x'}^{\alpha}}{d_{0,n}(x,x')} $ and $C_{\beta}^n := \sup_{x,x' \in [0,1]^2} \frac{d_{0,n}(x,x')}{\norme{x-x'}^{\beta}}$. Using \eqref{LowerHolder}  and  \eqref{UpperHolder}, $(C_{\alpha}^{n})_{n \geq 0}$ and $(C_{\beta}^{n})_{n \geq 0}$ are tight. Since $(\phi_{0,n}, \phi_{0,n} + f_n, d_{0,n},e^{\frac{\gamma}{2} f_n} \cdot d_{0,n}, C_{\alpha}^n, C_{\beta}^n)_{n \geq 0}$ is tight,  up to extracting a subsequence, we can assume it converges in law. By the Skorohod representation theorem, we obtain an almost sure convergence on a same probability space and we denote by $d_{0,\infty}$ (resp $d_{0,\infty}'$) the limit of $d_{0,n}$ (resp $e^{\frac{\gamma}{2} f_n} \cdot d_{0,n}$). We can thus introduce the random constants $C_{\alpha} := \sup_{n \geq 0} C_{\alpha}^n < \infty$ and $C_{\beta} := \sup_{n \geq 0} C_{\beta}^n < \infty$. On this probability space, the following condition of Lemma \ref{StabMetric} is satisfied: a.s. for every $n \geq 0$, $x,x' \in [0,1]^2$,  
$$
\frac{\norme{x-x'}^{\alpha}}{C_{\alpha}} \leq \frac{\norme{x-x'}^{\alpha}}{C_{\alpha}^n} \leq d_{0,n} (x,x') \leq C_{\beta}^n \norme{x-x'}^{\beta} \leq C_{\beta} \norme{x-x'}^{\beta}.
$$
By using Lemma \ref{StabMetric}, we can identify the almost sure limit of $e^{\frac{\gamma}{2} f_n} \cdot d_{0,n}$: $d_{0,\infty}' = e^{\frac{\gamma}{2} f} \cdot d_{0,\infty}$. 
Finally, notice that (iii) follows from the previous proofs.
\end{proof}
The main result of this subsection is the following proposition. In order to state it, let us recall that the kernel of $\phi_{0,\infty}$ is given by $C_{0,\infty}(x,x') = \int_0^1 c(\frac{x-x'}{t}) \frac{dt}{t} = \int_0^1 k \ast k(\frac{x-x'}{t}) \frac{dt}{t}$ and let us make the following remark: the map $C_{0,\infty} : \mathcal{S}(\R^2) \to \mathcal{S}(\R^2)$ defined for $f\in \mathcal{S}(\R^2) $ by $C_{0,\infty} f : = C_{0,\infty} * f$  is a bijection. Indeed, notice that $\hat{C}_{0,\infty}(\xi) = \norme{\xi}^{-2} \int_{0}^{\norme{\xi}} u \hat{k}(u)^2 du$ (see the remark before \eqref{coupling} for a proof). In particular,  we have $\hat{C}_{0,\infty}(0) = \frac{\hat{k}(0)^2}{2} > 0$ (since $\hat{k}(0) = \int_{B(0,r_0)} k(x) dx$ with $k$ nonnegative and non-identically zero), and $\hat{C}_{0,\infty}(\xi) \sim_{\infty} \frac{1}{2 \pi \norme{\xi}^2}$. Thus, the equation $C_{0,\infty} * f = g$ admits the solution $f$ given by $f(x) = \frac{1}{(2\pi)^2} \int_{\R^2} \frac{\hat{g}(\xi)}{\hat{C}_{0,\infty}(\xi)} e^{i x \cdot \xi}$. In particular, if $f \in \mathcal{S}(\R^2)$, $C_{0,\infty}^{-1} f \in \mathcal{S}(\R^2)$ is well-defined.

\begin{Prop}
\label{WeylScaling}
For $f \in \mathcal{S}(\R^2)$, the coupling $P_{\infty}^f = \mathcal{L} (\phi_{0,\infty}+f, e^{ \frac{\gamma}{2} f} \cdot d_{0,\infty})$ is absolutely continuous with respect to  $P_{\infty} = \mathcal{L} (\phi_{0,\infty}, d_{0,\infty})$ and its Radon-Nikod\'ym derivative is given by 
$$
\frac{d P_{\infty}^f}{d P_{\infty}} = \frac{ d \mathcal{L} \left(\phi_{0,\infty} + f, e^{ \frac{\gamma}{2} f} \cdot d_{\infty} \right)}{d \mathcal{L} \left(\phi_{0,\infty}, d_{\infty} \right)} = \frac{d \mathcal{L} ( \phi_{0,\infty} + f)}{d \mathcal{L} (\phi_{0,\infty})} = \exp \left( \langle \phi_{0,\infty}, C_{0,\infty}^{-1} f \rangle - \frac{1}{2} \langle f, C_{0,\infty}^{-1} f \rangle \right)
$$
In particular, $d_{0,\infty}$ and $\phi_{0,\infty}$ are not independent.
\end{Prop}

To prove this proposition, we will use the following lemma, whose proof is postponed to the end of the section.

\begin{Lemma} 
\label{lem:intermediate} 
Fix $g \in \mathcal{S}(\R^2)$ and  define for $n \in \N \cup \lbrace \infty \rbrace$, $f_n:= C_{0,n} * g$. The following assertions hold:
\begin{enumerate}
\item
For every $n \in \N \cup \lbrace \infty \rbrace$, $\phi_{0,n} + f_n$ is absolutely continuous with respect to $\phi_{0,n}$ and \\ $\frac{d \mathcal{L}(\phi_{0,n}+f_n)}{d \mathcal{L}(\phi_{0,n})} = \exp (\langle \phi_{0,n}, g \rangle - \frac{1}{2} \langle f_n , g\rangle)$. 

\item 
$(f_n)_{n \geq 0}$ converges uniformly on $\R^2$ and in $L^2(\R^2)$ to $C_{0,\infty} * g$. 

\item
$\left(\phi_{0,n} \right)_{n \geq 0}$ converges in law to $\phi_{0,\infty}$ with respect to the weak topology on $\mathcal{S}'(\R^2)$.
\end{enumerate}
\end{Lemma}

\begin{proof}[Proof of Proposition \ref{WeylScaling}]
Take $f \in \mathcal{S}(\R^2)$, set $g := C_{0,\infty}^{-1} f \in \mathcal{S}(\R^2)$ and define $f_n := C_{0,n} * g$. By using Lemma \ref{lem:intermediate} assertion (i) for $n = \infty$ we have:
$$
D_{\infty}^f := \frac{d \mathcal{L} (\phi_{0,\infty} + f)}{d \mathcal{L} (\phi_{0,\infty})} = \exp \left( \langle \phi_{0,\infty}, g \rangle - \frac{1}{2} \langle f , g \rangle \right).
$$
Using again Lemma \ref{lem:intermediate} assertion (i) but for finite $n$ we have:
$$
\frac{d \mathcal{L} (\phi_{0,n} + f_n)}{d \mathcal{L} (\phi_{0,n})}  = \exp \left( \langle \phi_{0, n}, g \rangle - \frac{1}{2} \langle f_n , g \rangle \right).
$$

Now  we prove that $ \left( \phi_{0,\infty} + f, e^{ \frac{\gamma}{2} f} \cdot d_{0,\infty} \right) $ is absolutely continuous with respect to $\left( \phi_{0,\infty} , d_{0,\infty} \right) $  and that the Radon-Nikod\'ym derivative is given by $D_{\infty}^f$. By introducing the function $G$ which maps a smooth field $\phi$ to the Riemannian metric whose metric tensor is $e^{ \gamma \phi} ds^2$, we have, for every continuous and bounded functional $F$:
\begin{align*}
\E \left( F \left(\phi_{0,n} + f_n,e^{ \frac{\gamma}{2} f_n} \cdot d_{0,n} \right) \right)  & = \E \left( F(\phi_{0,n} + f_n, \mu_n^{-2} G(\phi_{0,n}+f_n) ) \right)    \\
& = \E \left( F \left(\phi_{0,n}, \mu_n^{-2} G(\phi_{0,n}) \right) \frac{d \mathcal{L} (\phi_{0,n} + f_n)}{d \mathcal{L} (\phi_{0,n})}\right)   \\
& =  \E \left( F \left(\phi_{0,n}, d_{0,n}  \right) \exp \left( \langle \phi_{0, n}, g \rangle - \frac{1}{2} \langle f_n , g \rangle \right)\right).
\end{align*}
Now we claim that the left-hand side converges to $\E ( F (\phi_{0,\infty} + f,e^{ \frac{\gamma}{2} f} \cdot d_{0,\infty}))$ and that the right-hand side converges to $\E ( F (\phi_{0,\infty}, d_{0,\infty}  ) D_{\infty}^f )$. 

\smallskip

The first claim follows from the convergence in law from Corollary \ref{CorConvLaw} since $(f_n)_{n \geq 0}$ converges uniformly on $[0,1]^2$ and in $L^2(\R^2)$ to $f$ by Lemma \ref{lem:intermediate} assertion (ii).

\smallskip

The second one comes from the convergence in law of $(\phi_{0,n},d_{0,n})_{n \geq 0}$ and from the convergence of $(f_n)_{n \geq 0}$ to $f$ in $L^2(\R^2)$ (Lemma \ref{lem:intermediate} assertion (ii)). To be precise, for $M > 0$ the map $(\phi,d) \mapsto F(\phi,d) \exp (\langle \phi, g \rangle)  \wedge M$ is continuous and bounded thus $$
\lim_{n \to \infty} \E \left( F(\phi_{0,n},d_{0,n}) \exp (\langle \phi_{0,n}, g \rangle) \wedge M \right) = \E \left( F(\phi_{0,\infty},d_{0,\infty}) \exp (\langle \phi_{0,\infty}, g \rangle) \wedge M \right).
$$
By the triangle inequality and since $F$ is bounded we have
\begin{align*}
& \abs{\E \left( F(\phi_{0,n} , d_{0,n} ) \exp( \langle \phi_{0,n} , g \rangle)  \right) - \E \left( F(\phi_{0,\infty} , d_{0,\infty} ) \exp( \langle \phi_{0,\infty} , g \rangle)  \right) } \\
& \leq \abs{ \E \left( F(\phi_{0,n},d_{0,n}) \exp (\langle \phi_{0,n}, g \rangle) \wedge M \right) - \E \left( F(\phi_{0,\infty},d_{0,\infty}) \exp (\langle \phi_{0,\infty}, g \rangle) \wedge M \right)} \\
& \hphantom{\leq} +  \abs{ \E \left( F(\phi_{0,\infty},d_{0,\infty}) \exp (\langle \phi_{0,\infty}, g \rangle) \wedge M \right) - \E \left( F(\phi_{0,\infty} , d_{0,\infty} ) \exp( \langle \phi_{0,\infty} , g \rangle)  \right) } \\
& \hphantom{\leq} +  C \E \left(  \exp(\langle \phi_{0,n} ,g \rangle) 1_{\exp(\langle \phi_{0,n} ,g \rangle) \geq M } \right).
\end{align*}
Taking the $\limsup$ when $n$ goes to infinity (the first term vanishes) and then letting $M$ goes to infinity (the second term vanishes by uniform integrability), we obtain the result since  $\limsup_{M \to \infty} \limsup_{n \to \infty} \E \left(  \exp(\langle \phi_{0,n} ,g \rangle) 1_{\exp(\langle \phi_{0,n} ,g \rangle) \geq M } \right) = 0$ (easy to check).
\end{proof}

Now, we come back to the proof of Lemma \ref{lem:intermediate}.

\begin{proof}[Proof of Lemma  \ref{lem:intermediate}] We will prove successively the assertions (i), (ii) and (iii).

\smallskip

(i). The proof follows from evaluating characteristic functionals. Define for $\phi \in \mathcal{S}(\R^2)$ the functional $F_{\varphi} : \mathcal{S}'(\R^2) \to \R^{+}$ such that $F_{\varphi}(\phi) = \exp (\langle \phi, \varphi \rangle)$. Using the Gaussian characteristic formula, we have $\E ( F_{\varphi}(\phi_{0,n} + f_n) )  = e^{\langle f_n, \varphi \rangle} \E ( e^{\langle \phi_{0,n} ,  \varphi \rangle}) =  e^{\langle f_n, \varphi \rangle} e^{\frac{1}{2} \mathrm{Var}(\langle \phi_{0,n} , \varphi \rangle)} = e^{\langle f_n, \varphi \rangle} e^{ \frac{1}{2} \langle C_{0,n} * \varphi, \varphi \rangle} $ and similarly, since $C_{0,n} * g = f_n$ and $\langle C_{0,n} * \varphi, g \rangle = \langle \varphi, C_{0,n} *  g \rangle = \langle \varphi, f_n \rangle = \langle f_n, \varphi \rangle$:
\begin{align*}
\E \left( F_{\varphi}(\phi_{0,n}) e^{\langle \phi_{0,n}, g \rangle - \frac{1}{2} \langle f_n, g   \rangle} \right) & = e^{- \frac{1}{2} \langle f_n, g \rangle} \E \left( e^{\langle \phi_{0,n} ,  \varphi + g \rangle} \right) \\
& = e^{- \frac{1}{2} \langle f_n, g \rangle} e^{\frac{1}{2} \langle C_{0,n} * (\varphi + g), \varphi + g \rangle} \\
& = e^{- \frac{1}{2} \langle f_n, g \rangle} e^{\frac{1}{2} \langle C_{0,n} * \varphi, \varphi \rangle + \langle C_{0,n} * \varphi, g \rangle  + \frac{1}{2} \langle C_{0,n} * g , g \rangle} \\
& = \E \left( F_{\varphi}(\phi_{0,n} + f_n) \right).
\end{align*}

(ii). First, we prove that $C_{0,n} * f $ converges uniformly to $C_{0,\infty} * f $ on $\R^2$. Notice that $\norme{C_{0,n} * f - C_{0,\infty} * f}_{\infty} = \norme{C_{n,\infty} * f}_{\infty}  \leq \norme{f}_{\infty} \norme{C_{n,\infty}}_{L^1(\R^2)}$. Furthermore:
$$
\norme{C_{n,\infty}}_{L^1(\R^2)} = \int_{\R^2}  \int_0^{2^{-n}} c \left( \frac{y}{t} \right) \frac{dt}{t}  dy  \leq  \norme{c}_{\infty} \int_{\R^2}  \int_0^{2^{-n}} 1_{y \in B(0,2r_0 t)} \frac{dt}{t}  dy 
\leq C 2^{-2n}.
$$
Now we prove that the convergence holds in $L^2(\R^2)$. By Parseval, we have
$$
\norme{C_{0,n} * g - C_{0,\infty} * g}_{L^2(\R^2)}^2  = \norme{ \hat{C}_{n,\infty} \hat{g}}_{L^2(\R^2)}^2.
$$
Moreover, since  $\hat{C}_{n,\infty}(\xi) = \norme{\xi}^{-2} \int_{0}^{2^{-n} \norme{\xi}} u \hat{k}(u)^2 du$ (see the remark before  \eqref{coupling}  for a proof), we have:
$$  \norme{ \hat{C}_{n,\infty} \hat{g}}_{L^2(\R^2)}^2  =  \int_{\R^2} \left( \norme{\xi}^{-2} \int_{0}^{2^{-n} \norme{\xi}} u \hat{k}(u)^2 du \right)^2 \abs{\hat{g}(\xi)}^2 d\xi  \leq C 2^{-4n} \norme{\hat{k}}_{\infty}^4 \norme{g}_{L^2(\R^2)}^2
$$
and this completes the proof of assertion (ii).

(iii). We want to prove here that $(\phi_{0,n})_{n \geq 0}$ converges in law to $\phi_{0,\infty}$ in $\mathcal{S}'(\R^2)$. To this end, take a function $f \in \mathcal{S}(\R^2)$ and notice that:
$$
\E \left( \langle \phi_{0,n} , f \rangle^2 \right) = \int_{\R^2 \times \R^2} f(x) C_{0,n}(x,y) f(y) dx dy = \frac{1}{(2\pi)^2} \int_{\R^2} \hat{C}_{0,n}(\xi) \abs{\hat{f}(\xi)}^2 d \xi.
$$
Since $\hat{C}_{0,n}(\xi) = \norme{\xi}^{-2} \int_{2^{-n} \norme{\xi}}^{\norme{\xi}} u \hat{k}(u)^2 du$ for $n \in \N \cup \lbrace \infty \rbrace$, by monotone convergence, we get that $\E ( \langle \phi_{0,n} , f \rangle^2 )$ converges to $\E ( \langle \phi_{0,\infty} , f \rangle^2 )$. Thus, we have the convergence of the characteristic functionals: $\E (e^{i \langle \phi_{0,n} , f \rangle}) = e^{- \frac{1}{2} \E \left( \langle \phi_{0,n} , f \rangle^2 \right)} \underset{n \to \infty}{ \rightarrow} e^{- \frac{1}{2} \E \left( \langle \phi_{0,\infty} , f \rangle^2 \right)}$, which is enough to obtain the convergence in law, see for instance \cite{GeneralizedLoi}.

\end{proof}

\section{Small noise regime: proof of Theorem \ref{th:BoundGamma}}

We want to prove here that $\gamma_c > 0$. To do it, we will show by induction that the ratio between large quantiles and small quantiles is uniformly bounded in $n$. Recall the notations $l_n$, $\bar{l}_n$ and $\delta_n$ from Subsection \ref{DefLengths}. Then $\delta_n \nearrow \delta_{\infty}$ when $n$ goes to $\infty$. We start by showing that when $\eps$ and $\gamma$ are small enough, but fixed, then $\delta_{\infty} < \infty$. By our tail estimates, Corollary \ref{LowerTailThin} (with $l_n \geq \mu_n \delta_{\infty}^{-1}$)  and Proposition \ref{LemmeRecu} (with $\delta_n l_n \leq \delta_{\infty} \mu_n$) this implies the tightness of $\log L_{1,1}^{(n)} - \log \mu_n$.

\begin{proof}[Proof of Theorem \ref{th:BoundGamma}]
We proceed according to the following steps:
\begin{enumerate}
\item Relate the ratio $\delta_n$  between small quantiles and high quantiles to $\mathrm{Var} \log L_{1,1}^{(n)}$. 

\item Give an upper bound on $\mathrm{Var} \log L_{1,1}^{(n)}$ using the Efron-Stein inequality. The bound obtained involves a sum indexed by blocks $P \in \mathcal{P}_k$  for $0 \leq k \leq n$. 

\item Get rid of the independent copy term which appears when using the Efron-Stein inequality and see how a small value of $\gamma$ makes the variance smaller.

\item Give an upper bound on diameter and a lower bound on the left-right distance involving the same quantities at a higher scale.

\item Use the tails estimates obtained for the higher scales and control the ratio of the upper bound over the lower bound using $\delta_{n-1}$.

\item Conclude the induction. 
\end{enumerate}

\textbf{Step 1.} To link the quantiles and the variance of a random variable $X$ notice that for $l' \geq l$ we have $2 \mathrm{Var}( X )  = \mathbb{E}( ( X'-X )^2 )  \geq \mathbb{E} ( 1_{X' \geq l'} 1_{X \leq l} (X' -X )^2 ) \geq  \Pro ( X \geq l' ) \Pro ( X \leq l ) ( l' - l )^2$ where $X'$ is an independent copy of $X$.
Together with the RSW estimates obtained in Theorem \ref{RSWQuantiles} (using \eqref{eq:RSW2} with $a' = 3$, $b'=1$, $a =1$, $b=1$ and \eqref{eq:RSW1} with $a'=1$, $b'=1$, $a=1$, $b =1$), we have, for some constant $C_{\eps}$ depending on $\eps$ but not on $n$:
\begin{equation}
\label{Goal}
\frac{\bar{l}_{3,1}^{(n)}(\eps)}{l_{1,3}^{(n)}(\eps)} \leq e^{C_\eps} \frac{\bar{l}_{1,1}^{(n)}(\eps^C/3)}{l_{1,1}^{(n)}(\eps/C)}  \leq e^{C_\eps} \exp \left( \sqrt{\frac{6C}{\eps^{C+1}} \mathrm{Var}\left( \log L_{1,1}^{(n)} \right)} \right).
\end{equation}

\textbf{Step 2.} The idea is then to bound  $\mathrm{Var} ( \log L_{1,1}^{(n)} )$ by a term involving $\delta_{n-1}$ and $\gamma$. To do it, we will use the Efron-Stein inequality, see for instance \cite{ADH} Section 3 where it is used to give an upper bound for the variance of the distance between two points in the model of first passage percolation, which is a similar problem to ours. To this end, note that the variable $L_{1,1}^{(n)}$ can be written as a function of independent fields attached to dyadic blocks: $L_{1,1}^{(n)} = F ( (\phi_{k,P} )_{0 \leq k \leq n,P \in \mathcal{P}_k})$ and only the blocks that intersect $[0,1]^2$ contribute. For $P \in \mathcal{P}_k$, we denote by $L_{1,1}^{(n),P}$ the length obtained by replacing the block field $\phi_{k,P}$ by an independent copy $\phi_{k,P}'$ and keeping all other block fields fixed. The Efron-Stein inequality gives:
\begin{equation}
\label{EfronStein}
\mathrm{Var} \log L_{1,1}^{(n)} \leq \sum_{k=0}^{n} \sum_{P \in \mathcal{P}_k} \mathbb{E} \left( \left( \log L_{1,1}^{(n),P} - \log L_{1,1}^{(n)} \right)_{+}^2 \right).
\end{equation}

\textbf{Step 3.} We then focus on the term in the summation. For $0 \leq k \leq n$, $P \in \mathcal{P}_k$, 
\begin{align*}
L_{1,1}^{(n),P}  & \leq  \int_{\pi_n} \left( e^{\frac{\gamma}{2} \left( \phi_{0,n} - \phi_{k,P} + \phi_{k,P}' \right)} - e^{\frac{\gamma}{2} \phi_{0,n}}  \right)ds  + L_{1,1}^{(n)} \\
& \leq \int_{\pi_n} e^{ \frac{\gamma}{2} \phi_{0,n}} \left( e^{\frac{\gamma}{2} \left(  - \phi_{k,P} + \phi_{k,P}' \right)} - 1 \right)_{+} 1_{\pi_{n}(s) \in P^{2r_0}} ds  + L_{1,1}^{(n)} \\ 
& \leq \gamma  \int_{\pi_n} e^{\frac{\gamma}{2}  \phi_{0,n}} e^{(1+\frac{\gamma}{2}) \left( - \phi_{k,P} + \phi_{k,P}' \right)_{+}}  1_{\pi_n (s) \in P^{2r_0}}ds  + L_{1,1}^{(n)} 
\end{align*}
where $P^{2r_0} := P + B(0,2^{-k} \cdot 2r_0)$ and where we used in the last inequality the bound 
$$
(e^{\gamma x} - 1)_{+} \leq e^{\gamma x_{+}} -1 = \sum_{k \geq 1} \frac{(\gamma \ x_{+})^{k}}{k!} \leq \gamma \ x_{+} \sum_{k \geq 1} \frac{(\gamma \ x_{+})^{k-1}}{(k-1)!} \leq \gamma e^{x_{+}} e^{\gamma x_{+}}.
$$
By setting $S_{k,P} := \sup_{P^{2r_0}} \abs{\phi_{k,P}} +\sup_{P^{2r_0}} \abs{\phi_{k,P}'}$, this gives, using $\log (1 + x) \leq x$: 
\begin{align*}
\mathbb{E} ( ( \log L_{1,1}^{(n),P} - \log L_{1,1}^{(n)} )_{+}^2 ) & \leq \gamma^2 \mathbb{E} ( ( L_{1,1}^{(n)} )^{-2} ( \int_{\pi_n}  e^{\frac{\gamma}{2} \phi_{0,n}} e^{(1+\frac{\gamma}{2}) (  - \phi_{k,P} + \phi_{k,P}' )_{+}}  1_{\pi_n (s) \in P^{2r_0}}ds  )^2 ) \\
& \leq \gamma^2   \mathbb{E} ( e^{ C S_{k,P} } ( L_{1,1}^{(n)} )^{-2} ( \int_{\pi_n}  e^{\frac{\gamma}{2} \phi_{0,n} }  1_{\pi_n (s) \in P^{2r_0}}ds  )^2 ) \\
\end{align*}
which finally gives: 
\begin{equation}
\label{EfronBound}
\mathrm{Var} \log L_{1,1}^{(n)} \leq \gamma^2 \sum_{k=0}^{n} \sum_{P \in \mathcal{P}_k} \mathbb{E} \left( \frac{e^{C S_{k,P}}}{\left(L_{1,1}^{(n)} \right)^2} \left( \int_{\pi_{n}} e^{\frac{\gamma}{2} \phi_{0,n}} 1_{\pi_n(s) \in P^{2r_0}} ds \right)^2 \right).
\end{equation}
Notice that for $k = 0$ the term in the summation corresponds to $\E (e^{ C S_{0,[0,1]^2} } )$. 

\medskip

\textbf{Step 4.} We focus now on the case where $k \in  \lbrace 1, \dots ,  n \rbrace$. Since $ \E ( e^{C S_{k,P}} )^{1/2} $ is independent of $k$ and $P$ by scaling and finite by Fernique, we have by Cauchy-Schwarz:
\begin{align*}
& \sum_{P \in \mathcal{P}_k} \mathbb{E} \left( e^{C S_{k,P}}    \left( L_{1,1}^{(n)} \right)^{-2} \left( \int_{\pi_{n}} e^{\frac{\gamma}{2} \phi_{0,n}} 1_{\pi_n(s) \in P^{2r_0}} ds \right)^2 \right) \\ 
& \leq  \sum_{P \in \mathcal{P}_k} \mathbb{E} \left( e^{ C S_{k,P}}  \right)^{1/2}   \E \left( \left( L_{1,1}^{(n)} \right)^{-4} \left( \int_{\pi_{n}} e^{\frac{\gamma}{2} \phi_{0,n}} 1_{\pi_n(s) \in P^{2r_0}} ds \right)^4 \right)^{1/2} \\
 & \leq C  \sum_{P \in \mathcal{P}_k}  \E \left( \left( L_{1,1}^{(n)} \right)^{-4} \left( \int_{\pi_{n}} e^{\frac{\gamma}{2} \phi_{0,n}} 1_{\pi_n(s) \in P^{2r_0}} ds \right)^4 \right)^{1/2}.
\end{align*}

\medskip

\textbf{Step 4. (a).} Upper bound. Notice that for $P \in \mathcal{P}_k$, $\int_{\pi_{n}} e^{\frac{\gamma}{2} \phi_{0,n}} 1_{\pi_n(s) \in P^{2r_0}} ds \leq 9 \max_{Q \sim P} \mathrm{Diam} (Q, e^{\gamma \phi_{0,n}} ds^2)$. Indeed, $P^{2 r_0}$ is included in the union of $P$ and its eight neighboring squares (see Figure \ref{fig:9Diam}). Thus, the length of the parts of $\pi_n$ included in $P^{2r_0}$ is less than the diameter of this union, which itself is less than the sum of the diameter of all these squares. 

\begin{figure}[h!]
\centering
\includegraphics[scale=0.3]{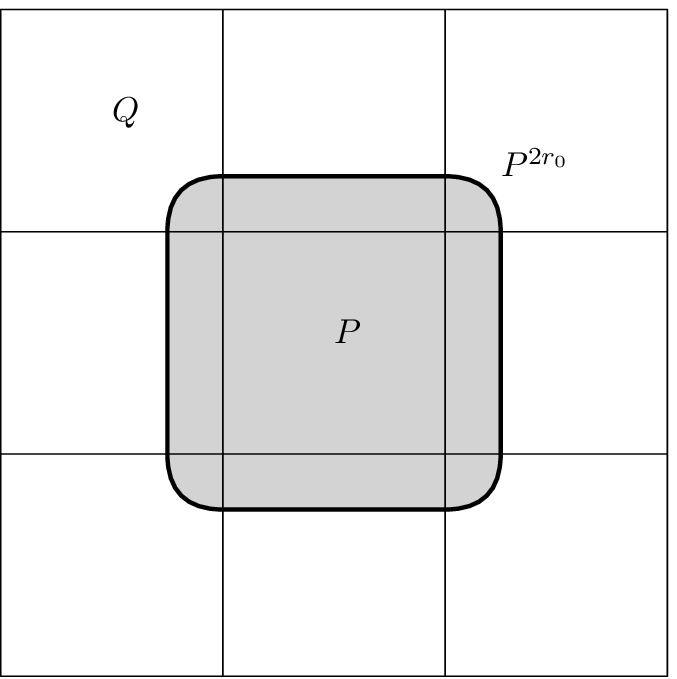}
\caption{$2r_{0}$-enlargement of $P$ with its neighbors}
\label{fig:9Diam}
\end{figure}

Let $N_{k}$ denote the number of dyadic squares of size $2^{-k}$ visited by $\pi_n$. Since the number of blocks $P^{2 r_0}$ (with $P \in \mathcal{P}_k$) visited by $\pi_n$ is less than $9 N_k$, a.s.
$$
\sum_{P \in \mathcal{P}_k}  \left( \int_{\pi_{n}} e^{\frac{\gamma}{2} \phi_{0,n}} 1_{\pi_n(s) \in P^{2r_0}} ds \right)^4  \leq C N_k \underset{ P \in \mathcal{P}_k}{\sup} \mathrm{Diam}\left(P, e^{\gamma \phi_{0,n}} ds^2 \right)^4
$$
and by decoupling the first $k-1$ scales of the field $\phi_{0,n} = \phi_{0,k-1} + \phi_{k,n}$, a.s.
\begin{equation}
\label{eq:upperbound}
\sum_{P \in \mathcal{P}_k}  \left( \int_{\pi_{n}} e^{\frac{\gamma}{2} \phi_{0,n}} 1_{\pi_n(s) \in P^{2r_0}} ds \right)^4  \leq C e^{ 2\gamma \sup_{[0,1]^2} \phi_{0,k-1}} N_k \underset{ P \in \mathcal{P}_k}{\sup} \mathrm{Diam}\left(P, e^{\gamma \phi_{k,n}} ds^2 \right)^4.
\end{equation}

\medskip

\textbf{Step 4. (b).} Lower bound. If $\tilde{N}_k$ denotes the maximal number of disjoint left-right rectangle crossings of size $2^{-k} (1,3)$ for  $\pi_n$, among such rectangles filling vertically and horizontally $[0,1]^2$, spaced by $2^{-k}$ (this set is denoted by $I_k$ and defined in \eqref{Def:Ik}), we have $\tilde{N}_k \geq c N_k$ and $\tilde{N}_k \geq c 2^k$ for a small constant $c > 0$. Indeed, if a dyadic square is visited, one of the four rectangles around it is crossed (see Figure \ref{fig:CrossNumber}). Considering a fraction of them gives the first claim. It is easy to check the second claim by noticing that $\pi_n$ crosses each rectangle of size $2^{-k} \times 1$ filling $[0,1]^2$.

\begin{figure}[h!]
\centering
\includegraphics[scale=0.5]{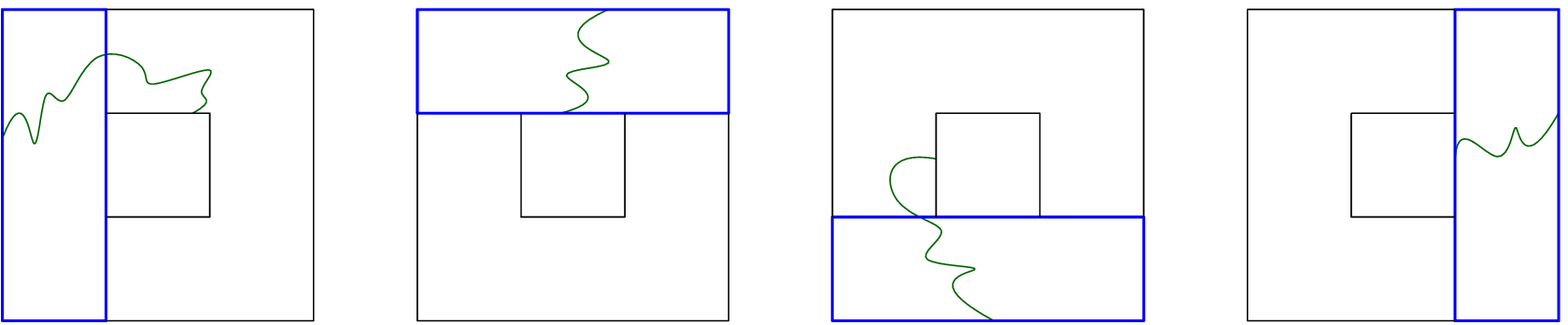}
\caption{Square visited and associated rectangle crossings}
\label{fig:CrossNumber}
\end{figure}

By decoupling the first $k-1$ scales, we get $L_{1,1}^{(n)} \geq c N_k e^{\frac{\gamma}{2} \inf_{[0,1]^2} \phi_{0,k-1}} \inf_{P \in I_k} L^{(k,n)}(P)$ as well as  $L_{1,1}^{(n)} \geq c 2^k e^{\frac{\gamma}{2} \inf_{[0,1]^2} \phi_{0,k-1}} \inf_{P \in I_k} L^{(k,n)}(P)$ hence:
\begin{equation}
\label{eq:lowerbound}
\left( L_{1,1}^{(n)} \right)^4 \geq c 2^{3 k} N_k e^{ 2\gamma \underset{[0,1]^2}{\inf} \phi_{0,k-1}} \left( \inf_{P \in I_k} L^{(k,n)}(P) \right)^4.
\end{equation}

\textbf{Step 5.} Moment estimates and inductive inequality. By concavity of the map $x \mapsto \sqrt{x}$ we have:
\begin{align*}
& \sum_{P \in \mathcal{P}_k}  \E \left( \left(L_{1,1}^{(n)} \right)^{- 4} \left( \int_{\pi_{n}} e^{\frac{\gamma}{2} \phi_{0,n}} 1_{\pi_n(s) \in P^{2r_0}} ds \right)^4 \right)^{1/2}  \\
&  \leq \abs{\mathcal{P}_k}^{1/2}   \E \left(\left(L_{1,1}^{(n)} \right)^{- 4} \sum_{P \in \mathcal{P}_k}  \left( \int_{\pi_{n}} e^{\frac{\gamma}{2} \phi_{0,n}} 1_{\pi_n(s) \in P^{2r_0}} ds \right)^4 \right)^{1/2}.
\end{align*}
Gathering, \eqref{eq:upperbound} and \eqref{eq:lowerbound}, 
\begin{align*}
& \left( L_{1,1}^{(n)} \right)^{-4} \sum_{P \in \mathcal{P}_k}  \left( \int_{\pi_{n}} e^{\frac{\gamma}{2} \phi_{0,n}} 1_{\pi_n(s) \in P^{2r_0}} ds \right)^4 \\
& \leq C 2^{-3k}  e^{ 4 \gamma \sup_{[0,1]^2} \abs{\phi_{0,k-1}}} \underset{ P \in \mathcal{P}_k}{\sup} \mathrm{Diam}\left(P, e^{\gamma \phi_{k,n}} ds^2 \right)^4 \left( \inf_{P \in I_k} L^{(k,n)}(P) \right)^{-4}.
\end{align*}
Since $\abs{\mathcal{P}_k} = 4^k$, by independence between scales,  
\begin{align*}
& \sum_{P \in \mathcal{P}_k}  \E \left( \left(L_{1,1}^{(n)} \right)^{- 4} \left( \int_{\pi_{n}} e^{\frac{\gamma}{2} \phi_{0,n}} 1_{\pi_n(s) \in P^{2r_0}} ds \right)^4 \right)^{1/2} \\
& \leq C 2^{- \frac{1}{2} k} \mathbb{E} \left( e^{4 \gamma \underset{[0,1]^2}{\sup} \abs{\phi_{0,k-1}}} \right)^{1/2} \mathbb{E}\left( \underset{ P \in \mathcal{P}_k}{\sup} \mathrm{Diam}\left(P, e^{\gamma \phi_{k,n}} ds^2 \right)^4 \left(   \inf_{P \in I_k} L^{(k,n)}(P) \right)^{-4}\right)^{1/2}.
\end{align*}
Using Lemma \ref{ExpoMoment} to control the exponential moment,  the first term is bounded by $2^{4\gamma k} e^{C \sqrt{k}}$. For the second term, notice that the product inside the expectation is between an increasing and a decreasing function of the field. Hence, by the positive association property (Theorem \ref{PosAssoc}):
\begin{align*}
& \E \left( \underset{ P \in \mathcal{P}_k}{\sup} \mathrm{Diam}  \left(P, e^{\gamma \phi_{k,n}} ds^2 \right)^4 \left(   \inf_{P \in I_k} L^{(k,n)}(P) \right)^{-4}\right)^{1/2} \\
& \leq \mathbb{E}\left( \underset{ P \in \mathcal{P}_k}{\sup} \mathrm{Diam}\left(P, e^{\gamma \phi_{k,n}} ds^2 \right)^4 \right)^{1/2} \E \left( \left( \inf_{P \in I_k} L^{(k,n)}(P) \right)^{-4}\right)^{1/2}.
\end{align*}
By scaling, the field involved is $\phi_{0,n-k}$. We use our estimates for the diameters, Proposition \ref{SupTails}, for the first term and Corollary \ref{LowerTailThin} for the second one. More precisely, by standard inequality between expected value of positive random variable and integration of tail estimates we have:
$$
\mathbb{E}\left( \underset{ P \in \mathcal{P}_k}{\sup} \mathrm{Diam}\left(P, e^{\gamma \phi_{k,n}} ds^2 \right)^4 \right)^{1/2}  \leq 2^{-2k} \delta_{n-k}^2 l_{n-k}^2 e^{c \gamma k} \leq \delta_{n-1}^2 2^{-2k}  l_{n-k}^2 e^{c \gamma k}
$$
and
$$
\E \left( \left( \inf_{P \in I_k} L^{(k,n)}(P) \right)^{-4}\right)^{1/2} \leq 2^{2k} l_{n-k}^{-2} e^{C \sqrt{k}}.
$$
Altogether, we get  for $1 \leq k \leq  n$:
\begin{equation}
\label{eq:BoundStep5}
\sum_{P \in \mathcal{P}_k} \mathbb{E} \left( \frac{e^{C S_{k,P}}}{L_{1,1}^{(n) 2}} \left( \int_{\pi_{n}} e^{\frac{\gamma}{2} \phi_{0,n}} 1_{\pi_n(s) \in P^{2r_0}} ds \right)^2 \right) \leq  \delta_{n-1}^2 2^{- \frac{1}{2} k} e^{c \gamma k} e^{C \sqrt{k}}
\end{equation}
for some constant $c > 0$. 

\medskip

\textbf{Step 6.} Combining \eqref{EfronBound} and \eqref{eq:BoundStep5} we get
\begin{equation}
\label{eq:SeriesBound}
\mathrm{Var} \log L_{1,1}^{(n)} \leq \gamma^2 \delta_{n-1}^2 \sum_{k=0}^n  2^{- \frac{1}{2} k} e^{c \gamma k} e^{C \sqrt{k}} \leq \gamma^2 \delta_{n-1}^2 \sum_{k=0}^{\infty}  2^{- \frac{1}{2} k} e^{c \gamma k} e^{C \sqrt{k}}.
\end{equation}

 Hence for $\gamma$ small enough the series in the right-hand side of \eqref{eq:SeriesBound} converges and we have the bound $\mathrm{Var} \log L_{1,1}^{(n)} \leq  \gamma^2 \left(C + C  \delta_{n-1}^2  \right)$. Coming back to \eqref{Goal}, if $\delta_{n-1} <  M$ then $\delta_{n} < e^{C_{\eps}} \exp (C \gamma \delta_{n-1}) < e^{C_{\eps}} \exp (C \gamma M) $. Hence taking $M > e^{C_{\eps}}$ and $\gamma$ small enough so that $e^{C_{\eps}} \exp (C \gamma M) < M$ shows that there exists $\gamma_0$  (which depends on $\eps$) such that if $\gamma < \gamma_0$, $\delta_{\infty} < \infty$.  Finally, we can conclude that $\gamma_c >0$  by use of Corollary \ref{LowerTailThin} and Proposition \ref{LemmeRecu}.
\end{proof}

\section{Independence of $\gamma_c$ with respect to $k$: proof of Theorem \ref{th:Indep}}

We want to prove that $\gamma_c$ is independent of $k$ i.e. if we have two bump functions $k_1$, $k_2$ then $\gamma_c(k_1) = \gamma_c(k_2)$. We will prove that if $\log L_{1,1}(\phi_{0,n}^1)  - \log \mu_n^1$ is tight then $\log L_{1,1}(\phi_{0,n}^2) - \log \mu_n^2$ is also tight, where the superscripts corresponds to the bump function $k_i$ for $i \in \lbrace 1,2 \rbrace$.  The proof presented here relies on the assumption that $\hat{k}_1$ and $\hat{k}_2$ have similar tails. 

\medskip

\textbf{Main lines of the proof.} The main idea of the proof is to couple $\phi_{0,n}^1$ and $\phi_{0,n}^2$ up to some additive noises that don't affect too much the lengths. To control the perturbation due to the noises, note that if $\delta \phi$ is a low frequency noise, the length $L_{1,1}(\phi)$ is comparable to the length $L_{1,1}(\phi + \delta \phi)$  by a uniform bound a.s.: 
\begin{equation}
\label{IneLowFreq}
e^{\inf_{[0,1]^2} \delta \phi} L_{1,1}(\phi) \leq L_{1,1}(\phi + \delta \phi)  \leq e^{\sup_{[0,1]^2} \delta \phi} L_{1,1}(\phi) 
\end{equation}
and if $\delta \phi$ is a high frequency noise with bounded pointwise variance we have a one-sided bound on high and low quantiles given by the following lemma.
\begin{Lemma}
\label{Inequality} If $\Phi$ is a continuous field and $\delta \Phi$ is an independent continuous centered Gaussian field with variance bounded by $C$ then 
\begin{enumerate}
\item
$l_{1,1}^{\Phi + \delta \Phi} (\eps) \leq \eps^{-1} e^{\frac{1}{2} C} l_{1,1}^{\Phi}(2\eps)$,

\item 
$\bar{l}_{1,1}^{\Phi + \delta \Phi}(2\eps) \leq \eps^{-1} e^{\frac{1}{2} C}  \bar{l}_{1,1}^{\Phi}(\eps)$.
\end{enumerate} 
\end{Lemma}
\begin{proof}
To bound from above $L_{1,1}^{\Phi + \delta \Phi}$, we take a geodesic for $\Phi$ and use a moment estimate on $\delta \Phi$.  We start with the lower tail. For $s > 0$ we have 
\begin{align*}
\Pro \left( L_{1,1}^{\Phi} \leq l_{1,1}^{\Phi + \delta \Phi}(\eps) e^{-s}  \right) & \leq \Pro \left( L_{1,1}^{\Phi + \delta \Phi} \leq e^s L_{1,1}^{\Phi}, L_{1,1}^{\Phi} \leq l_{1,1}^{\Phi + \delta \Phi}(\eps) e^{-s} \right) + \Pro \left( L_{1,1}^{\Phi + \delta \Phi} > e^s L_{1,1}^{\Phi}\right) \\
& \leq \Pro \left( L_{1,1}^{\Phi + \delta \Phi} \leq l_{1,1}^{\Phi + \delta \Phi}(\eps)  \right)  +  \Pro \left(  \int_{\pi^{\Phi}} e^{\Phi + \delta \Phi} ds > e^s L_{1,1}^{\Phi}\right) \\
& \leq \eps + e^{ \frac{1}{2} \sup \mathrm{Var}( \delta \Phi) - s}
\end{align*}
where we used Chebychev inequality and the independence between the field $\Phi$ and $\delta \Phi$ in the last inequality. Taking then $s = \frac{1}{2} \sup \mathrm{Var} (\delta \Phi) - \log \eps$ completes the proof of (i). For the upper tails taking the same $s$ gives 
\begin{align*}
\Pro \left( L_{1,1}^{\Phi + \delta \Phi} \geq \bar{l}_{1,1}^{\Phi}(\eps) e^s  \right)  & \leq \Pro \left( L_{1,1}^{\Phi + \delta \Phi} \geq \bar{l}_{1,1}^{\Phi}(\eps) e^s,   \bar{l}_{1,1}^{\Phi}(\eps) \geq L_{1,1}^{\Phi}  \right) +  \Pro \left( L_{1,1}^{\Phi} \geq  \bar{l}_{1,1}^{\Phi}(\eps) \right)  \\
& \leq \Pro \left(L_{1,1}^{\Phi + \delta \Phi}  \geq e^s L_{1,1}^{\Phi}\right)   +  \eps  \\
&  \leq 2 \eps
\end{align*}
which concludes the proof of the lemma.
\end{proof}

Note that if $\delta \phi$ is a high frequency noise, with scale dependence $2^{-n}$, say an approximation of $4^n$ i.i.d. standard Gaussian variables, its supremum is of order $\sqrt{n}$ and the inequality \eqref{IneLowFreq} is inappropriate compared to Lemma \ref{Inequality} which gives a bound of order one, but one-sided. However, for a low frequency noise $\delta \phi$, independent of $n$, the bound \eqref{IneLowFreq} gives two-sided bounds on quantiles.

\medskip

If $(X_n)$ and $(Y_n)$ denote two sequences of positive random variables, with positive density with respect to the Lebesgue measure on $(0,\infty)$, we write $X_n \lesssim Y_n$ if there exists a constant $C$ independent of $n$ such that for every $\eps > 0$ small, there exists $C_{\eps}$, independent of $n$, such that $F_{X_n}^{-1}(\eps/C) \leq  C_{\eps} F_{Y_{n}}^{-1}(\eps)$ and $F_{X_n}^{-1}(1-C \eps) \leq C_{\eps} F_{Y_n}^{-1}(1- \eps)$, where $F_{X}(x) := \Pro (X \leq x)$ for a random variable $X$. A direct corollary of  Lemma \ref{Inequality} is the following: if $(\phi_n)_{n \geq 0}$ and $(\delta \phi_n)_{n \geq 0}$ are two sequences of independent centered continuous Gaussian fields,  and that the pointwise variance of $\delta \phi_n$ is bounded, then $L_{1,1}(\phi_n + \delta \phi_n) \lesssim L_{1,1}(\phi_n)$. Similarly, a direct consequence of \eqref{IneLowFreq} is that, under the same assumptions for $(\phi_n)_{n \geq 0}$, if $\psi$ is a continuous centered Gaussian field, then $ L_{1,1} (\phi_n) \lesssim L_{1,1} (\phi_n + \psi) \lesssim L_{1,1} (\phi_n)$.

\medskip

Now that the notations and the key tools are settled, let us explain the main idea of the proof. Let us assume for now that we have the following couplings, for a fixed $k$:

\begin{enumerate}
\item 
$\left( \phi_{0,n}^1(x) + \delta_n^1(x) \right)_{x \in \R^2} \overset{(d)}{=}  \left( \phi_{0,n}^2(x) + \delta_n^2(x) \right)_{x \in \R^2}$

\item $\left( \phi_{n,n+k}^{1}(x) + \psi(x) \right)_{x \in \R^2} \overset{(d)}{=}  \left( \delta_n^1(x)  + r_n^1(x) \right)_{x \in \R^2}$

\item $\left( \phi_{n,n+k}^{2}(x) + \psi(x) \right)_{x \in \R^2} \overset{(d)}{=}  \left( \delta_n^2(x)  + r_n^2(x) \right)_{x \in \R^2}$
\end{enumerate}
where fields in the same side of an equality are independent and all fields are centered, continuous and Gaussian. Let us also assume that  $\psi$ is a fixed continuous Gaussian field, independent of $n$ and thus a low frequency noise. Notice that if such couplings hold, it is clear that the $\delta_n^i$'s and $r_n^i$'s have bounded pointwise variance since this is the case for the fields in the left-hand sides of (ii) and (iii). We then have, since $\psi$ is a low frequency noise, by using (ii) and Lemma \ref{Inequality}:
$$
L_{1,1} \left( \phi_{0,n+k}^1 \right)  \lesssim L_{1,1} \left( \phi_{0,n}^1 + \delta_n^1 + r_n^1 \right)   \lesssim  L_{1,1} \left( \phi_{0,n}^1 + \delta_n^1 \right)  \lesssim  L_{1,1} \left( \phi_{0,n}^1 \right)
$$
which gives, using (i):
\begin{equation}
\label{eq:IneApprox}
L_{1,1} \left( \phi_{0,n+k}^1 \right) \lesssim L_{1,1}(\phi_{0,n}^2 + \delta_n^2) \lesssim L_{1,1}(\phi_{0,n}^1).
\end{equation}
If we suppose that $\log L_{1,1} ( \phi_{0,n}^1 ) - \log \mu_n^1$ is tight,  then $( ( \mu_n^1)^{-1} \mu_{n+k}^1 )_{n \geq 0}$  is bounded by Lemma \ref{BoundsScales}. But then, using \eqref{eq:IneApprox}, $\log L_{1,1} ( \phi_{0,n}^2+\delta_n^2 ) - \log \mu_n^1$ is tight. Furthermore, this implies the tightness of $ \log L_{1,1} ( \phi_{0,n}^2) - \log \mu_n^1$ since
$$
  L_{1,1} \left( \phi_{0,n+k}^2 + \delta_{n+k}^2 \right)   \lesssim   L_{1,1} \left( \phi_{0,n+k}^2 \right) \lesssim L_{1,1} \left( \phi_{0,n}^2 + \delta_n^2 \right).
$$
Finally, the tightness of $ \log L_{1,1}  ( \phi_{0,n}^2 ) - \log \mu_n^2$ follows from the fact that if $X$ is random variable and $\mu(X)$ is its median, then for every $a \in \R$, $\mu(X + a) =  \mu(X) + a$. This concludes the proof up to the results we claimed on the couplings.

\medskip

All the fields in the couplings will be defined by using the following standard result:

\begin{Lemma}
\label{lem:exisfield}
If $f$ is a continuous, symmetric and nonnegative function on $\R^d$ such that $\norme{\xi} f(\xi) \in L^1(\R^d)$, then one can define a continuous stationary centered Gaussian field with covariance given by:
$$
C(x,y)  := \frac{1}{(2\pi)^2} \int_{\R^d} f(\xi) e^{i (x-y) \cdot \xi} d\xi.
$$
\end{Lemma}

\begin{proof} Since $f \in L^1(\R^d)$, $C$ is well-defined.
Then, since $f$ is symmetric, a change of variables gives that $C$ is real-valued and $C(x,y) = C(y,x)$. Moreover, notice that $(C(x,y))_{x,y \in \R^2}$ is positive semidefinite: for every $(a_k)_{1 \leq k \leq n}$ and $(x_k)_{1 \leq k \leq n}$ in $(\R^d)^n$ we have
\begin{align*}
\sum_{k,l=1}^n a_k C(x_k,x_l) a_l &  = \frac{1}{(2\pi)^2} \int_{\R^d} f(\xi) \left( \sum_{k=1}^n  a_k e^{i x_k \cdot \xi}  \right) \left( \sum_{l=1}^n  a_l e^{- i x_l \cdot \xi}  \right) d\xi  \\
& =  \frac{1}{(2\pi)^2} \int_{\R^d} f(\xi) \abs{ \sum_{k=1}^n  a_k e^{i x_k \cdot \xi} }^2 d\xi \geq 0.
\end{align*}
By a standard result on Gaussian processes (see \cite{Adler} Section 1), there exists a centered Gaussian process $(h(x))_{x \in \R^d}$ whose covariance is given by $\E ( h(x) h(y)) = C(x,y)$. Finally, since we have the Lipschitz bound $\E ( (h(x) - h(y) )^2) \leq 2 \norme{x-y} \int_{\R^d} f(\xi) \norme{\xi} d\xi$ and $\norme{\xi} f(\xi) \in L^1(\R^d)$, by the Kolmogorov continuity criterion there exists a modification of $h$ which is continuous. 
\end{proof}

We also recall that $C_{0,n}(x) = \int_{2^{-n}}^1 c\left( \frac{x}{t} \right) \frac{dt}{t} = \int_{2^{-n}}^1 c_t(x) \frac{dt}{t}$ with $c_t(\cdot) = c(\cdot /t)$ thus its Fourier transform satisfies $\hat{C}_{0,n}(\xi) = \int_{2^{-n}}^1 \hat{c}_t(\xi) \frac{dt}{t} = \int_{2^{-n}}^1 t \hat{c}(t \xi) dt $ and since $c = k * k $, $\hat{c} = \hat{k}^2$ and then $ \hat{C}_{0,n}(\xi) = \int_{2^{-n}}^1 t \hat{k}(t \xi)^2 dt = \norme{\xi}^{-2} \int_{2^{-n} \norme{\xi}}^{\norme{\xi}} u \hat{k}(u)^2 du$. 

\medskip

\textbf{Coupling $\phi_{0,n}^1$ and $\phi_{0,n}^2$.} First we define $\delta_n^1$ and $\delta_n^2$ such that
\begin{equation}
\label{coupling}
\left( \phi_{0,n}^1(x) + \delta_n^1(x) \right)_{x \in \R^2} \overset{(d)}{=}  \left( \phi_{0,n}^2(x) + \delta_n^2(x) \right)_{x \in \R^2} 
\end{equation}
where $\delta_n^1$ (resp $\delta_n^2$) is a noise independent of $\phi_{0,n}^1$ (resp $\phi_{0,n}^2$). The covariance kernel of $\phi_{0,n}^i$ is given by $C_{0,n}^i(x,y) = \int_{2^{-n}}^1 c_i \left(  \frac{x-y}{t} \right) \frac{dt }{t}$ where $c_i = k_i * k_i$.  We recall also that these kernels are isotropic i.e. $C_{0,n}^i(x,y) = C_{0,n}^i ( \norme{x-y})$. By Fourier inversion (of Schwartz function) we can write 
$$
C_{0,n}^i(x) = \frac{1}{(2\pi)^2} \int_{\R^2} \hat{C}_{0,n}^i(\xi) e^{i \xi \cdot x} d\xi.
$$
We define $R_n^1$ by replacing the term $ \hat{C}_{0,n}^i(\xi) $  in the integrand by $f_n^1 (\xi) := \hat{C}_{0,n}^1(\xi)  \vee \hat{C}_{0,n}^2(\xi)  - \hat{C}_{0,n}^1(\xi)  \geq 0$ and similarly $R_n^2$ associated with $f_n^2(\xi) := \hat{C}_{0,n}^2(\xi)  \vee \hat{C}_{0,n}^1(\xi)  - \hat{C}_{0,n}^2(\xi)  \geq 0$ so that $C_{0,n}^1 + R_n^1 = C_{0,n}^2 + R_n^2$. By using Lemma \ref{lem:exisfield}, the covariance kernels $R_{n}^1$ and $R_n^2$ correspond to some continuous Gaussian fields $\delta_n^1$ and $\delta_n^2$ so that \eqref{coupling} holds and for $i \in \lbrace 1,2 \rbrace$, $\phi_{0,n}^i$ is independent of  $\delta_n^i$. 

\medskip

\textbf{Coupling the remaining noise with the lower scales.} We now prove the second coupling: 
\begin{equation}
\label{eq:coupling2}
\left( \phi_{n,n+k}^{1}(x) + \psi(x) \right)_{x \in \R^2} = \left( \delta_n^1(x)  + r_n^1(x) \right)_{x \in \R^2}.
\end{equation}
The goal is to show that the Fourier transform of the kernel of $\phi_{n,n+k}^1 + \psi$ (for $\psi$ to be specified) is larger than the one of $\delta_n^1$ in order to define, in a similar way as before, the continuous Gaussian field $r_n^1$, independent of $\delta_n^1$.

\medskip

To be precise, recall first that the spectrum of $\delta_n^1$ and $\phi_{n,n+k}^1$ are given respectively by $f_n^1 (\xi) = ( \hat{C}_{0,n}^2(\xi)  - \hat{C}_{0,n}^1(\xi) )  1_{\hat{C}_{0,n}^2(\xi)  \geq \hat{C}_{0,n}^1(\xi) }$ with $\hat{C}_{0,n}^i(\xi) = \norme{\xi}^{-2} \int_{2^{-n} \norme{\xi}}^{\norme{\xi}} u \hat{k}_i(u)^2 du$ and  $ \hat{C}_{n,n+k}^1(\xi) =  \norme{\xi}^{-2} \int_{2^{-n-k} \norme{\xi}}^{2^{-n} \norme{\xi}} u \hat{k}_1(u)^2 du$. If the spectrum of $\psi$ is given by $\norme{\xi}^{-2} g(\xi)$, we look for the inequality  $ f_n^1 (\xi) \leq \hat{C}_{n,n+k}^1(\xi) + \norme{\xi}^{-2} g(\xi)$ which is equivalent to
\begin{equation}
\label{eq:IneFou}
\left( \int_{2^{-n} \norme{\xi}}^{\norme{\xi}} u\hat{k}_2(u)^2 du - \int_{2^{-n} \norme{\xi}}^{\norme{\xi}} u\hat{k}_1(u)^2 du \right)_{+} \leq  \int_{2^{-(n+k)} \norme{\xi}}^{ 2^{-n} \norme{\xi}} u\hat{k}_1(u)^2 du + g(\xi).
\end{equation}
If the left-hand side is $0$, the inequality trivially holds. Otherwise, we want to get: 
$$
\int_{2^{-n} \norme{\xi}}^{\norme{\xi}} u\hat{k}_2(u)^2 du \leq  \int_{2^{-(n+k)} \norme{\xi}}^{\norme{\xi}} u\hat{k}_1(u)^2 du + g(\xi).
$$
Our analysis of this inequality will be separated in three steps, corresponding respectively to the low frequencies $[0,c 2^n]$, the high ones $[C 2^n, \infty)$ and the remaining part of the spectrum $[c 2^n , C 2^n]$, for $c$ and $C$ to be specified.   The field $\psi$ in \eqref{eq:coupling2} is defined in the first step. An additional step is devoted to the conclusion.

\medskip

\textbf{Step 1.} We start with the low frequencies $\norme{\xi} \leq c 2^n$. Since $\hat{k}_1$ and $\hat{k}_2$ are radially symmetric with the same $L^2$ normalization, $\int_{(0,\infty)} u\hat{k}_1(u)^2 du = \int_{(0,\infty)} u\hat{k}_2(u)^2 du$ and
\begin{align*}
\left( \int_{2^{-n} \norme{\xi}}^{\norme{\xi}} u\hat{k}_2(u)^2 du - \int_{2^{-n} \norme{\xi}}^{\norme{\xi}} u\hat{k}_1(u)^2 du \right)_{+}  &  \leq  \left( \int_{0}^{2^{-n}\norme{\xi}} u\hat{k}_1(u)^2 du - \int_{0}^{2^{-n}\norme{\xi}} u\hat{k}_2(u)^2 du \right)_{+}  \\
& \hphantom{ \leq } +  \left( \int_{\norme{\xi}}^{\infty} u\hat{k}_1(u)^2 du - \int_{\norme{\xi}}^{\infty} u\hat{k}_2(u)^2 du \right)_{+}.
\end{align*}

\smallskip

We define the continuous Gaussian field $\psi$ (independent of $n$), whose covariance kernel has Fourier transform defined  by $\norme{\xi}^{-2} g(\xi) :=  \norme{\xi}^{-2} \abs{ \int_{\norme{\xi}}^{\infty} u\hat{k}_1(u)^2du - \int_{\norme{\xi}}^{\infty} u\hat{k}_2(u)^2du }$. 

\smallskip

Since we want to show that the Fourier transform of the kernel of $\phi_{n,n+k}^1 + \psi$ is larger than the one of $\delta_n^1$, we want to prove  that for $\norme{\xi} \leq c 2^n$ ($c$ to be specified, small): 
$$
\left( \int_{0}^{2^{-n}\norme{\xi}} u\hat{k}_1(u)^2 du - \int_{0}^{2^{-n}\norme{\xi}} u\hat{k}_2(u)^2 du \right)_{+}  \leq \int_{2^{-(n+k)} \norme{\xi}}^{2^{-n}\norme{\xi}} u\hat{k}_1(u)^2 du.
$$
By setting $r = 2^{-n} \norme{\xi} $, we want to prove that for $r$ small enough ($r  \leq c$), and $k$ large enough but fixed:
\begin{equation}
\label{eq:ineqsmall}
\left( \int_{0}^{r} u\hat{k}_1(u)^2 du - \int_{0}^{r} u\hat{k}_2(u)^2 du \right)_{+}  \leq \int_{2^{-k} r}^{r} u\hat{k}_1(u)^2 du.
\end{equation}
Notice that when $r$ goes to $0$, $ \int_{0}^{r} u(\hat{k}_1(u)^2 du - \int_{0}^{r} u\hat{k}_2(u)^2 du \sim \frac{1}{2} r^2 (\hat{k}_1(0)^2 -  \hat{k}_2(0)^2)$. If the left-hand side is $0$, there is nothing to prove. Thus we can restrict to the case where it is $> 0$ i.e when $\hat{k}_1(0)^2 > \hat{k}_2(0)^2$ (notice that $\hat{k}(0) = \int_{B(0,r_0)} k(u) du > 0$ since $k$ is non-negative and $\int_{B(0,r_0)} k(x)^2 dx =1$). The asymptotic of the right-hand side is given by $\int_{2^{-k} r}^{r} u\hat{k}_1(u)^2 du \sim \frac{1}{2} r^2 \hat{k}_1(0)^2 (1-2^{-2k})$. Thus as soon as $\hat{k}_1(0)^2 - \hat{k}_2(0)^2 < \hat{k}_1(0)^2   (1-2^{-2k})$, there exists $r(k)$ such that for $r \leq r(k)$, the inequality \eqref{eq:ineqsmall} is satisfied.

\medskip

\textbf{Step 2.} We now deal with the large frequencies i.e. $\norme{\xi} \geq C 2^n$. Again, we look for the inequality \eqref{eq:IneFou}. Since we added the field $\psi$ and the following inequality holds,
\begin{align*}
\left( \int_{2^{-n} \norme{\xi}}^{\norme{\xi}} u\hat{k}_2(u)^2 du - \int_{2^{-n} \norme{\xi}}^{\norme{\xi}} u\hat{k}_1(u)^2 du \right)_{+}  &  \leq  \left( \int_{2^{-n}\norme{\xi}}^{\infty} u\hat{k}_2(u)^2 du -  \int_{2^{-n}\norme{\xi}}^{\infty} u\hat{k}_1(u)^2 du  \right)_{+}  \\
& \hphantom{ \leq } +  \left( \int_{\norme{\xi}}^{\infty} u\hat{k}_1(u)^2 du - \int_{\norme{\xi}}^{\infty} u\hat{k}_2(u)^2 du \right)_{+}  
\end{align*}
we look for the inequality:
$$
 \left( \int_{2^{-n}\norme{\xi}}^{\infty} u\hat{k}_2(u)^2 du -  \int_{2^{-n}\norme{\xi}}^{\infty} u\hat{k}_1(u)^2 du  \right)_{+}  \leq  \int_{2^{-(n+k)} \norme{\xi}}^{2^{-n} \norme{\xi}} u\hat{k}_1(u)^2 du.
$$
By setting $r = 2^{-n} \norme{\xi} $, we want to prove that for $r$ large enough ($r \geq C$), and $k$ large enough but fixed: 
\begin{equation}
\label{eq:inelarge}
\int_{r}^{\infty} u \hat{k}_2(u)^2 du \leq  \int_{2^{-k} r}^{\infty} u \hat{k}_1(u)^2 du.
\end{equation}
Since $\hat{k}_1(u) = e^{-b u^{\alpha}(1+o(1))}$ and $\hat{k}_2(u) = e^{-a u^{\alpha}(1+o(1))}$, we may assume that $0 < a \leq b$ (otherwise $k=0$ would be fine). Notice that there exists some $R >0$ such that for every $r \geq R$, $\int_r^{\infty} u \hat{k}_2(u)^2 du \leq e^{-b r^{\alpha}}$ and $e^{-3 a r^{\alpha}} \leq \int_r^{\infty} u \hat{k}_2(u)^2 du$. Then, by taking $k$ large enough so that $b > 3 a 2^{-k \alpha}$, for $r \geq 2^k R$ the inequality \eqref{eq:inelarge} is satisfied.

\medskip

\textbf{Step 3.} Take $k_0$ such that $\hat{k}_1(0)^2 - \hat{k}_2(0)^2 < \hat{k}_1(0)^2   (1-2^{-2k_0})$ and $b > 3a 2^{-k_0 \alpha}$ are satisfied. Set $c := r(k_0)$ and $C := 2^{k_0} R$, keeping the notations of Step 1 and Step 2. We proved there that \eqref{eq:IneFou} holds for $\norme{\xi} \leq c 2^n$ and  $\norme{\xi} \geq C 2^n$ and this inequality still holds by taking $k$ larger, with the same $c$ and $C$. We are left with the frequencies $c 2^n \leq \norme{\xi} \leq C2^n$. First, fix $k \geq k_0$ such that $\int_{2^{-k}C}^{\infty} u\hat{k}_1(u)^2 du >  \int_c^{\infty} u\hat{k}_2(u)^2 du$ (since $\int_{2^{-k}C}^{\infty} u\hat{k}_1(u)^2 du \to  \int_0^{\infty} u\hat{k}_2(u)^2 $). Then, fix $n_0$ such that $\int_{2^{-k} C}^{2^{n_0} c} u\hat{k}_1(u)^2 du  \geq \int_c^{\infty} u\hat{k}_2(u)^2 du $. Thus, for every $n \geq n_0$, $\norme{\xi} \in [c 2^n, C 2^n]$ we have: 
$$
\int_{2^{-(n+k)} \norme{\xi}}^{\norme{\xi}} u\hat{k}_1(u)^2 du \geq \int_{2^{-k} C}^{2^n c} u\hat{k}_1(u)^2 du  \geq \int_c^{\infty} u\hat{k}_2(u)^2 du   \geq  \int_{2^{-n} \norme{\xi}}^{\norme{\xi}}u\hat{k}_2(u)^2 du.
$$

\medskip

\textbf{Step 4.} We have proved that if $k$ is large enough, but fixed, for every $n \geq n_0$ the inequality \eqref{eq:IneFou} holds for all $\xi \in \R^2$. Also, our arguments prove that the same result is true by exchanging the subscripts $1$ and $2$ in \eqref{eq:IneFou}. Therefore, we can define for $i \in \lbrace 1,2 \rbrace$, $r_n^i$ whose covariance kernel has Fourier transform given by the positive difference in the inequality  \eqref{eq:IneFou}, multiplied by $\norme{\xi}^{-2}$. In particular, we get the couplings (ii) and (iii) with the desired properties on the fields. This completes the proof of the existence of the couplings, therefore the proof of Theorem \ref{th:Indep}.

\section{Appendix}

\subsection{Tail estimates for the supremum of $\phi_{0,n}$}

We derive in the following lemma some tail estimates for the field $\phi_{0,n}$. The tail estimates are obtained by controlling a discretization of $\phi_{0,n}$ (by union bound and Gaussian tail estimates) and its gradient.

\begin{Lemma}
\label{SupTailsAppendix}
The supremum of the field $\phi_{0,n}$ satisfies the following tails estimates
\begin{equation}
\label{TailSup}
\mathbb{P}\left( \underset{[0,1]^2}{\sup} \abs{\phi_{0,n}} \geq \alpha ( n + C\sqrt{n}) \right) \leq  C 4^n e^{- \frac{\alpha^2}{\log 4} n}
\end{equation}
as well as 
\begin{equation}
\label{TailSupSmall}
\mathbb{P}\left( \underset{[0,1]^2}{\sup} \abs{\phi_{0,n}}  \geq n \log 4 + C \sqrt{n} + C s \right) \leq C e^{- s}.
\end{equation}
\end{Lemma}

\begin{proof}
First we bound a discretization of the field $\phi_{0,n}$. Since the variance of $\phi_{0,n}(x)$ is equal to $(n+1) \log 2$, by union bound and classical Gaussian tail estimates we have $\mathbb{P} ( \max_{[0,1]^2 \cap 2^{-n} \Z^2} \abs{\phi_{0,n}(x)} \geq x ) \leq 4^n e^{- \frac{ x^2 }{(n+1)\log 4}}$ hence by  introducing $x_n := \sqrt{n+1} \sqrt{n}$ we get
\begin{equation}
\label{tail}
\mathbb{P} \left( \underset{x \in [0,1]^2 \cap 2^{-n} \Z^2}{\max} \abs{\phi_{0,n}(x)} \geq \alpha x_n \right)  \leq 4^n e^{- \frac{ \alpha^2}{\log 4} n}.
\end{equation}
Now we want to bound $\sup_{[0,1]^2} \abs{\phi_{0,n}(x)}$ for which we want an equivalent of the bound \eqref{tail}. By Fernique's theorem, we have a tail estimate for the gradient of $\phi_{0}$ i.e. there exists some $C > 0$ so that for every $x > 0$, $\mathbb{P} ( \sup_{[0,1]^2} \abs{\nabla \phi_0} \geq x ) \leq C e^{-x^2/2C}$. Then, by scaling, for any dyadic cube $P \in \mathcal{P}_k$, $\mathbb{P} ( \sup_{P} \abs{\nabla \phi_k} \geq 2^k x ) \leq C e^{-x^2/2C}$ thus, by union bound $\mathbb{P} ( \sup_{[0,1]^2} \abs{\nabla \phi_k} \geq 2^k x ) \leq C 4^k e^{-x^2/2C}$. We can now work out the gradient field $\nabla \phi_{0,n}$: $\mathbb{P} ( \sup_{[0,1]^2} \abs{\nabla \phi_{0,n}} \geq 2^{n+1} x ) \leq \mathbb{P} (\sum_{k=0}^n \sup_{[0,1]^2}\abs{\nabla \phi_{k}} \geq \sum_{k=0}^n 2^{k} x ) \leq C 4^{n} e^{-x^2 / 2 C}$ hence $\mathbb{P} ( 2^{-n} \sup_{[0,1]^2}\abs{\nabla \phi_{0,n}} \geq x ) \leq C 4^n e^{-x^2/2C}$. This inequality can be rewritten by introducing  $y_n :=  C \sqrt{n}$ as: 
\begin{equation}
\label{gradient}
\mathbb{P} \left( 2^{-n} \sup_{[0,1]^2} \abs{\nabla \phi_{0,n}} \geq \alpha y_n \right) \leq C 4^n e^{-\frac{\alpha^2}{\log 4}n}.
\end{equation}
Using the discrete bound \eqref{tail} and the gradient one \eqref{gradient}, since
$$\sup_{[0,1]^2}\abs{\phi_{0,n}} \leq \max_{[0,1]^2 \cap 2^{-n} \Z^2} \abs{\phi_{0,n}}  +  2^{-n} \sup_{[0,1]^2} \abs{\nabla \phi_{0,n}},$$ 
we get the result \eqref{TailSup} by union bound. Indeed, with $z_n := x_n + y_n$. $\mathbb{P} (  \underset{[0,1]^2}{\sup} \abs{\phi_{0,n}} \geq \alpha z_n ) \leq  \mathbb{P}( X_n \geq \alpha x_n ) + \mathbb{P} ( Y_n \geq \alpha Y_n ) \leq C 4^n e^{- \frac{\alpha^2}{\log 4} n}$.
Taking $\alpha = \log 4 \sqrt{1+\frac{s}{n \log 4}} \leq \log 4+\frac{s}{n}$ gives the second part \eqref{TailSupSmall}.
\end{proof}

The following lemma is a corollary of the previous one: using the tail estimates we control exponential moments.
\begin{Lemma}
\label{ExpoMoment}
We have the following upper bounds for the exponential moments of the field $\phi_{0,n}$: for $\gamma < 2$ and $n \geq 0$, $\mathbb{E} \left( e^{ \gamma \sup_{[0,1]^2} \abs{\phi_{0,n}}} \right) \leq C 4^{ \gamma n (1 + o(1)) }$, where $o(1)$ is of the form $O(n^{-1/2})$.
\end{Lemma}

\begin{proof}
Fix $0 < \gamma < 2$. We use the bound \eqref{TailSup} as follows. By introducing $s_n := n + C \sqrt{n}$  we have, by using the elementary bound $\E (e^{\gamma X}) \leq e^{\gamma x} + \int_{x}^{\infty} \gamma e^{\gamma t} \Pro (X \geq t) dt$ and for $\alpha$ to be specified:
$$
\mathbb{E} \left( e^{\gamma \underset{[0,1]^2}{\sup} \abs{\phi_{0,n}}} \right)  \leq  e^{\gamma \alpha s_n} +  \gamma \int_{\alpha s_n}^{\infty} e^{\gamma t} \mathbb{P} \left( \underset{[0,1]^2}{\sup} \abs{\phi_{0,n}} \geq t \right) dt.
$$
Setting $t = s_nu$, $\int_{\alpha s_n}^{\infty} e^{\gamma t} \mathbb{P} ( \sup_{[0,1]^2} \abs{\phi_{0,n}} \geq t ) dt =   s_n \int_{\alpha}^{\infty} e^{\gamma s_n u} \mathbb{P} ( \sup_{[0,1]^2} \abs{\phi_{0,n}} \geq s_n u) du$ and by using the bound \eqref{TailSup}
$$
\int_{\alpha}^{\infty} e^{\gamma s_n u} \mathbb{P} \left( \sup_{[0,1]^2} \abs{\phi_{0,n}} \geq s_n u \right) du \leq  C 4^n \int_{\alpha}^{\infty} e^{\gamma s_n u} e^{-\frac{u^2}{\log 4}n} du.
$$
By introducing $r_n := n^{-1} s_n$, by a change of variables we obtain:
$$
 \int_{\alpha}^{\infty} e^{\gamma s_n u} e^{-\frac{u^2}{\log 4}n} du  \leq   4^{ \frac{\gamma^2 r_n^2}{4} n} \int_{\alpha - \gamma r_n \frac{\log 4 }{2}}^{\infty} e^{- \frac{n}{\log 4} u^2} du.
$$
Taking $\alpha := r_n \log 4$, the integral in the right-hand side becomes
$$
\int_{\alpha - \gamma r_n \frac{\log 4 }{2}}^{\infty} e^{- \frac{n}{\log 4} u^2} du   = \int_{\left( 1 - \gamma/2 \right)r_n \log 4 }^{\infty} e^{- \frac{n}{\log 4} u^2} du  \leq \frac{4^{-n \left( 1- \frac{\gamma}{2} \right)^2 r_n^2 }}{\left( 2-\gamma \right) n r_n},
$$
by using the inequality $\int_a^{\infty} e^{-b x^2} dx \leq (2ab)^{-1} e^{-ba^2}$ valid for $a> 0$ and $b >0$. Gathering the pieces we get $\mathbb{E} ( e^{\gamma  \sup_{[0,1]^2} \abs{\phi_{0,n}}} ) \leq (1 + C \frac{\gamma}{2-\gamma} )4^{\gamma r_n^2 n}$ hence the result.
\end{proof}

We add here a Lemma which is in the same vein as the previous one.

\begin{Lemma}
\label{AddLemma}
Suppose that we have the following tail estimate on a sequence of positive random variables $(X_k)_{k \geq 0}$: for $k \geq 0$ and $s > 2$,
$$
\Pro \left( X_k \geq e^{s} \right) \leq 4^k e^{- c \frac{s^2}{\log s}}.
$$
Then, we have the following moment estimate: there exists $C >0$ depending only on $c$ such that for $k$ large,
$$
\E \left( X_k \right) \leq e^{C \sqrt{k \log k}}.
$$
\end{Lemma}
\begin{proof}
Fix $x_k > 2$ to be specified. We can rewrite $\E (X_k) - e^{x_k}$ as
$$
\int_{e^{x_k}}^{\infty} \Pro \left( X_k \geq x \right) dx =  \int_{x_k}^{\infty} \Pro \left( X_k \geq e^s \right) e^s ds \leq 4^k \int_{x_k}^{\infty} e^{-c \frac{s^2}{\log s}} e^s ds \leq 4^k e^{x_k} \int_{x_k}^{\infty} e^{-c\frac{s^2}{\log x_k}} ds.
$$
By using $\int_{a}^{\infty} e^{-bx^2} dx \leq (2ab)^{-1} e^{-b a^2}$, we get $\E(X_k) \leq e^{x_k} +  4^k e^{x_k} (2 x_k \frac{c}{\log x_k})^{-1} e^{-c \frac{x_k^2}{\log x_k}}$. Taking $x_k$ such that $k \log 4 = c \frac{x_k^2}{\log x_k}$ gives $\log k \sim 2 \log x_k$ and $x_k \sim C \sqrt{k \log k}$.
\end{proof}

\subsection{Upper bound for $F(s)$} 

\label{UpperF(s)}

In this subsection, we derive two lemmas that allow us to bound the term $F(s)$ which appears in the proof of Proposition \ref{SupTails}. The first one corresponds to $a_{t_s}$, the second one to $\int_{0}^{\infty} a_t dt$.

\begin{Lemma}
\label{F(s)First}
If $a,b,c >0$ and $\alpha \in (0,1/2)$ then the function $f_s(t) := -at + b t^{1/2 + \alpha} + c s \sqrt{t}$ in increasing on $[0,t_s]$, decreasing on $[t_s,\infty]$ for some $t_s > 0$ which satisfy $a t_s^{1/2} = \frac{1}{2} cs + O(s^{2\alpha})$. In particular, we have: $\exp(f_s(t_s)) \leq e^{ \frac{c^2 s^2}{4 a} + C s^{1+ 2 \alpha} }$.
\end{Lemma}

\begin{proof}
First, notice that $f_s'(t) = -a + (\frac{1}{2}+\alpha) b t^{-1/2+\alpha} + \frac{1}{2} cs t^{-1/2}$. Since $f_s'(t_s) = 0$ we obtain $a =  (\frac{1}{2} + \alpha) b t_s^{-1/2 + \alpha} + \frac{1}{2} c s t_s^{-1/2} $ which we write:
\begin{equation}
\label{eq:relation}
a t_s^{1/2} =  \frac{cs}{2} + (\frac{1}{2} + \alpha) b t_s^{\alpha}.
\end{equation}
Thus $a t_s^{1/2} \geq cs/2$. In particular, $\lim_{s \to \infty} t_s = + \infty$. Using \eqref{eq:relation}, we obtain $a t_s^{1/2} \sim_{s \to \infty} \frac{1}{2}cs$. Using again \eqref{eq:relation}, we have $a t_s^{1/2} = \frac{1}{2} cs + O(s^{2\alpha})$. Using again \eqref{eq:relation} we conclude by noticing that: $f_s(t_s) = -a t_s + b t_s^{1/2 + \alpha} + c s t_s^{1/2} = a t_s - 2b \alpha t_s^{1/2+\alpha}$.
\end{proof}

\begin{Lemma}
\label{BoundIntegral}
Let $\alpha, a,b >0$ with $\alpha < 1/2$. For every $s > 0$ the following inequality holds
$$
\int_0^{\infty} e^{-t + a t^{1/2+\alpha} + b s \sqrt{t}} dt  \leq C_{\alpha,a} (2 + bs) e^{\frac{(bs)^2}{4}} e^{C_{\alpha} (bs)^{1+2\alpha}},
$$
where $C_{\alpha,a} < \infty$ just depends on $a$ and $C_{\alpha}$ just depends on $\alpha$.
\end{Lemma}

\begin{proof} By writing $-t + bs \sqrt{t} = \frac{(bs)^2}{4} - (\sqrt{t} - \frac{bs}{2})^2 $ and  the change of variable $u = \sqrt{t}$,
$$
\int_0^{\infty} e^{-t + a t^{1/2+\alpha} + b s \sqrt{t}} dt = e^{\frac{(bs)^2}{4}}  \int_0^{\infty} e^{-(u - \frac{bs}{2})^2 + a u^{1+2\alpha}}  2 u du .
$$
Now, by the change of variables $v = u - bs/2$, we get 
$$
\int_0^{\infty} e^{-(u - \frac{bs}{2})^2 + a u^{1+2\alpha}}  2 u du  =   \int_{-\frac{bs}{2}}^{\infty} e^{-v^2 + a (v+ \frac{bs}{2})^{1+2\alpha}} (2v + bs) dv.
$$
Finally, by Jensen's inequality, $(v + \frac{bs}{2})^{1+2\alpha}  \leq C_{\alpha} (\abs{v}^{1+2\alpha} + (bs)^{1+2\alpha})$ thus
\begin{align*}
 \int_{-\frac{bs}{2}}^{\infty} e^{-v^2 + a (v+ \frac{bs}{2})^{1+2\alpha}} (2v + bs) dv &  \leq e^{C_{\alpha}  a (bs)^{1+2\alpha}} \int_{-\frac{bs}{2}}^{\infty} e^{-v^2 + C_{\alpha} a |v|^{1+2\alpha}}  (2v + bs)  dv \\
  & \leq  e^{C_{\alpha} a (bs)^{1+2\alpha}} (2 + bs) \int_{-\infty}^{\infty} e^{-v^2 + C_{\alpha} a |v|^{1+2\alpha}}  \left(1+|v|\right)dv.
\end{align*}
\end{proof}  

Now, we bound $F(s)$. Recall first that $F(s) \leq 2 a_{t_s} + \int_{0}^{\infty} a_{t} dt$ where $a_t = \exp (f_s(t))$, $f_s(t) := - t(1-\lambda) \log 2 + C t^{1/2+ \alpha} +  \beta s \sqrt{t}$, $\lambda := (1+a_{\eps}) \gamma$, $\alpha := \frac{\delta}{2}$ and $\beta := \frac{\gamma}{2} \sqrt{\log 4}$. By Lemma \ref{F(s)First}, $a_{t_s} \leq e^{\frac{\beta^2 s^2}{4 (1-\lambda) \log 2} + C s^{1+2\alpha}} = e^{\frac{\gamma^2 \log 4 s^2}{16 (1-(1+a_{\eps})\gamma) \log 2} + C s^{1+\delta}} = e^{\frac{\gamma^2 s^2}{8(1-(1+a_{\eps})\gamma)} + Cs^{1+ \delta}}$. By the change of variable $u = t (1-\lambda)\log 2$ and Lemma \ref{BoundIntegral}, we obtain the integral bound $\int_0^{\infty} a_t dt \leq C e^{\frac{\gamma^2 s^2}{8(1-(1+a_{\eps})\gamma)}}  e^{C s^{1+ \delta}}$. Altogether we get $F(s) \leq C e^{\frac{\gamma^2 s^2}{8(1-(1+a_{\eps})\gamma)}} e^{C s^{1+\delta}}$.

\bibliographystyle{spmpsci}

\bibliography{Bibli}

\end{document}